\documentclass[11pt,a4paper]{article}
\usepackage[margin=1in]{geometry}
\usepackage{amsmath,amsthm,amsfonts,amssymb,mathtools,bm,mathrsfs}
\usepackage[round,authoryear]{natbib}
\usepackage{booktabs,enumitem}
\usepackage[hidelinks]{hyperref}
\hypersetup{pdftitle={High-Dimensional Bootstrap and Asymptotic Expansion for the k-th Largest Coordinate},pdfauthor={Fengyi Song and Long Feng}}
\DeclareMathAlphabet{\mathbf}{OT1}{cmr}{bx}{n}
\SetMathAlphabet{\mathbf}{bold}{OT1}{cmr}{bx}{n}
\AtBeginDocument{\lefthyphenmin=2\righthyphenmin=3}
\let\cite\citet
\theoremstyle{plain}
\newtheorem{theorem}{Theorem}[section]
\newtheorem{proposition}{Proposition}[section]
\newtheorem{lemma}{Lemma}[section]
\newtheorem{corollary}{Corollary}[section]

\theoremstyle{definition}
\newtheorem{remark}{Remark}[section]
\newtheorem{assumption}{Assumption}[section]

\newcommand{\R}{\mathbb{R}}
\newcommand{\E}{\mathbb{E}}
\newcommand{\Prob}{\mathbb{P}}
\newcommand{\Var}{\mathrm{Var}}
\newcommand{\Cov}{\mathrm{Cov}}
\newcommand{\Corr}{\mathrm{Corr}}
\newcommand{\tr}{\mathrm{tr}}

\newcommand{\ind}{\mathbf 1}

\newcommand{\ph}{\phi}

\newcommand{\bPhi}{\bar\Phi}
\newcommand{\epsn}{\varepsilon_n}
\newcommand{\deln}{\delta_n}
\newcommand{\Tcal}{\mathcal T}

\newcommand{\1}{\mathbf 1}

\begin{document}

\title{High-Dimensional Bootstrap and Asymptotic Expansion for the $k$-th Largest Coordinate}
\author{Fengyi Song\textsuperscript{1}\qquad Long Feng\textsuperscript{1,2}\\[4pt]
  {\small\textsuperscript{1}School of Statistics and Data Science, Nankai University, Tianjin, China}\\
  {\small\textsuperscript{2}LEBPS, KLMDASR, AAIS and LPMC, Nankai University, Tianjin, China}\\
  {\small E-mail (Long Feng): \texttt{flnankai@nankai.edu.cn}}}
\date{}
\maketitle

\begin{abstract}
We study bootstrap inference for the $k$th largest coordinate of a normalized sum of independent high-dimensional random vectors, with $k$ fixed. Combining exceedance counts, weighted inclusion--exclusion and global Gaussian derivative bounds, we establish Edgeworth and Cornish--Fisher expansions for nonrectangular order-statistic events under Stein-kernel regularity. A joint coverage expansion accounts for the dependence between the statistic and its random bootstrap critical value. Under the stated covariance and dimension conditions, third-moment matching gives a coverage error of inverse sample-size order up to logarithmic factors. A prepivoted double wild bootstrap attains the same error scale with bounded Stein-kernel multipliers and third-moment matching at both levels. For stationary Gaussian reference sequences with exponential strong mixing, we also obtain an explicit Poisson-approximation remainder, which reaches inverse sample-size order under sufficient dimension growth. These results extend high-dimensional second-order bootstrap theory from maxima to nonmaximal order statistics.
\end{abstract}



\section{Introduction}

High-dimensional Gaussian approximation for maxima and rectangular probabilities is now a basic tool in modern high-dimensional inference. For the maximum of a sum of independent random vectors, the seminal work of \cite{CCK2013} established Gaussian approximation and Gaussian multiplier bootstrap validity when the dimension is allowed to be much larger than the sample size. This line of work was sharpened substantially by \cite{CCK2017}, who extended the approximation theory to hyperrectangles and improved the first-order rate. Later, \cite{DengZhang2020} showed that third-moment matching bootstrap procedures enjoy a better logarithmic dependence in the first-order bound, and \cite{Koike2021Notes} proved that the same logarithmic rate is already available for normal approximation. Among the general first-order results under mild moment assumptions, \cite{CCKK2022} further improved the error bound to an $n^{-1/4}$-type rate up to logarithmic factors. Under additional nondegeneracy or structural assumptions, nearly parametric $n^{-1/2}$ rates up to logarithmic losses are also available; see, for example, \citep{LopesLinMuller2020,FangKoike2021,CCKoike2023,FangKoike2024}.

\cite{Koike2026} developed high-dimensional Edgeworth and Cornish--Fisher expansions for maxima and related rectangular probabilities by combining Stein-kernel arguments, smoothing inequalities, and a careful analysis of Gaussian anti-concentration. As a consequence, \cite{Koike2026} obtained a second-order bootstrap coverage expansion and showed that, in several important regimes, the coverage error can be improved from the first-order scale to $O\!\left(\log^{a}(dn)/n\right)$ for a suitable constant $a>0$. In particular, for third-moment matching wild bootstrap, the maximum statistic becomes second-order accurate even without studentization under suitable covariance assumptions.

Compared with the theory for maxima, the literature for the $k$th largest coordinate is still sparse. Classical results on order statistics and extremes, such as \citep{FisherTippett1928,Mu1966,WattsRootzenLeadbetter1982}, do not address high-dimensional Gaussian approximation for sums of random vectors. On the Gaussian side, \cite{Kozbur2021} studied dimension-free anti-concentration inequalities for Gaussian order statistics. In the genuinely high-dimensional setting, \cite{DingLiShiSunZhang2026} established Gaussian and Gaussian multiplier bootstrap approximations for the $k$th largest coordinate and for more general functionals of the top-$k$ order statistics. For the $k$th largest coordinate, their Kolmogorov bounds are of order
\[
k^2\Bigl(\frac{B_n^2\log^5(dn)}{n}\Bigr)^{1/4},
\]
up to universal constants, and the bounds for general top-$k$ functionals are of even larger order. Here $d$ is the dimension and $B_n$ is their moment envelope; our comparison is for fixed $k$ under the additional Stein-kernel assumptions used below. Therefore the currently available theory for the $k$th largest coordinate is still essentially first-order and does not provide a second-order coverage expansion comparable to the one available for maxima.

The purpose of the present paper is to fill this gap. We prove that the $k$th largest coordinate of a high-dimensional normalized sum also admits a Koike-type second-order bootstrap expansion. Our argument starts from the exact exceedance-count representation of the event $\{T_{n,[k]}\le t\}$ and combines weighted inclusion--exclusion with global Gaussian derivative bounds and a joint coverage argument. This allows us to transfer the second-order expansion machinery from maxima to the $k$th order statistic. As a result, we show that third-moment matching wild bootstrap retains second-order accuracy for the $k$th largest coordinate under the stated covariance and dimension conditions, and we also obtain a second-order result for the prepivoted double wild bootstrap with third-moment matching at both levels.

We also give a complementary dependence formulation based on a stationary Gaussian reference field with exponentially decaying strong-mixing coefficients. This assumption is structurally different from the maximal-correlation condition used in the Gaussian factorial-moment analysis: it exploits one-dimensional dependence and allows local clusters of highly correlated coordinates. In that setting we obtain the same distributional, quantile, and coverage expansions, together with a Gaussian Poisson approximation whose explicit remainder $r_d$ isolates the effect of local exceedance clustering. This remainder can be of order $n^{-1}$ when the dimension grows sufficiently quickly relative to the sample size.

The remainder of the paper is organized as follows. Section~\ref{sec:main} presents the main theoretical results, including the exponential-mixing alternative in Section~\ref{subsec:mixing-alt}. Section~\ref{sec:sim} reports simulation results comparing several bootstrap methods. Section~\ref{sec:con} concludes. All proofs are collected in Appendices~\ref{supp-sec:proofs} and~\ref{supp-sec:mixing-alt}.

\medskip
\noindent\textbf{Notation.}
We write $[d]:=\{1,\dots,d\}$. For a vector $\bm x\in\R^m$, let
\[
\|\bm x\|_2:=\Bigl(\sum_{j=1}^m x_j^2\Bigr)^{1/2},
\qquad
\|\bm x\|_\infty:=\max_{1\le j\le m}|x_j|.
\]
We denote by $\mathbf{1}_d=(1,\dots,1)^\top\in\R^d$ the all-ones vector. For $r\in\mathbb N$, $(\R^m)^{\otimes r}$ denotes the set of real-valued $m$-dimensional $r$-tensors. If $\mathsf T\in(\R^m)^{\otimes q}$ and $\mathsf U\in(\R^m)^{\otimes r}$, then $\mathsf T\otimes\mathsf U\in(\R^m)^{\otimes(q+r)}$ denotes their tensor product. When $q=r$, we write
\[
\langle \mathsf T,\mathsf U\rangle
:=
\sum_{j_1,\dots,j_r=1}^m
T_{j_1,\dots,j_r}U_{j_1,\dots,j_r},
\]
and
\[
\|\mathsf T\|_1:=\sum_{j_1,\dots,j_r=1}^m |T_{j_1,\dots,j_r}|,
\qquad
\|\mathsf T\|_\infty:=\max_{1\le j_1,\dots,j_r\le m}|T_{j_1,\dots,j_r}|.
\]
For $\bm x\in\R^m$, $\bm x^{\otimes r}$ denotes the $r$th tensor power of $\bm x$. Whenever $\bm X_1,\dots,\bm X_n$ are under discussion, we set
\[
\bar{\bm X}^{\,r}:=\frac1n\sum_{i=1}^n \bm X_i^{\otimes r}.
\]
Given an $r$-times differentiable function $h: \mathbb{R}^d \rightarrow \mathbb{R}$, we set 
\[
\nabla^r h(x):= \left(\partial^{j_1, \ldots, j_r} h(x)\right)_{1 \leq j_1, \ldots, j_r \leq d} \in\left(\mathbb{R}^d\right)^{\otimes r}
\]
 for $x \in \mathbb{R}^d$, where $\partial^{j_1, \ldots, j_r}=\frac{\partial^r}{\partial x_{j_1} \cdots \partial x_{j_r}}$. For $m \in \mathbb{N} \cup\{\infty\}, C_b^m\left(\mathbb{R}^d\right)$ denotes the set of bounded $C^m$ functions with bounded derivatives. For a multi-index $\alpha=(\alpha_1,\dots,\alpha_m)\in\mathbb{N}_0^m$, we write
\[
|\alpha|:=\sum_{j=1}^m \alpha_j,
\qquad
\partial^\alpha:=\partial_1^{\alpha_1}\cdots \partial_m^{\alpha_m}.
\]
For a positive definite matrix $V$, let $\ph_V$ denote the density of $N(\bm 0,V)$. We write $\Phi$ for the standard normal distribution function and $\bPhi:=1-\Phi$ for its survival function. For a distribution function $F:\R\to[0,1]$, its generalized inverse is defined by
\[
F^{-1}(p):=\inf\{t\in\R:F(t)\ge p\},
\qquad p\in(0,1).
\]
For $\alpha>0$ and a scalar random variable $Y$, let
\[
\|Y\|_{\psi_\alpha}
:=
\inf\Bigl\{C>0:\ \E\exp(|Y|^\alpha/C^\alpha)\le 2\Bigr\}.
\]
For a matrix $\mathbf{A}=(a_{j\ell})$, we set
\[
\|\mathbf{A}\|_{\max}:=\max_{1\le j,\ell\le m}|a_{j\ell}|,
\qquad
R_j(\mathbf{A}):=\sum_{\ell\ne j}|a_{j\ell}|.
\]
Also, $\Prob^*$ and $\E^*$ denote conditional probability and expectation given the data. 
We assume $d\ge 3$ whenever an expression containing $\log d$ appears, and similarly for $n$.

\section{Main Results}\label{sec:main}

\subsection{Asymptotic expansion of coverage probability}

Let $\bm X_1,\dots,\bm X_n$ be independent centered random vectors in $\R^d$, and define
\begin{equation*}
\bm S_n:=\frac{1}{\sqrt n}\sum_{i=1}^n \bm X_i,
\qquad
\bm Z\sim N(\bm 0,\mathbf{\Sigma}),
\qquad
\mathbf{\Sigma}:=\Var(\bm S_n).
\end{equation*}
Write
\begin{equation*}
T_{n,[1]}\ge T_{n,[2]}\ge \cdots \ge T_{n,[d]}
\end{equation*}
for the descending order statistics of the coordinates of $\bm S_n$, and define $T_{\bm Z,[k]}$ analogously from $\bm Z$. Set
\begin{equation*}
G_k(t):=\Prob(T_{\bm Z,[k]}\le t),
\qquad
f_k(t):=G_k'(t)
\end{equation*}
whenever the derivative exists.

Let $w_1,\dots,w_n$ be i.i.d. multipliers independent of the data. Put
\begin{equation*}
\bar{\bm X}:=\frac1n\sum_{i=1}^n \bm X_i,
\qquad
\bm S_n^*:=\frac1{\sqrt n}\sum_{i=1}^n w_i(\bm X_i-\bar{\bm X}).
\end{equation*}
Let $T_{n,[k]}^*$ denote the $k$th largest coordinate of $\bm S_n^*$, and write
\begin{equation*}
\hat F_{n,k}(t):=\Prob^*(T_{n,[k]}^*\le t),
\qquad
\hat c_{p,k}:=\inf\{t\in\R:\hat F_{n,k}(t)\ge p\}.
\end{equation*}
For each coordinate,
$\sigma_j^2:=\Sigma_{jj},
\underline\sigma:=\min_{1\le j\le d}\sigma_j,
\overline\sigma:=\max_{1\le j\le d}\sigma_j.$
For $p\in(0,1)$, define the Gaussian quantile
$ c^G_{p,k}:=G_k^{-1}(p).$
Fix $\epsilon\in(0,1/2)$ and define the quantile window
\begin{equation*}
\Tcal_{k,\epsilon}:=\{c^G_{p,k}:\ p\in[\epsilon/2,1-\epsilon/2]\}.
\end{equation*}
For $t\in\R$, define the exceedance counts
\begin{equation*}
\begin{gathered}
N_n(t):=\sum_{j=1}^d \ind\{S_{n,j}>t\},
\qquad
N_n^*(t):=\sum_{j=1}^d \ind\{S_{n,j}^*>t\},
\\
N_Z(t):=\sum_{j=1}^d \ind\{Z_j>t\}.
\end{gathered}
\end{equation*}
Then
\begin{equation*}
T_{n,[k]}\le t \iff N_n(t)\le k-1.
\end{equation*}
For every integer $s\ge1$, define
\begin{equation*}
V_{n,s}(t):=\E\binom{N_n(t)}{s},
\qquad
V^*_{n,s}(t):=\E^*\binom{N_n^*(t)}{s},
\qquad
V_{Z,s}(t):=\E\binom{N_Z(t)}{s}.
\end{equation*}
For each nonempty $I\subset[d]$, write
\begin{equation*}
B_I(t):=\{\bm z\in\R^d:\ z_j>t, \forall j\in I\}
\end{equation*}
and
\begin{equation*}
\pi_I(t):=\Prob(Z_j>t, \forall j\in I).
\end{equation*}
Let $\ph_{\mathbf{\Sigma}}$ denote the density of $N(\bm0,\mathbf{\Sigma})$, and abbreviate $\ph:=\ph_{\mathbf{\Sigma}}$. The first-order Edgeworth density for $\bm S_n$ is
\begin{equation*}
 p_n(\bm z)
 :=
 \ph(\bm z)
 -\frac{1}{6\sqrt n}\bigl\langle \E[\bar{\bm X}^{\,3}],\nabla^3\ph(\bm z)\bigr\rangle.
\end{equation*}
Let $\gamma:=\E(w_1^3).$
The bootstrap Edgeworth density is defined by
\begin{equation*}
\hat p_{n,\gamma}(\bm z)
:=
\ph(\bm z)
+\frac12\bigl\langle \bar{\bm X}^{\,2}-\mathbf{\Sigma},\nabla^2\ph(\bm z)\bigr\rangle
-\frac{\gamma}{6\sqrt n}\bigl\langle \bar{\bm X}^{\,3},\nabla^3\ph(\bm z)\bigr\rangle.
\end{equation*}
For each integer $s\ge1$, define
\begin{equation*}
M_{n,s}(t)
:=
\sum_{\substack{I\subset[d]\\ |I|=s}}
\int_{(t,\infty)^s} p_{n,I}(\bm u)\,d\bm u,
\qquad
\hat M_{n,s,\gamma}(t)
:=
\sum_{\substack{I\subset[d]\\ |I|=s}}
\int_{(t,\infty)^s}\hat p_{n,\gamma,I}(\bm u)\,d\bm u,
\end{equation*}
where, writing $\bar{\bm X}_I^{\,r}=n^{-1}\sum_i\bm X_{i,I}^{\otimes r}$ and
$\ph_I=\ph_{\mathbf\Sigma_{II}}$, the projected densities are
\begin{align*}
 p_{n,I}(\bm u)
 &=\ph_I(\bm u)-\frac1{6\sqrt n}
       \langle\E\bar{\bm X}_I^{\,3},\nabla^3\ph_I(\bm u)\rangle,\\
 \hat p_{n,\gamma,I}(\bm u)
 &=\ph_I(\bm u)+\frac12
       \langle\bar{\bm X}_I^{\,2}-\mathbf\Sigma_{II},\nabla^2\ph_I(\bm u)\rangle
       -\frac\gamma{6\sqrt n}
       \langle\bar{\bm X}_I^{\,3},\nabla^3\ph_I(\bm u)\rangle.
\end{align*}
Also set
\begin{equation*}
M_{Z,s}(t):=\sum_{\substack{I\subset[d]\\ |I|=s}}\pi_I(t)=V_{Z,s}(t).
\end{equation*}

We now present the assumptions underlying our analysis.

\begin{assumption}
\label{ass:data}
The vectors $\bm X_1,\dots,\bm X_n$ are independent and centered. Each $\bm X_i$ admits a Stein kernel $\bm\tau_i(\bm X_i)$ in the sense that
\begin{equation*}
\E\bigl[\bm X_i^\top f(\bm X_i)\bigr]
=
\E\bigl[\tr\{\bm\tau_i(\bm X_i)\nabla f(\bm X_i)^\top\}\bigr]
\end{equation*}
for every smooth vector-valued test function for which both sides are finite. There exist constants $b>0$ and $\sigma_*>0$ such that
\begin{enumerate}[label=(\roman*)]
\item $\lambda_{\min}(\mathbf{\Sigma})\ge \sigma_*^2$;
\item
$\max_{1\le i\le n}\max_{1\le j\le d}\|X_{ij}\|_{\psi_1}\le b,
\max_{1\le i\le n}\max_{1\le j,\ell\le d}
\|\tau_{i,j\ell}(\bm X_i)-\E\tau_{i,j\ell}(\bm X_i)\|_{\psi_1}
\le b^2;$
\item
$\deln:=\frac{b^5}{\sigma_*^5}\frac{\log^3(dn)}{n},
\epsn:=\sqrt{\deln\log n}$
 satisfies $\epsn\to0$.
\end{enumerate}
\end{assumption}

\begin{remark}
    {Assumption~\ref{ass:data} is the data-side regularity condition.}
The Stein identity provides the analytic device behind the projected Edgeworth expansion and is a convenient substitute for classical Cram\'er-type smoothness conditions in high dimension. As emphasized in \cite[Remark~2.4(a)--(b)]{Koike2026}, for a nondegenerate univariate law the existence of a Stein kernel implies a nontrivial absolutely continuous component, and hence Cram\'er's condition, whereas in higher dimensions Stein kernels remain available even in situations where a multivariate Cram\'er condition is not appropriate, such as Gaussian laws with singular covariance matrices. In our setting, part~(i) requires a uniform lower bound on $\lambda_{\min}(\mathbf{\Sigma})$, which prevents global degeneracy of the Gaussian comparison law and guarantees that the projected Gaussian densities and their derivatives remain well behaved. Part~(ii) imposes sub-exponential control on both the coordinates $X_{ij}$ and the fluctuations of the Stein-kernel entries. Since
\[
\E\{\bm\tau_i(\bm X_i)\}=\E(\bm X_i\bm X_i^\top),
\]
the centered quantity
\[
\tau_{i,j\ell}(\bm X_i)-\E\tau_{i,j\ell}(\bm X_i)
\]
measures the random fluctuation of the local covariance proxy around its population counterpart. The centered bound implies the full-kernel moment bound used in Koike's decomposition. Its use for order-statistic events additionally requires the Gaussian derivative and signed-boundary estimates established in Appendix~\ref{supp-sec:proofs}. Finally, part~(iii) is the high-dimensional scaling condition ensuring that the marginal Edgeworth remainder vanishes. The coverage results below impose the additional growth condition needed for their joint remainder.
\end{remark}

\begin{assumption}
\label{ass:mult}
The multipliers $w_1,w_2,\dots$ are i.i.d., independent of the data, satisfy
\begin{equation*}
\E w_1=0,
\qquad
\E w_1^2=1,
\end{equation*}
and, in addition, satisfy one of the following two conditions:
\begin{enumerate}[label=(\roman*)]
\item $w_1\sim N(0,1)$;
\item $w_1$ admits a Stein kernel $\tau^w(w_1)$ and there exists a constant $b_w\ge1$ such that
\begin{equation*}
|w_1|\le b_w,
\qquad
|\tau^w(w_1)|\le b_w^2
\qquad\text{a.s.}
\end{equation*}
\end{enumerate}
All absolute moments are finite under either alternative.
For the dependence of constants, set $b_w=1$ in case~(i).
The constants in the sequel are allowed to depend on $b_w$.
\end{assumption}

\begin{remark}
{Assumption~\ref{ass:mult} is the bootstrap analogue of Assumption~\ref{ass:data}.}
It ensures that, conditional on the data, the multiplier statistic admits the same kind of Stein--Edgeworth expansion as the original statistic. For Gaussian multipliers, the conditional bootstrap sum is Gaussian. The alternative bounded Stein-kernel condition covers smooth non-Gaussian multipliers and is particularly useful for moment matching, which is central to the second-order improvement. As discussed in \cite{Koike2026}, this framework does not cover two-point multipliers such as Mammen's weights, since two-point laws do not admit Stein kernels.
\end{remark}

\begin{assumption}
\label{ass:variance}
There exist constants $0<\underline\sigma\le \overline\sigma<\infty$ such that
\begin{equation*}
\underline\sigma^2\le \Sigma_{jj}\le \overline\sigma^2,
\qquad j=1,\dots,d.
\end{equation*}
\end{assumption}

\begin{remark}
Assumption~\ref{ass:variance} is implied by Assumption~\ref{ass:data}.
The coordinate $\psi_1$ bound gives $\Prob(|X_{ij}|>t)\le2e^{-t/b}$.
Independence and centering therefore yield
\begin{align*}
\E X_{ij}^2
&=2\int_0^\infty t\Prob(|X_{ij}|>t)\,dt\le4b^2,\\
\sigma_*^2
&\le\Sigma_{jj}=\frac1n\sum_{i=1}^n\E X_{ij}^2\le4b^2.
\end{align*}
The notation $\underline\sigma,\overline\sigma$ is retained for the Gaussian
comparisons. These bounds permit unequal marginal variances.
\end{remark}

\begin{assumption}
\label{ass:weakcorr}
Let
$\rho_d:=\max_{1\le i\ne j\le d}|\Sigma_{ij}|/(\sigma_i\sigma_j).$
We assume $\rho_d\log d\to0.$
\end{assumption}

\begin{remark}
{Assumption~\ref{ass:weakcorr} controls weak dependence in the Gaussian factorial-moment calculations.}
The finite Gaussian comparison in Appendix~\ref{supp-sec:proofs} retains its dependence on the projection order. For projections of order at most $k_0+1$ on the central threshold window, its relative error is controlled by
\[
 (k_0+1)^2\rho_d\log d,
\]
provided this quantity is sufficiently small. Thus a growing-order relative comparison uses this stronger scale, rather than $\rho_d\log d$ alone. The coverage expansion below uses global Gaussian derivative bounds and a joint expansion, and does not require Assumption~\ref{ass:weakcorr}.
\end{remark}

Fix a constant $A>0$ and define
$k_0:=\lceil A\log(\epsn^{-1})\rceil$ for the Gaussian block truncation below.
Throughout, $1\le k\le d$ is fixed. Write
\[
 L_n=\log(end),\qquad L_d=\log(ed),\qquad r_n=L_n^{12}/n,
\]
and define
\[
 A_k(\bm a)=\left\{\bm z\in\R^d:
       \sum_{j=1}^d\ind\{z_j>a_j\}\le k-1\right\}.
\]
For $k>1$, $A_k(\bm a)$ is generally not a hyperrectangle.
Appendix~\ref{supp-sec:proofs} proves Gaussian derivative and signed-boundary
bounds for these sets before adapting Koike's interpolation argument.
The rectangle approximation theorems of \cite{Koike2026} are used only
for the projected orthant expansions.
Finally define
\begin{equation*}
Q_{n,k}(t)
:=-\sum_{s=k}^{d}(-1)^{s-k}\binom{s-1}{k-1}
       \{M_{n,s}(t)-M_{Z,s}(t)\},
\end{equation*}
and define $\hat Q_{n,\gamma,k}(t)$ analogously with $\hat M_{n,s,\gamma}(t)$ in place of $M_{n,s}(t)$.
The finite inclusion--exclusion identity gives the equivalent formulas
\begin{align*}
 Q_{n,k}(t)&=\int_{A_k(t\mathbf1_d)}\{p_n(\bm z)-\ph(\bm z)\}\,d\bm z,\\
 \hat Q_{n,\gamma,k}(t)
 &=\int_{A_k(t\mathbf1_d)}\{\hat p_{n,\gamma}(\bm z)-\ph(\bm z)\}\,d\bm z.
\end{align*}
For later use, put
\[
 \mathsf J_{r,k}(t)=\int_{A_k(t\mathbf1_d)}\nabla^r\ph(\bm z)\,d\bm z,
 \qquad \mathsf M_3=\E\bar{\bm X}^{\,3},
\]
\[
\begin{gathered}
 B_{n,k}(t)=\tfrac12\langle\bar{\bm X}^{\,2}-\mathbf\Sigma,
                             \mathsf J_{2,k}(t)\rangle,\\
 c_\alpha=c^G_{1-\alpha,k},\quad f_\alpha=f_k(c_\alpha),\quad
 \mathbf\Psi_\alpha=\mathsf J_{2,k}(c_\alpha).
\end{gathered}
\]
The quadratic term used in quantile inversion is
\[
 C_{n,k}(\alpha)
 :=\frac{f_k'(c_\alpha)}{2f_\alpha^3}
       \hat Q_{n,\gamma,k}(c_\alpha)^2
   -\frac{\hat Q_{n,\gamma,k}'(c_\alpha)}{f_\alpha^2}
       \hat Q_{n,\gamma,k}(c_\alpha).
\]

\begin{theorem}
\label{thm:coverage-main}
Assume Assumptions~\ref{ass:data}--\ref{ass:variance}, with fixed $k$ and $\epsilon\in(0,1/4)$, and suppose
$L_n^{12}\le cn$ for a sufficiently small constant $c>0$. Then
\begin{equation*}
\begin{aligned}
&\sup_{\epsilon<\alpha<1-\epsilon}
\left|
\Prob\bigl(T_{n,[k]}\ge\hat c_{1-\alpha,k}\bigr)
-\left[
\alpha-(1-\gamma)Q_{n,k}(c^G_{1-\alpha,k})
                  -\E\{R_{n,k}(\alpha)\}
\right]
\right|
\\
&\qquad\le Cr_n,
\end{aligned}
\end{equation*}
where
\[
 R_{n,k}(\alpha)
 :=\frac{\left\langle\bar{\bm X}^{\,3}\otimes\mathbf1_d,
                      \mathbf\Psi_{\alpha}^{\otimes2}\right\rangle}
              {2\sqrt n\,f_k(c^G_{1-\alpha,k})}.
\]
The same bound holds with a strict rejection inequality. Moreover,
\[
 \sup_{\epsilon<\alpha<1-\epsilon}\E|R_{n,k}(\alpha)|
 \le C L_d^2/\sqrt n.
\]
The constants $c,C$ depend only on $k,b,\sigma_*,b_w,\epsilon$.
\end{theorem}

Theorem~\ref{thm:coverage-main} is the main second-order coverage statement for the single wild bootstrap. It shows that the leading coverage distortion is described by the deterministic linear term $(1-\gamma)Q_{n,k}$ together with the joint dependence correction $\E\{R_{n,k}(\alpha)\}$, while the remaining error is of order $r_n$. The bound tends to zero when $L_n^{12}/n\to0$ and is $o(n^{-1/2})$ when $L_n^{12}=o(\sqrt n)$.

\begin{corollary}[Third-moment matching]
\label{cor:thirdmatch}
Assume Assumptions~\ref{ass:data}--\ref{ass:variance}, with fixed $k$, $\epsilon\in(0,1/4)$ and $L_n^{12}\le cn$. Suppose
$\gamma=1$, $\Sigma_{jj}=\sigma^2$ for every $j$, and
$\lambda_{\max}(\mathbf\Sigma/\sigma^2)\le K$, where
$\sigma>0$ and $K<\infty$ are fixed. Then
\[
 \sup_{\epsilon<\alpha<1-\epsilon}
 \left|\Prob\bigl(T_{n,[k]}\ge\hat c_{1-\alpha,k}\bigr)-\alpha\right|
 \le C\left\{r_n+\frac{L_d^2}{\sqrt{dn}}\right\}.
\]
In particular, if $d\ge n$, the right-hand side is at most $Cr_n$.
The constant may also depend on $\sigma,K$.
\end{corollary}

Corollary~\ref{cor:thirdmatch} shows that matching the third multiplier moment removes the linear coverage distortion identified in Theorem~\ref{thm:coverage-main}. Under the covariance and dimension conditions in Corollary~\ref{cor:thirdmatch}, the wild bootstrap then becomes second-order accurate on the $r_n$ scale without any further correction.

\begin{corollary}[Persistence of the first-order term]
\label{cor:firstorderpersist}
Assume Assumptions~\ref{ass:data}--\ref{ass:variance}, with fixed $k$, $\epsilon\in(0,1/4)$ and $L_n^{12}\le cn$. Suppose
$\Sigma_{jj}=\sigma^2$ for every $j$ and
$\lambda_{\max}(\mathbf\Sigma/\sigma^2)\le K$, where
$\sigma>0$ and $K<\infty$ are fixed. Put
$\zeta_{n,d}:=L_d^{3/2}/\sqrt n$. Then
\[
\sup_{\epsilon<\alpha<1-\epsilon}
\left|
\Prob\bigl(T_{n,[k]}\ge\hat c_{1-\alpha,k}\bigr)-\alpha
 +(1-\gamma)Q_{n,k}(c^G_{1-\alpha,k})
\right|
\le C\left\{r_n+\frac{L_d^2}{\sqrt{dn}}\right\}.
\]
In particular, if for some $\alpha_0\in(\epsilon,1-\epsilon)$ and
$c_0>0$ one has
\[
 |Q_{n,k}(c^G_{1-\alpha_0,k})|\ge c_0\zeta_{n,d},
\]
then
\[
\left|
\Prob\bigl(T_{n,[k]}\ge\hat c_{1-\alpha_0,k}\bigr)-\alpha_0
\right|
\ge |1-\gamma|c_0\zeta_{n,d}
      -C\left\{r_n+\frac{L_d^2}{\sqrt{dn}}\right\}.
\]
The constant may also depend on $\sigma,K$; no third-moment matching
is assumed.
\end{corollary}

Corollary~\ref{cor:firstorderpersist} shows that the term $(1-\gamma)Q_{n,k}$ is not an artifact of the proof. For a fixed multiplier law with $\gamma\ne1$, the displayed nonvanishing condition and $r_n+L_d^2/\sqrt{dn}=o(\zeta_{n,d})$ imply a coverage error of order at least $\zeta_{n,d}$.

\subsection{Double wild bootstrap}\label{subsec:db}
Let $v_1,\dots,v_n$ be i.i.d. multipliers, independent of everything else, satisfying
\begin{equation*}
\E v_1=0,
\qquad
\E v_1^2=1,
\qquad
\E v_1^3=1,
\end{equation*}
and the bounded Stein-kernel condition in Assumption~\ref{ass:mult}(ii). At the first level, impose the same bounded condition and $\E w_1^3=1$. Define
\begin{equation*}
\bm X_i^*:=w_i(\bm X_i-\bar{\bm X}),
\qquad
\bar{\bm X}^*:=\frac1n\sum_{i=1}^n\bm X_i^*,
\qquad
\bm S_n^{**}:=\frac1{\sqrt n}\sum_{i=1}^n v_i(\bm X_i^*-\bar{\bm X}^*).
\end{equation*}
Let $T^{**}_{n,[k]}$ be the $k$th largest coordinate of $\bm S_n^{**}$, let
\begin{equation*}
\hat F^*_{n,k}(t):=\Prob^{**}(T^{**}_{n,[k]}\le t),
\end{equation*}
where $\Prob^{**}$ conditions on the data and the first-level multipliers. Define
\begin{equation*}
\begin{gathered}
 \hat H_k(u):=\Prob^*\bigl(\hat F^*_{n,k}(T_{n,[k]}^*)\le u\bigr),
 \\
 \hat\beta_{\alpha,k}:=\inf\{u\in[0,1]:\hat H_k(u)\ge1-\alpha\}.
\end{gathered}
\end{equation*}
At $p=0,1$ the outer quantile $\hat c_{p,k}$ uses the same infimum convention, with $\inf\varnothing=+\infty$ and $\hat c_{0,k}=-\infty$.
The prepivoted double-bootstrap test rejects when
\begin{equation*}
T_{n,[k]}\ge \hat c_{\hat\beta_{\alpha,k},k}.
\end{equation*}

\begin{theorem}
\label{thm:doublewild}
Assume Assumptions~\ref{ass:data}--\ref{ass:variance}, with fixed $k$ and $\epsilon\in(0,1/4)$. At both levels the multipliers satisfy Assumption~\ref{ass:mult}(ii), with fixed bounds $b_w,b_v$, and
\[
 \E w_1=\E v_1=0,\qquad \E w_1^2=\E v_1^2=1,
 \qquad \E w_1^3=\E v_1^3=1.
\]
If $L_n^{12}\le cn$ for sufficiently small $c>0$, then
\begin{equation*}
\sup_{2\epsilon<\alpha<1-2\epsilon}
\left|
\Prob\bigl(T_{n,[k]}\ge \hat c_{\hat\beta_{\alpha,k},k}\bigr)-\alpha
\right|
\le Cr_n.
\end{equation*}
The same bound holds for strict rejection. Constants depend only on
$k,b,\sigma_*,b_w,b_v,\epsilon$.
\end{theorem}

Theorem~\ref{thm:doublewild} shows that, with third-moment matching at both levels, prepivoting removes the leading dependence correction and restores second-order accuracy. Thus the double wild bootstrap achieves the same $r_n$ coverage scale as the third-moment matching single bootstrap under the conditions of Corollary~\ref{cor:thirdmatch}.

\subsection{A stationary exponential-mixing alternative}\label{subsec:mixing-alt}

A one-dimensional dependence condition can be used when the Gaussian reference field is generated by a stationary Gaussian sequence. An explicit Gaussian Poisson-approximation remainder records the contribution of local clusters of exceedances.

\begin{assumption}[Stationary Gaussian coordinates with exponential strong mixing]
\label{ass:mixing-alt}
The Gaussian reference vector $\bm Z=(Z_1,\dots,Z_d)^\top$ is the first $d$ coordinates of a centered stationary Gaussian sequence $\{Z_j\}_{j\in\mathbb Z}$ with covariance function
\[
\Gamma(h):=\Cov(Z_0,Z_h),
\qquad
\Gamma(0)=\sigma^2\in[\underline\sigma^2,\overline\sigma^2].
\]
Its strong-mixing coefficients satisfy
\[
\begin{split}
\alpha_Z(\ell)
&:=\sup\Bigl\{\bigl|\Prob(A\cap B)-\Prob(A)\Prob(B)\bigr|:\\
&\hspace{12mm} A\in\sigma(Z_j:j\le0),\ B\in\sigma(Z_j:j\ge\ell)\Bigr\}\\
&\le C_\alpha e^{-a_\alpha\ell},\qquad\ell\ge1,
\end{split}
\]
for some constants $C_\alpha\ge1$ and $a_\alpha>0$.
\end{assumption}

Write
\[
\begin{gathered}
\rho(h):=\Corr(Z_0,Z_h)=\Gamma(h)/\sigma^2,\qquad h\in\mathbb Z,\\
p(t):=\Prob(Z_1>t)=\bar\Phi(t/\sigma),\qquad\lambda(t):=d\,p(t).
\end{gathered}
\]
Because every $2\times2$ principal submatrix of $\Sigma$ has diagonal entries at most $\overline\sigma^2$ and smallest eigenvalue at least $\sigma_*^2$, we have
\begin{equation}
\sup_{1\le h<d}|\rho(h)|
\le
1-\frac{\sigma_*^2}{\overline\sigma^2}
=:\vartheta_*
<1.
\label{eq:mix-rem-vartheta}
\end{equation}
Set
\begin{equation*}
\beta_*:=\frac{1-\vartheta_*}{1+\vartheta_*}>0.
\end{equation*}
Let $\Lambda_{k,\epsilon}>0$ be the unique constant satisfying
\begin{equation*}
h_k(\Lambda_{k,\epsilon})=\epsilon/8,
\qquad
h_k(\lambda):=e^{-\lambda}\sum_{m=0}^{k-1}\frac{\lambda^m}{m!}.
\end{equation*}

Fix
\begin{equation}
\begin{gathered}
m_d:=\left\lceil d^{\beta_*/4}\right\rceil,
\qquad
\ell_d:=\left\lceil \frac{8(k_0+2)\log(2d)+8\log n}{a_\alpha}\right\rceil,
\\
q_d:=\left\lfloor \frac{d}{m_d+\ell_d}\right\rfloor.
\end{gathered}
\label{eq:mix-rem-block-lengths}
\end{equation}
For integers $q\ge1$ and $K\ge k$, define
\[
 D_{q,K,k}:=\sum_{s=k}^{K}\binom{s-1}{k-1}(s-1)\binom qs
       +\binom K{k-1}K\binom q{K+1}.
\]
Define
\begin{equation*}
\eta_{1,d}
:=
\frac{\ell_d}{m_d}
+
\frac{m_d+\ell_d}{d}
+
d^{-3\beta_*/4}(\log d)^{\beta_*/2},
\end{equation*}
and, when $q_d\ge1$ and $k_0\ge k$, define
\begin{equation*}
r_d
:=
\eta_{1,d}
+
q_d^{-1}
+
D_{q_d,k_0,k}\alpha_Z(\ell_d)
+
\frac{(3\Lambda_{k,\epsilon})^{k_0+1}}{(k_0+1)!}.
\end{equation*}
Since \eqref{eq:mix-rem-block-lengths} implies
\begin{equation*}
\alpha_Z(\ell_d)
\le
C_\alpha n^{-8}(2d)^{-8(k_0+2)},
\end{equation*}
the sequence $r_d$ tends to $0$ when $\ell_d=o(m_d)$; the block approximation is used only when $q_d\ge1$ and $k_0\ge k$.

\begin{theorem}[Stationary exponential-mixing alternative]
\label{prop:mixing-alt-main-rd}
Assume Assumptions~\ref{ass:data}, \ref{ass:mult}, \ref{ass:variance}, and
\ref{ass:mixing-alt}, with fixed $k$, $\epsilon\in(0,1/4)$ and
$L_n^{12}\le cn$ for sufficiently small $c>0$. Then
\begin{align*}
\sup_{t\in\Tcal_{k,\epsilon}}
\left|
\Prob(T_{n,[k]}\le t)-\bigl(G_k(t)+Q_{n,k}(t)\bigr)
\right|
&\le C\epsn^2,
\\
\sup_{t\in\Tcal_{k,\epsilon}}
\left|
\Prob^*(T_{n,[k]}^*\le t)-\bigl(G_k(t)+\hat Q_{n,\gamma,k}(t)\bigr)
\right|
&\le C\epsn^2
\end{align*}
with probability at least $1-C/n$ for the bootstrap assertion,
\begin{equation*}
\sup_{\epsilon<\alpha<1-\epsilon}
\left|
\hat c_{1-\alpha,k}
-\left[
 c^G_{1-\alpha,k}
 -\frac{\hat Q_{n,\gamma,k}(c^G_{1-\alpha,k})}{f_k(c^G_{1-\alpha,k})}
 -C_{n,k}(\alpha)
 \right]
\right|
\le Cr_n,
\end{equation*}
with probability at least $1-C/n$, and
\begin{equation*}
\begin{aligned}
&\sup_{\epsilon<\alpha<1-\epsilon}
\left|
\Prob\bigl(T_{n,[k]}\ge\hat c_{1-\alpha,k}\bigr)
-\left[
\alpha-(1-\gamma)Q_{n,k}(c^G_{1-\alpha,k})-\E\{R_{n,k}(\alpha)\}
\right]
\right|
\\
&\qquad\le Cr_n.
\end{aligned}
\end{equation*}
If $\gamma=1$ and $d\ge n$, then
\begin{equation*}
\sup_{\epsilon<\alpha<1-\epsilon}
\left|
\Prob\bigl(T_{n,[k]}\ge\hat c_{1-\alpha,k}\bigr)-\alpha
\right|
\le Cr_n.
\end{equation*}
If both multiplier levels satisfy the bounded alternative of
Assumption~\ref{ass:mult} and $\E w_1^3=\E v_1^3=1$, then
\begin{equation*}
\sup_{2\epsilon<\alpha<1-2\epsilon}
\left|
\Prob\bigl(T_{n,[k]}\ge\hat c_{\hat\beta_{\alpha,k},k}\bigr)-\alpha
\right|
\le Cr_n.
\end{equation*}
The coverage bounds also hold for strict rejection. If in addition
$q_d\ge1$, $k_0\ge k$ and $r_d\le c_{k,\epsilon}$, then
\[
 \sup_{t\in\Tcal_{k,\epsilon}}
       |G_k(t)-h_k\{\lambda(t)\}|\le Cr_d.
\]
Constants depend only on $k,b,\sigma_*,b_w,\epsilon,\sigma,C_\alpha,a_\alpha$
and, for double bootstrap, the second-level multiplier bound.
\end{theorem}

Theorem~\ref{prop:mixing-alt-main-rd} gives the coverage expansion under a one-dimensional dependence assumption on the Gaussian reference field. The explicit remainder $r_d$ isolates the effect of local clustering in the Gaussian Poisson approximation.

\begin{remark}\label{rem:mix-rate}
Assume $L_n^{12}\le cn$, $q_d\ge1$ and $k_0\ge k$.
Constants in this remark may also depend on the fixed truncation
constant $A>0$. The remainder $r_d$ is driven mainly by the
block-length ratio $\ell_d/m_d$. The definition of $\epsn$ gives
\[
 c_1\frac{(\log n)^4}{n}
 \le \epsn^2
 \le C_1 n^{-3/4}\log n,
\]
where $c_1,C_1>0$ depend only on the fixed assumption constants.
Thus $k_0\asymp\log n$ for every dimension sequence in the stated regime,
and the definition of $\ell_d$ yields
\[
 \ell_d\le C\log n\,\log(2d).
\]
For $s=k_0+1$, Stirling's inequality gives
\[
 \frac{(3\Lambda_{k,\epsilon})^s}{s!}
 \le\left(\frac{3e\Lambda_{k,\epsilon}}s\right)^s
 =o(n^{-M})\qquad\text{for every fixed }M>0.
\]
Consequently,
\begin{align*}
r_d
&\le
C\frac{\log n\,\log(2d)}{d^{\beta_*/4}}
+
C d^{-1+\beta_*/4}
+
C d^{-3\beta_*/4}(\log d)^{\beta_*/2}
\\
&\quad+C(k_0+1)^{k+1}n^{-8}d^{-7k_0-15}
+
C\frac{(3\Lambda_{k,\epsilon})^{k_0+1}}{(k_0+1)!}.
\end{align*}
In particular, a sufficient condition for $r_d=O(n^{-1})$ is
\[
d^{\beta_*/4}\ge C n\log n\,\log(2d).
\]
If $d=n^c$, then
\begin{align*}
r_d
&\le
C n^{-c\beta_*/4}\log^2 n
+
C n^{-c(1-\beta_*/4)}
+
C n^{-3c\beta_*/4}(\log n)^{\beta_*/2}
\\
&\quad+C(k_0+1)^{k+1}n^{-8-c(7k_0+15)}
+
C\frac{(3\Lambda_{k,\epsilon})^{k_0+1}}{(k_0+1)!},
\end{align*}
so $r_d=O(n^{-1})$ whenever $c>4/\beta_*$. The polynomial dimension regime also satisfies \mbox{$L_n^{12}/n\to0$}.
\end{remark}

The proof of Theorem~\ref{prop:mixing-alt-main-rd} is given in Appendix~\ref{supp-sec:mixing-alt}. The coverage argument uses the joint expansion and the covariance bound implied by Gaussian mixing; the remainder $r_d$ follows from the finite Gaussian block calculation.

\section{Simulation}\label{sec:sim}

We investigate the finite-sample size of the bootstrap procedures for the $k$th largest coordinate
$T_{n,[1]} \ge T_{n,[2]} \ge \cdots \ge T_{n,[d]}$
of
$\bm S_n = \frac{1}{\sqrt n}\sum_{i=1}^n \bm X_i.$
We report results for
$k \in \{2,5,10\}.$
The case $k=1$ coincides with the maximum and is therefore omitted here.

Throughout the simulation, the dimension is fixed at
$d=400,$
and the sample size is taken from
$n \in \{200,400\}.$
For the dependence structure, we consider two correlation designs. In Design I,
\[
\mathbf{R} = \rho\, \mathbf{1}_d \mathbf{1}_d^\top + (1-\rho) \mathbf{I}_d,
\]
and in Design II,
\[
\mathbf{R}= (\rho^{|j-\ell|}),
\qquad 1 \le j,\ell \le d,
\]
with
$\rho \in \{0.2,0.8\}.$

Let $\Phi$ denote the standard normal distribution function. For $\theta>0$, let $F_\theta$ be the distribution function of the gamma distribution with shape parameter $\theta$ and unit scale. For each Monte Carlo repetition, we first generate
\[
\bm Z_i = (Z_{i1},\dots,Z_{id})^\top \sim N(\bm 0,\mathbf{R}),
\qquad i=1,\dots,n,
\]
independently, and then define
\[
U_{ij} = F_\theta^{-1}\bigl(\Phi(Z_{ij})\bigr),
\qquad 1\le i\le n,\; 1\le j\le d.
\]
This yields a Gaussian-copula model with gamma marginals.

We consider two cases.
\begin{itemize}
\item {\it Asymmetric case.}
We set $\theta=1$ and define
\[
\bm X_i = \bm U_i - \mathbf{1}_d,
\qquad i=1,\dots,n,
\]
where $\bm U_i = (U_{i1},\dots,U_{id})^\top$. Since each marginal has mean $1$, the vector $\bm X_i$ is centered.

\item {\it Symmetric case.}
We set $\theta=\tfrac12$. Let $\bm U_i'$ be an independent copy of $\bm U_i$, and define
\[
\bm X_i = \bm U_i - \bm U_i',
\qquad i=1,\dots,n.
\]
This symmetrization removes skewness. The choice $\theta=\tfrac12$ keeps the marginal kurtosis on the same scale as in the asymmetric setup.
\end{itemize}

The distributional regularity of these models and of the beta multipliers
below is verified in Section~\ref{app:simulation-regularity}.

We consider the following bootstrap methods:
\begin{itemize}

\item \textit{Empirical bootstrap (EB):} the standard nonparametric bootstrap.
\item \textit{Gaussian wild bootstrap (GB):}
$w_i \sim N(0,1).$

\item 
\textit{Mammen wild bootstrap (MB):}
\[
\Prob\!\left(w_i=\frac{1+\sqrt5}{2}\right)=\frac{\sqrt5-1}{2\sqrt5},
\qquad
\Prob\!\left(w_i=-\frac{\sqrt5-1}{2}\right)=\frac{\sqrt5+1}{2\sqrt5}.
\]

\item 
\textit{Rademacher wild bootstrap (RB):}
$\Prob(w_i=1)=\Prob(w_i=-1)=\frac12.$

\item 
\textit{Beta wild bootstrap (BB):}
let $\nu=0.1$ and define
\[
c_\nu = \nu^2 + 20\nu + 20,
\alpha_\nu
=
\frac{\nu}{2}
\left(
1-\frac{\nu+2}{\sqrt{c_\nu}}
\right),
\beta_\nu
=
\frac{\nu}{2}
\left(
1+\frac{\nu+2}{\sqrt{c_\nu}}
\right).
\]
Let $\eta_i \sim \mathrm{Beta}(\alpha_\nu,\beta_\nu)$ i.i.d., and standardize by
\[
w_i
=
\frac{\eta_i-\E[\eta_i]}{\sqrt{\Var(\eta_i)}}.
\]
Then
\[
\E[w_i]=0,
\qquad
\E[w_i^2]=1,
\qquad
\E[w_i^3]=1.
\]

\item 
\textit{Double wild bootstrap (DB):} the procedure in Section~\ref{subsec:db}.
\end{itemize}

For the Monte Carlo implementation, we use
$B_1=499$
first-level bootstrap replications for EB, GB, MB, RB, and BB, and for DB we use
$B_1=499,
B_2=99$
at the first and second bootstrap levels, respectively.

Tables \ref{tab:rej_case_k2}-\ref{tab:rej_case_k10} report the empirical sizes of different bootstrap methods at the $10\%$ level for $k=2,5,10$, respectively. Across $k\in\{2,5,10\}$, the qualitative ordering of the bootstrap procedures is largely unchanged. Finite-sample distortions vary more substantially across symmetry and covariance designs than across the values of $k$ considered. EB is generally conservative, with the under-rejection being especially visible in the asymmetric settings and in symmetric Design~II, although the distortion is somewhat mitigated as $n$ increases. GB is more design-sensitive: it is reasonably well calibrated in symmetric Design~I, but becomes distinctly liberal under asymmetry, particularly when $n$ is small and $\rho=0.2$. MB and BB display the most stable behavior overall; both are typically mildly conservative, yet they avoid the substantial over-rejection exhibited by GB and, more markedly, RB, and their performance is comparatively robust across designs and values of $k$. RB is the least robust method: it is very accurate, and often closest to the nominal level, in the symmetric experiments, but it becomes severely oversized under asymmetry, especially in Design~II. DB is often close to the nominal level in the asymmetric designs, with mild over-rejection in most settings; under symmetry it also tends to be slightly liberal. Larger $n$ generally improves calibration, and increasing $k$ attenuates some distortions, but these effects are quantitative rather than qualitative. Across the designs considered, MB and BB provide comparatively stable calibration, whereas RB is more sensitive to asymmetry.

\begin{table}[!htbp]
\small
\centering
\caption{Empirical sizes at the $10\%$ level for the second largest coordinate $T_{n,[2]}$.}
\label{tab:rej_case_k2}
\setlength{\tabcolsep}{6pt}
\renewcommand{\arraystretch}{1.0}

\begin{tabular}{lcccccccc}
\hline
Design & $n$ & $\rho$ & EB & GB & MB & RB & BB & DB \\
\hline
\multicolumn{9}{l}{\textbf{Panel A: Asymmetric}} \\
\hline
I  & 200 & 0.2 & 0.0612 & 0.1211 & 0.0757 & 0.1464 & 0.0743 & 0.1114 \\
I  & 200 & 0.8 & 0.0731 & 0.0869 & 0.0739 & 0.0907 & 0.0719 & 0.1002 \\
I  & 400 & 0.2 & 0.0732 & 0.1172 & 0.0809 & 0.1301 & 0.0802 & 0.1044 \\
I  & 400 & 0.8 & 0.0814 & 0.0954 & 0.0809 & 0.0962 & 0.0821 & 0.1034 \\
\cline{1-9}
II & 200 & 0.2 & 0.0619 & 0.152 & 0.0888 & 0.215 & 0.0884 & 0.115 \\
II & 200 & 0.8 & 0.0686 & 0.137 & 0.0865 & 0.174 & 0.0857 & 0.107 \\
II & 400 & 0.2 & 0.0790 & 0.153 & 0.0946 & 0.182 & 0.0936 & 0.109 \\
II & 400 & 0.8 & 0.0826 & 0.134 & 0.0927 & 0.156 & 0.0921 & 0.105 \\
\hline
\multicolumn{9}{l}{\textbf{Panel B: Symmetric}} \\
\hline
I  & 200 & 0.2 & 0.0731 & 0.0829 & 0.0892 & 0.1047 & 0.0907 & 0.120 \\
I  & 200 & 0.8 & 0.0970 & 0.1015 & 0.1002 & 0.1058 & 0.0993 & 0.114 \\
I  & 400 & 0.2 & 0.0868 & 0.0914 & 0.0963 & 0.1035 & 0.0934 & 0.109 \\
I  & 400 & 0.8 & 0.0993 & 0.1016 & 0.1029 & 0.1038 & 0.1026 & 0.109 \\
\cline{1-9}
II & 200 & 0.2 & 0.0584 & 0.0653 & 0.0830 & 0.106 & 0.0814 & 0.116 \\
II & 200 & 0.8 & 0.0677 & 0.0759 & 0.0859 & 0.101 & 0.0848 & 0.104 \\
II & 400 & 0.2 & 0.0763 & 0.0807 & 0.0902 & 0.102 & 0.0893 & 0.107 \\
II & 400 & 0.8 & 0.0877 & 0.0911 & 0.0977 & 0.106 & 0.0984 & 0.110 \\
\hline
\end{tabular}
\end{table}

\begin{table}[!htbp]
\small
\centering
\caption{Empirical sizes at the $10\%$ level for the 5th largest coordinate $T_{n,[5]}$.}
\label{tab:rej_case_panel}
\setlength{\tabcolsep}{6pt}
\renewcommand{\arraystretch}{1.0}

\begin{tabular}{lcccccccc}
\hline
Design & $n$ & $\rho$ & EB & GB & MB & RB & BB & DB \\
\hline
\multicolumn{9}{l}{\textbf{Panel A: Asymmetric}} \\
\hline
I  & 200 & 0.2 & 0.0645 & 0.1140 & 0.0731 & 0.1298 & 0.0730 & 0.107 \\
I  & 200 & 0.8 & 0.0738 & 0.0866 & 0.0742 & 0.0878 & 0.0735 & 0.099 \\
I  & 400 & 0.2 & 0.0754 & 0.1119 & 0.0823 & 0.1235 & 0.0809 & 0.102 \\
I  & 400 & 0.8 & 0.0835 & 0.0953 & 0.0852 & 0.0955 & 0.0843 & 0.103 \\
\cline{1-9}
II & 200 & 0.2 & 0.0565 & 0.1510 & 0.0887 & 0.2220 & 0.0854 & 0.118 \\
II & 200 & 0.8 & 0.0712 & 0.1380 & 0.0883 & 0.1730 & 0.0875 & 0.111 \\
II & 400 & 0.2 & 0.0731 & 0.1510 & 0.0907 & 0.1880 & 0.0884 & 0.111 \\
II & 400 & 0.8 & 0.0792 & 0.1290 & 0.0890 & 0.1450 & 0.0886 & 0.101 \\
\hline
\multicolumn{9}{l}{\textbf{Panel B: Symmetric}} \\
\hline
I  & 200 & 0.2 & 0.0820 & 0.0891 & 0.0918 & 0.1050 & 0.0914 & 0.113 \\
I  & 200 & 0.8 & 0.0985 & 0.1035 & 0.1006 & 0.1060 & 0.0974 & 0.112 \\
I  & 400 & 0.2 & 0.0933 & 0.0962 & 0.0986 & 0.1050 & 0.0979 & 0.109 \\
I  & 400 & 0.8 & 0.1005 & 0.1009 & 0.1007 & 0.1030 & 0.1017 & 0.107 \\
\cline{1-9}
II & 200 & 0.2 & 0.0587 & 0.0629 & 0.0826 & 0.1070 & 0.0798 & 0.122 \\
II & 200 & 0.8 & 0.0711 & 0.0780 & 0.0867 & 0.1010 & 0.0883 & 0.109 \\
II & 400 & 0.2 & 0.0781 & 0.0802 & 0.0918 & 0.1050 & 0.0917 & 0.115 \\
II & 400 & 0.8 & 0.0886 & 0.0906 & 0.0960 & 0.1040 & 0.0957 & 0.107 \\
\hline
\end{tabular}
\end{table}

\begin{table}[!htbp]
\small
\centering
\caption{Empirical sizes at the $10\%$ level for the 10th largest coordinate $T_{n,[10]}$.}
\label{tab:rej_case_k10}
\setlength{\tabcolsep}{6pt}
\renewcommand{\arraystretch}{1.0}

\begin{tabular}{lcccccccc}
\hline
Design & $n$ & $\rho$ & EB & GB & MB & RB & BB & DB \\
\hline
\multicolumn{9}{l}{\textbf{Panel A: Asymmetric}} \\
\hline
I  & 200 & 0.2 & 0.0695 & 0.1073 & 0.0769 & 0.1173 & 0.0757 & 0.1052 \\
I  & 200 & 0.8 & 0.0745 & 0.0853 & 0.0749 & 0.0871 & 0.0730 & 0.0974 \\
I  & 400 & 0.2 & 0.0770 & 0.1046 & 0.0798 & 0.1111 & 0.0796 & 0.0983 \\
I  & 400 & 0.8 & 0.0843 & 0.0944 & 0.0867 & 0.0948 & 0.0856 & 0.1027 \\
\cline{1-9}
II & 200 & 0.2 & 0.0537 & 0.142 & 0.0812 & 0.223 & 0.0782 & 0.117 \\
II & 200 & 0.8 & 0.0774 & 0.132 & 0.0883 & 0.159 & 0.0875 & 0.110 \\
II & 400 & 0.2 & 0.0732 & 0.146 & 0.0905 & 0.187 & 0.0895 & 0.113 \\
II & 400 & 0.8 & 0.0842 & 0.129 & 0.0932 & 0.146 & 0.0919 & 0.104 \\
\hline
\multicolumn{9}{l}{\textbf{Panel B: Symmetric}} \\
\hline
I  & 200 & 0.2 & 0.0851 & 0.0894 & 0.0901 & 0.0997 & 0.0900 & 0.110 \\
I  & 200 & 0.8 & 0.1005 & 0.1033 & 0.1014 & 0.1056 & 0.1002 & 0.110 \\
I  & 400 & 0.2 & 0.0945 & 0.0969 & 0.0974 & 0.1019 & 0.0982 & 0.106 \\
I  & 400 & 0.8 & 0.1004 & 0.1014 & 0.1002 & 0.1026 & 0.1017 & 0.105 \\
\cline{1-9}
II & 200 & 0.2 & 0.0608 & 0.0636 & 0.0822 & 0.108 & 0.0789 & 0.124 \\
II & 200 & 0.8 & 0.0764 & 0.0786 & 0.0890 & 0.103 & 0.0904 & 0.110 \\
II & 400 & 0.2 & 0.0756 & 0.0774 & 0.0894 & 0.104 & 0.0891 & 0.114 \\
II & 400 & 0.8 & 0.0885 & 0.0902 & 0.0951 & 0.102 & 0.0960 & 0.107 \\
\hline
\end{tabular}
\end{table}

\section{Conclusion}\label{sec:con}
This paper studies Gaussian and bootstrap approximations for the $k$th largest coordinate statistic $T_{n,[k]}$ in high dimensions. We establish theoretical guarantees that justify bootstrap critical values when the ambient dimension is allowed to grow with the sample size, thereby extending valid inference beyond the maximum to nonmaximal order statistics. The simulation results show how finite-sample calibration varies with marginal symmetry, covariance design and the choice of bootstrap multipliers.

An important direction for future research is to develop analogous Gaussian approximation results for temporally dependent observations. Doing so would require a theory that accommodates serial dependence, long-run covariance estimation, and resampling schemes that preserve the time-series structure; see, for example, \citep{Shao2010,ZhangWu2017,ZhangCheng2014,ZhangCheng2018,ChangChenWu2024,ChangJiangShao2023,ChangJiangMcElroyShao2025}.

\section*{Declarations}
\noindent\textbf{Funding.}
This work was supported by the National Key Research and Development Program of China (2022YFA1003703 and 2022YFA1003800); the National Natural Science Foundation of China (12271271, 12426303, 12531011, 12231011, 12671323, and 12601527); the Tianjin Science Fund for Outstanding Young Scholar (23JCJQJC00150); the Fundamental Research Funds for the Central Universities (63233075 and ZB22000105); and the Tianjin Science and Technology Program (24ZXZSSS00320).

\noindent\textbf{Competing interests.}
The authors declare no competing interests.

\noindent\textbf{Author contributions.}
Both authors contributed to the research and the manuscript.

\noindent\textbf{Data availability.}
The numerical results are based on simulated data. The data-generating models
and simulation settings are described in Section~\ref{sec:sim}, and the reported
results appear in Tables~\ref{tab:rej_case_k2}--\ref{tab:rej_case_k10}.

\noindent\textbf{Use of artificial intelligence.}
ChatGPT (OpenAI) was used to assist with checking and revising mathematical
arguments, language editing, and manuscript preparation. Responsibility for
the mathematical results and the final manuscript rests with the authors.

\clearpage
\appendix
\numberwithin{equation}{section}
\section*{Appendices}
The combinatorial and projected identities are followed by the Gaussian
derivative bounds and distributional expansions. We then prove the quantile,
coverage and double-bootstrap results. Section~\ref{app:simulation-regularity}
verifies the distributional conditions of the simulation models.
Appendix~\ref{supp-sec:mixing-alt} treats the stationary mixing alternative.

\section{Proofs of the main results}\label{supp-sec:proofs}
Write $L=L_n=\log(end)$ and $r_n=L_n^{12}/n$.
We use $0<\epsilon<1/4$; statements for larger central trimming parameters
follow by restriction. For the scalar Stein identities below, use the
symmetric kernels
\[
 \mathbf T_i(\bm X_i)
 =\{\bm\tau_i(\bm X_i)+\bm\tau_i(\bm X_i)^\top\}/2.
\]
Assumption~\ref{ass:data}, applied to gradient tests, gives
\[
 \E\{\bm X_i^\top\nabla h(\bm X_i)\}
 =\E\langle\mathbf T_i(\bm X_i),\nabla^2h(\bm X_i)\rangle,
 \qquad \E\mathbf T_i=\E\bm X_i^{\otimes2}.
\]
Put $\mathbf A_\alpha=\mathbf\Psi_\alpha/(2f_\alpha)$.
Symmetrization preserves the stated moment bounds. The constants below
may depend on the fixed parameters in the main assumptions and on $k,\epsilon$.
For a random tensor $\mathsf V$ and $q\ge1$, write
\[
 \|\mathsf V\|_{L^q(\max)}
 :=\{\E\|\mathsf V\|_\infty^q\}^{1/q}.
\]
Conditional versions use the conditional expectation specified in the argument.

\subsection{Combinatorial identities}

\begin{lemma}[Finite inclusion--exclusion and its remainder]
\label{lem:comb}\label{lem:finite-ie-corrected}
Let $1\le k\le K$ be integers and let $N$ be a nonnegative integer-valued random variable. Put $a_{s,k}:=\binom{s-1}{k-1}$ and interpret $\binom{N}{s}=0$ for $s>N$. Then
\begin{equation*}
 \ind\{N\ge k\}=\sum_{s=k}^{N}(-1)^{s-k}a_{s,k}\binom Ns.
\end{equation*}
Moreover, almost surely,
\[
0\le(-1)^{K-k}\left\{
 \sum_{s=k}^{K}(-1)^{s-k}a_{s,k}\binom Ns-\ind\{N\ge k\}
\right\}
\le \binom K{k-1}\binom N{K+1}.
\]
Consequently, whenever $\E\binom N{K+1}<\infty$,
\[
\left|\Prob(N\ge k)-\sum_{s=k}^{K}(-1)^{s-k}a_{s,k}\E\binom Ns\right|
\le\binom K{k-1}\E\binom N{K+1}.
\]
\end{lemma}

\begin{proof}[Proof of Lemma~\ref{lem:finite-ie-corrected}]
For integers $m\ge s\ge k$,
\[
\binom{s-1}{k-1}\binom ms
=\frac{k}{s}\binom mk\binom{m-k}{s-k}.
\]
Thus, if $m\ge k$, the complete sum equals
\begin{align*}
\sum_{s=k}^{m}(-1)^{s-k}a_{s,k}\binom ms
&=k\binom mk\int_0^1x^{k-1}
 \sum_{r=0}^{m-k}(-1)^r\binom{m-k}{r}x^r\,dx\\
&=k\binom mk\int_0^1x^{k-1}(1-x)^{m-k}\,dx\\
&=k\binom mk\frac{(k-1)!(m-k)!}{m!}=1.
\end{align*}
For $m<k$, both sides vanish. If $m\le K$, the truncated sum is already complete.

For $m\ge K+1$, write $h=m-k$, $M=K-k$, and $P_M(x)=\sum_{r=0}^{M}(-1)^r\binom hrx^r$. Taylor's integral remainder gives, for $0\le x\le1$,
\begin{gather*}
(1-x)^h-P_M(x)
=\frac{(-1)^{M+1}h!}{M!(h-M-1)!}
 \int_0^x(x-u)^M(1-u)^{h-M-1}\,du,\\
0\le(-1)^M\{P_M(x)-(1-x)^h\}
\le\binom h{M+1}x^{M+1}.
\end{gather*}
Multiply by $k\binom mkx^{k-1}$ and integrate over $[0,1]$. The upper bound becomes
\[
\frac{k}{K+1}\binom mk\binom{m-k}{K+1-k}
=\binom K{k-1}\binom m{K+1}.
\]
Substitution of $m=N$ and expectation prove the assertions.
\end{proof}

Taking expectations of the complete identity gives
\begin{equation*}
 \Prob(T_{n,[k]}>t)
 =\sum_{s=k}^{d}(-1)^{s-k}\binom{s-1}{k-1}V_{n,s}(t).
\end{equation*}
The same identity holds for $N_n^*(t)$ conditionally and for $N_Z(t)$.

\subsection{Projected quantities}

For every nonempty $I=\{i_1,\dots,i_s\}\subset[d]$, let $\bm P_I:\R^d\to\R^s$ denote the coordinate projection. Define
\begin{equation*}
\bm S_{n,I}:=\bm P_I\bm S_n,
\qquad
\mathbf{\Sigma}_{II}:=\bm P_I\mathbf{\Sigma}\bm P_I^\top,
\qquad
\ph_I:=\ph_{\mathbf{\Sigma}_{II}}.
\end{equation*}
Also define
\begin{equation*}
\bar{\bm X}_I^{\,r}:=\frac1n\sum_{i=1}^n (\bm P_I\bm X_i)^{\otimes r},
\qquad
\bm b_{i,I}:=\bm P_I(\bm X_i-\bar{\bm X}),
\qquad
\bar{\bm b}_I^{\,r}:=\frac1n\sum_{i=1}^n \bm b_{i,I}^{\otimes r}.
\end{equation*}
The data projection and the centered bootstrap projection have signed densities
\begin{align*}
 p_{n,I}(\bm u)
 &:=
 \ph_I(\bm u)
 -\frac{1}{6\sqrt n}\bigl\langle\E[\bar{\bm X}_I^{\,3}],\nabla^3\ph_I(\bm u)\bigr\rangle,
 \\
 \hat p^{\,c}_{n,\gamma,I}(\bm u)
 &:=
 \ph_I(\bm u)
 +\frac12\bigl\langle\bar{\bm b}_I^{\,2}-\mathbf{\Sigma}_{II},\nabla^2\ph_I(\bm u)\bigr\rangle
 -\frac{\gamma}{6\sqrt n}\bigl\langle\bar{\bm b}_I^{\,3},\nabla^3\ph_I(\bm u)\bigr\rangle.
\end{align*}

\begin{lemma}[Projection preserves the data-side assumptions]
\label{lem:projection}
Assume Assumption~\ref{ass:data}. For every nonempty $I\subset[d]$,
the projected vectors $\bm P_I\bm X_1,\dots,\bm P_I\bm X_n$ satisfy the same
Stein identity and sub-exponential bounds with envelope $b$.
Their normalized sum has covariance $\mathbf\Sigma_{II}$, and
\begin{equation*}
\lambda_{\min}(\mathbf{\Sigma}_{II})\ge \sigma_*^2.
\end{equation*}
\end{lemma}

\begin{proof}[Proof of Lemma~\ref{lem:projection}]
Let $\bm Y_i:=\bm P_I\bm X_i$. For any smooth $g:\R^{|I|}\to\R^{|I|}$ define
\begin{equation*}
 f(\bm x):=\bm P_I^\top g(\bm P_I\bm x).
\end{equation*}
Then
\begin{equation*}
\nabla f(\bm x)^\top
=
\bm P_I^\top \nabla g(\bm P_I\bm x)^\top \bm P_I.
\end{equation*}
Applying the Stein identity for $\bm X_i$ yields
\begin{align*}
\E[\bm Y_i^\top g(\bm Y_i)]
&=
\E[\bm X_i^\top f(\bm X_i)]\\
&=
\E\Bigl[\tr\{\bm\tau_i(\bm X_i)\bm P_I^\top \nabla g(\bm P_I\bm X_i)^\top\bm P_I\}\Bigr]\\
&=
\E\Bigl[\tr\{\bm P_I\bm\tau_i(\bm X_i)\bm P_I^\top\nabla g(\bm Y_i)^\top\}\Bigr].
\end{align*}
Hence $\E\{\bm P_I\bm\tau_i(\bm X_i)\bm P_I^\top\mid\bm Y_i\}$ is a Stein kernel for $\bm Y_i$. 
Projection preserves the coordinate envelope. Conditional Jensen also gives
\[
 \left\|\E\{\tau_{i,j\ell}(\bm X_i)-\E\tau_{i,j\ell}(\bm X_i)
                      \mid\bm Y_i\}\right\|_{\psi_1}
 \le\|\tau_{i,j\ell}(\bm X_i)-\E\tau_{i,j\ell}(\bm X_i)\|_{\psi_1}
 \le b^2.
\]
 Finally, for every nonzero $\bm u\in\R^{|I|}$,
\begin{equation*}
\bm u^\top\mathbf{\Sigma}_{II}\bm u
=
(\bm P_I^\top\bm u)^\top\mathbf{\Sigma}(\bm P_I^\top\bm u)
\ge
\sigma_*^2\|\bm P_I^\top\bm u\|_2^2
=
\sigma_*^2\|\bm u\|_2^2.
\end{equation*}
\end{proof}

The raw-moment projected density $\hat p_{n,\gamma,I}$ is obtained from
the main-text $\hat p_{n,\gamma}$ by marginalization. The difference between
raw and centered bootstrap moments is controlled below.

\subsection{External matrix, Gaussian-comparison, and Koike lemmas}

\begin{lemma}[Gershgorin interval theorem]
\label{lem:gershgorin}
Let $A=(a_{j\ell})\in\R^{s\times s}$ be symmetric. Then
\begin{equation}
\lambda_{\min}(A)\ge \min_{1\le j\le s}\{a_{jj}-R_j(A)\},
\qquad
\lambda_{\max}(A)\le \max_{1\le j\le s}\{a_{jj}+R_j(A)\}.
\label{eq:gershgorin-interval}
\end{equation}
\end{lemma}

\begin{proof}[Proof of Lemma~\ref{lem:gershgorin}]
Let $\lambda$ be an eigenvalue of $A$ with eigenvector $x=(x_1,\dots,x_s)^\top\ne 0$. Choose
\begin{equation*}
j_0\in\arg\max_{1\le j\le s}|x_j|.
\end{equation*}
Since $Ax=\lambda x$,
\begin{equation*}
(\lambda-a_{j_0j_0})x_{j_0}
=
\sum_{\ell\ne j_0}a_{j_0\ell}x_\ell.
\end{equation*}
Hence
\begin{equation*}
|\lambda-a_{j_0j_0}||x_{j_0}|
\le
\sum_{\ell\ne j_0}|a_{j_0\ell}||x_\ell|
\le
R_{j_0}(A)|x_{j_0}|,
\end{equation*}
and therefore
\begin{equation*}
|\lambda-a_{j_0j_0}|\le R_{j_0}(A).
\end{equation*}
This proves that every eigenvalue belongs to at least one Gershgorin interval
\begin{equation*}
[a_{jj}-R_j(A),a_{jj}+R_j(A)],
\qquad
j=1,\dots,s.
\end{equation*}
Taking the minimum and maximum over these intervals yields \eqref{eq:gershgorin-interval}.
\end{proof}

\begin{lemma}[Berman--Li--Shao normal comparison inequality]
\label{lem:berman}
Let $\xi=(\xi_1,\dots,\xi_s)^\top$ and $\eta=(\eta_1,\dots,\eta_s)^\top$ be centered Gaussian vectors with
\begin{equation*}
\Var(\xi_j)=\Var(\eta_j)=1,
\qquad
1\le j\le s.
\end{equation*}
Write $r_{j\ell}^\xi:=\Corr(\xi_j,\xi_\ell)$ and $r_{j\ell}^\eta:=\Corr(\eta_j,\eta_\ell)$, and define
\begin{equation*}
\rho_{j\ell}:=\max\{|r_{j\ell}^\xi|,|r_{j\ell}^\eta|\}.
\end{equation*}
Then for every $u=(u_1,\dots,u_s)^\top\in\R^s$,
\begin{align}
&\left|
\Prob(\xi_1\le u_1,\dots,\xi_s\le u_s)
-
\Prob(\eta_1\le u_1,\dots,\eta_s\le u_s)
\right|\notag\\
&\qquad\le
\frac{1}{2\pi}
\sum_{1\le j<\ell\le s}
\left|\arcsin(r_{j\ell}^\xi)-\arcsin(r_{j\ell}^\eta)\right|
\exp\!\left(
-\frac{u_j^2+u_\ell^2}{2(1+\rho_{j\ell})}
\right).
\label{eq:berman-main}
\end{align}
In particular, if $\eta$ has independent coordinates and
\begin{equation*}
\bar\rho:=\max\bigl(\{0\}\cup
       \{|r_{j\ell}^\xi|:1\le j<\ell\le s\}\bigr)<1,
\end{equation*}
then
\begin{align}
&\left|
\Prob(\xi_1\le u_1,\dots,\xi_s\le u_s)
-
\prod_{j=1}^s \Phi(u_j)
\right|\notag\\
&\qquad\le
\frac{1}{2\pi\sqrt{1-\bar\rho^2}}
\sum_{1\le j<\ell\le s}|r_{j\ell}^\xi|
\exp\!\left(
-\frac{u_j^2+u_\ell^2}{2(1+\bar\rho)}
\right).
\label{eq:berman-independent}
\end{align}
\end{lemma}

\begin{proof}[Proof of Lemma~\ref{lem:berman}]
Apply \cite[Theorem~2.1, equation~(2.3)]{LiShao2002} in both
directions to obtain \eqref{eq:berman-main}. This refines Berman's comparison
argument \citep{Berman1964}. Singular correlation matrices follow by continuity under
\[
 \mathbf R_\delta=(1-\delta)\mathbf R+\delta\mathbf I_s
       \longrightarrow\mathbf R,\qquad \delta\downarrow0,
\]
since every Gaussian marginal is continuous. If $\eta$ has independent
coordinates, then $r_{j\ell}^\eta=0$ for $j\ne\ell$, and the mean-value theorem gives
\begin{equation*}
\left|\arcsin(r_{j\ell}^\xi)-\arcsin(0)\right|
\le
\sup_{|u|\le \bar\rho}\frac{1}{\sqrt{1-u^2}}|r_{j\ell}^\xi|
\le
\frac{|r_{j\ell}^\xi|}{\sqrt{1-\bar\rho^2}}.
\end{equation*}
Substituting this estimate into \eqref{eq:berman-main} yields \eqref{eq:berman-independent}.
\end{proof}

\begin{lemma}[Koike smoothing identity]
\label{lem:koike-smoothing}
Let $A\subset\R^s$ be measurable, let $Z_I\sim N(0,\Sigma_{II})$ with $\Sigma_{II}$ positive definite, and define for $u\in(0,1]$,
\begin{equation*}
h_{A,u}(x)
:=
\E\bigl[\1_A(\sqrt{1-u}\,x+\sqrt u\,Z_I)\bigr].
\end{equation*}
Then
\begin{equation}
h_{A,u}(x)
=
\int_A
u^{-s/2}
\phi_I\!\left(\frac{y-\sqrt{1-u}\,x}{\sqrt u}\right)\,dy.
\label{eq:koike-h-repr}
\end{equation}
Moreover, for every multi-index $\alpha\in\mathbb N_0^s$ with $|\alpha|=r\ge 1$,
\begin{align}
\partial^\alpha h_{A,u}(x)
&=
(-1)^r
\left(\frac{1-u}{u}\right)^{r/2}
\int_A
\partial^\alpha\phi_I\!\left(\frac{y-\sqrt{1-u}\,x}{\sqrt u}\right)
u^{-s/2}\,dy .
\label{eq:koike-derivative-formula}
\end{align}
\end{lemma}

\begin{proof}[Proof of Lemma~\ref{lem:koike-smoothing}]
Formulas \eqref{eq:koike-h-repr}--\eqref{eq:koike-derivative-formula} are the projected specialization of equations~(4.4)--(4.5) in \cite{Koike2026}. From the definition,
\begin{align*}
h_{A,u}(x)
&=
\int_{\R^s}
\1_A(\sqrt{1-u}\,x+\sqrt u\,z)\phi_I(z)\,dz.
\end{align*}
Set
\begin{equation*}
y:=\sqrt{1-u}\,x+\sqrt u\,z,
\qquad
z=\frac{y-\sqrt{1-u}\,x}{\sqrt u},
\qquad
dz=u^{-s/2}\,dy.
\end{equation*}
Then \eqref{eq:koike-h-repr} follows. Differentiate \eqref{eq:koike-h-repr} under the integral sign. For each derivative $\partial_{x_j}$,
\begin{equation*}
\partial_{x_j}
\phi_I\!\left(\frac{y-\sqrt{1-u}\,x}{\sqrt u}\right)
=
-\sqrt{\frac{1-u}{u}}
\,
\partial_{j}\phi_I\!\left(\frac{y-\sqrt{1-u}\,x}{\sqrt u}\right).
\end{equation*}
Applying this identity $r=|\alpha|$ times gives \eqref{eq:koike-derivative-formula}.
\end{proof}

For $1\le k\le d$ and $\bm a\in\mathbb R^d$, define
\[
 A_k(\bm a)=\left\{\bm z\in\mathbb R^d:
   \sum_{j=1}^d\mathbf1\{z_j>a_j\}\le k-1\right\},
 \qquad \mathcal A_k=\{A_k(\bm a):\bm a\in\mathbb R^d\}.
\]
In particular, $A_k(t\bm1)=\{\bm z:z_{[k]}\le t\}$. Set $L_d=\log(ed)$.

\begin{lemma}[Global Gaussian derivatives]\label{lem:global-order-derivatives}\label{lem:koike-orthant-derivative}
Let $1\le k\le d$ and $r\ge1$ be fixed integers, and let
$\mathbf\Sigma\succeq\sigma_*^2\mathbf I_d$ with $\sigma_*>0$.
There is a constant $C_{k,r}$, independent of $d,\mathbf\Sigma,\bm a$, such that
\[
 \sup_{\bm a\in\mathbb R^d}
 \left\|\int_{A_k(\bm a)}\nabla^r\phi_{\mathbf\Sigma}(\bm z)\,d\bm z\right\|_1
 \le C_{k,r}\sigma_*^{-r}L_d^{r/2}.
\]
The tensor norm sums absolute values over all ordered $r$-tuples of coordinate indices.
\end{lemma}

\begin{proof}[Proof of Lemma~\ref{lem:global-order-derivatives}]
First let $\mathbf\Sigma=\mathbf I_d$ and let $B_j$ be independent
Bernoulli variables. Put
\[
 \begin{gathered}
 F_k^0(\bm a)=\Prob(\bm Z\in A_k(\bm a)),\quad
 \Prob(B_j=1)=p_j=\bar\Phi(a_j),\\
 N=\sum_jB_j,\quad N_{-J}=\sum_{j\notin J}B_j.
 \end{gathered}
\]

Let a multi-index have distinct coordinate set $J$, cardinality $l$, and
positive multiplicities $m_j$ satisfying $\sum_{j\in J}m_j=r$.
Repeated conditioning on the Bernoulli variables gives
\begin{align*}
 &\partial^{\bm m}F_k^0(\bm a)\\
 &\quad=\left\{\prod_{j\in J}(-1)^{m_j-1}H_{m_j-1}(a_j)\phi(a_j)\right\}
   \sum_{v=0}^{l-1}(-1)^v\binom{l-1}{v}
       \Prob(N_{-J}=k-1-v),
\end{align*}
Here $H_m$ is the probabilists' Hermite polynomial and negative target
counts have probability zero. The first derivative is
\[
 \partial_jF_k^0(\bm a)=\phi(a_j)\Prob(N_{-j}=k-1).
\]
Further derivatives in coordinate $j$ differentiate only $\phi(a_j)$.

Put $T=\{4(r+1)L_d\}^{1/2}$. Hermite polynomials and Mills' inequality give,
for $1\le m\le r$,
\[
 |H_{m-1}(x)|\phi(x)\le C_r e^{-x^2/4},\qquad
 |H_{m-1}(x)|\phi(x)\le C_rT^m\bar\Phi(x)\quad (|x|\le T).
\]
If a selected $|a_j|>T$, use the first bound for that factor and boundedness
of the remaining factors. The finite difference is bounded by $2^{r-1}$,
so the sum over at most $d^r$ ordered tuples is at most
\[
 C_rd^re^{-T^2/4}=C_rd^r(ed)^{-(r+1)}\le C_r.
\]

For the remaining derivatives,
\[
 \prod_{j\in J}|H_{m_j-1}(a_j)|\phi(a_j)
 \le C_rT^r\prod_{j\in J}p_j.
\]
For every $m\ge0$, assigning success to each selected coordinate gives
\[
 \begin{aligned}
 \sum_{|J|=l}\left(\prod_{j\in J}p_j\right)\Prob(N_{-J}=m)
 &=\E\left[\binom Nl\mathbf1\{N=m+l\}\right]\\
 &=\binom{m+l}{l}\Prob(N=m+l)
 \end{aligned}
\]
Here $m=k-1-v\le k-1$, so the last expression is bounded by $C_{k,r}$.
Summing over $l\le r$, $v\le l-1$ and the multiplicity patterns gives
\[
 \sup_{\bm a}\|\nabla^r F_k^0(\bm a)\|_1
 \le C_{k,r}(1+T^r)\le C_{k,r}L_d^{r/2}.
\]
The change of variables $\bm z=\bm y+\bm a$ yields
\[
 \nabla^rF_k^0(\bm a)
 =\int_{A_k(\bm0)}\nabla^r\phi_{\mathbf I_d}(\bm y+\bm a)\,d\bm y
 =\int_{A_k(\bm a)}\nabla^r\phi_{\mathbf I_d}(\bm z)\,d\bm z.
\]

For general covariance, represent
$\bm Z=\sigma_*\bm U+\bm V$, where
$\bm U\sim N(\bm0,\mathbf I_d)$ and
$\bm V\sim N(\bm0,\mathbf\Sigma-\sigma_*^2\mathbf I_d)$ are independent.
Then
\[
 \begin{aligned}
 F_{k,\mathbf\Sigma}(\bm a)
 &=\E F_k^0\{(\bm a-\bm V)/\sigma_*\},\\
 \|\nabla^rF_{k,\mathbf\Sigma}(\bm a)\|_1
 &\le\sigma_*^{-r}\E\|\nabla^rF_k^0\{(\bm a-\bm V)/\sigma_*\}\|_1.
 \end{aligned}
\]
The uniform bound just obtained justifies differentiation under expectation
and proves the result.
\end{proof}

\begin{lemma}[Translations, smoothing, and signed strips]\label{lem:global-order-strips}\label{lem:gauss-shift}\label{lem:gauss-strip}
Assume $\mathbf\Sigma\succeq\sigma_*^2\mathbf I_d$ with $\sigma_*>0$ and fix $k,r$.
For every $\bm a,\bm h\in\mathbb R^d$ and integer $r\ge0$,
\begin{align*}
 \left\|\int\{\mathbf1_{A_k(\bm a+\bm h)}(\bm z)
                 -\mathbf1_{A_k(\bm a)}(\bm z)\}
                 \nabla^r\phi_{\mathbf\Sigma}(\bm z)\,d\bm z\right\|_1
 \le C_{k,r}\|\bm h\|_\infty\sigma_*^{-(r+1)}L_d^{(r+1)/2}.
\end{align*}
For $r=0$ the tensor norm is absolute value. In particular, for every mean
vector $\bm\mu$, every $t\in\mathbb R$ and every $a\ge0$,
\[
 \Prob\{t<(\bm Z+\bm\mu)_{[k]}\le t+a\}
 \le C_ka\sigma_*^{-1}\sqrt{L_d}.
\]
If $h=\mathbf1_{A_k(\bm a)}$ and
$h_u(\bm x)=\E h(\sqrt{1-u}\bm x+\sqrt u\bm Z)$, then, for $0<u<1$,
\[
 \sup_{\bm x}\|\nabla^rh_u(\bm x)\|_1
 \le C_{k,r}\left\{\frac{1-u}{u\sigma_*^2}L_d\right\}^{r/2},\qquad r\ge1.
\]
\end{lemma}

\begin{proof}[Proof of Lemma~\ref{lem:global-order-strips}]
Let $F(\bm a)=\int_{A_k(\bm a)}\phi_{\mathbf\Sigma}$.
Integration along the straight segment gives
\[
 \nabla^rF(\bm a+\bm h)-\nabla^rF(\bm a)
 =\int_0^1\sum_{j=1}^dh_j\partial_j\nabla^rF(\bm a+v\bm h)\,dv.
\]
Apply Lemma~\ref{lem:global-order-derivatives} at order $r+1$.
The anti-concentration bound is its $r=0$ specialization with
$\bm a=t\bm1-\bm\mu$ and $\bm h=a\bm1$.

Gaussian differentiation gives
\[
 \nabla^rh_u(\bm x)
 =\left(-\sqrt{\frac{1-u}{u}}\right)^r
 \int_{A_k((\bm a-\sqrt{1-u}\bm x)/\sqrt u)}
       \nabla^r\phi_{\mathbf\Sigma}(\bm z)\,d\bm z.
\]
The claimed derivative estimate follows from
Lemma~\ref{lem:global-order-derivatives}.
\end{proof}

An approximate Stein kernel for a centered vector $\bm\xi$ is a pair
$(\mathbf\tau,\bm\beta)$ satisfying
\[
 \E\{\bm\xi^\top\nabla h(\bm\xi)\}
 =\E\langle\mathbf\tau(\bm\xi),\nabla^2h(\bm\xi)\rangle
  +\E\{\bm\beta(\bm\xi)^\top\nabla h(\bm\xi)\},
 \qquad h\in C_b^2(\R^d).
\]

For independent centered $\bm\xi_i$, abbreviate
$\mathbf\tau_i=\mathbf\tau_i(\bm\xi_i)$, $\bm\beta_i=\bm\beta_i(\bm\xi_i)$, and put
\[
 \begin{gathered}
 \bm W=\sum_i\bm\xi_i,\quad
 \overline{\mathbf T}=\sum_i\mathbf\tau_i-\mathbf\Sigma,\quad
 \bm B=\sum_i\bm\beta_i,\\
 \mathbf\Delta=\operatorname{Cov}(\bm W)-\mathbf\Sigma,\quad
 \mathsf K_3=\sum_i\E\bm\xi_i^{\otimes3}.
 \end{gathered}
\]
Use the signed density
\[
 p_W=\phi_{\mathbf\Sigma}
 +\tfrac12\langle\mathbf\Delta,\nabla^2\phi_{\mathbf\Sigma}\rangle
 -\tfrac16\langle\mathsf K_3,\nabla^3\phi_{\mathbf\Sigma}\rangle.
\]
Define the five coefficients in the Stein bound:
\begin{align*}
 a_4={}&\E\|\overline{\mathbf T}\|_\infty^2
       +\E\left\|\sum_i\mathbf\tau_i^{\otimes2}\right\|_\infty
       +\E\left\|\sum_i\bm\xi_i^{\otimes4}\right\|_\infty,\\
 a_5={}&\E\left\|\overline{\mathbf T}\otimes\sum_i\bm\xi_i^{\otimes3}\right\|_\infty
       +\E\left\|\sum_i\bm\xi_i^{\otimes3}\otimes\mathbf\tau_i\right\|_\infty,\\
 b_4={}&\E\left\|\bm B\otimes\sum_i\bm\xi_i^{\otimes3}\right\|_\infty
       +\E\left\|\sum_i\bm\xi_i^{\otimes3}\otimes\bm\beta_i\right\|_\infty\\
      &+\E\left\|\overline{\mathbf T}\otimes\sum_i\bm\beta_i\otimes\bm\xi_i\right\|_\infty
       +\E\left\|\sum_i\bm\beta_i\otimes\bm\xi_i\otimes\mathbf\tau_i\right\|_\infty,\\
 b_2={}&\E\|\bm B\|_\infty^2
       +\E\left\|\sum_i\bm\beta_i^{\otimes2}\right\|_\infty,\\
 b_3={}&\E\left\|\bm B\otimes\sum_i\bm\beta_i\otimes\bm\xi_i\right\|_\infty
       +\E\left\|\sum_i\bm\beta_i^{\otimes2}\otimes\bm\xi_i\right\|_\infty.
\end{align*}

\begin{proposition}[Global Stein bound]\label{prop:global-stein-transfer}
Let $\bm\xi_i$ be independent centered vectors with approximate Stein
kernels $(\mathbf\tau_i,\bm\beta_i)$ and
\[
 \E\|\bm\xi_i\|_\infty^5+
 \E\|\mathbf\tau_i\|_\infty^{5/2}+
 \E\|\bm\beta_i\|_\infty^{5/2}<\infty.
\]
Fix $k$, and let $\mathbf\Sigma\succeq\sigma_*^2\mathbf I_d$ with
$\sigma_*>0$ and $\bar\sigma^2=\max_j\Sigma_{jj}$.
For $0<u\le1/2$,
\begin{equation}
 \begin{aligned}
 &\sup_{A\in\mathcal A_k}\left|\Prob(\bm W\in A)-\int_Ap_W\right|\\
 &\quad\le C_k\bigg\{\frac{|\log u|L_d^2}{\sigma_*^4}a_4
       +\frac{L_d^{5/2}}{\sigma_*^5}a_5
       +\frac{L_d^2}{\sigma_*^4}b_4
       +\frac{L_d}{\sigma_*^2}b_2
       +\frac{L_d^{3/2}}{\sigma_*^3}b_3\\
 &\qquad+\bar\sigma\sqrt u\left(
            \frac{L_d}{\sigma_*}
            +\frac{L_d^2}{\sigma_*^3}\|\mathbf\Delta\|_\infty
            +\frac{L_d^{5/2}}{\sigma_*^4}\|\mathsf K_3\|_\infty
                         \right)\bigg\}.
 \end{aligned}
 \label{eq:rank-stein-bound}
\end{equation}
If $\bm W$ is centered Gaussian, possibly singular, then
\[
 \begin{aligned}
 &\sup_{A\in\mathcal A_k}\left|\Prob(\bm W\in A)-\int_Ap_W\right|\\
 &\quad\le C_k\left\{
 \frac{|\log u|L_d^2}{\sigma_*^4}\|\mathbf\Delta\|_\infty^2
 +\bar\sigma\sqrt u\left(\frac{L_d}{\sigma_*}
              +\frac{L_d^2}{\sigma_*^3}\|\mathbf\Delta\|_\infty\right)\right\},
 \end{aligned}
\]
with $\mathsf K_3=\bm0$ in $p_W$.
\end{proposition}

\begin{proof}[Proof of Proposition~\ref{prop:global-stein-transfer}]
For $h=\ind_{A_k(\bm a)}$, translations, positive dilations and
maximum-norm envelopes stay in the same class:
\[
 A_k(\bm a)-\bm y=A_k(\bm a-\bm y),\qquad
 sA_k(\bm a)=A_k(s\bm a)\quad(s>0),
\]
\[
 \sup_{\|\bm y-\bm x\|_\infty\le b}h(\bm y)
     =\ind_{A_k(\bm a+b\bm1)}(\bm x),\qquad
 \inf_{\|\bm y-\bm x\|_\infty\le b}h(\bm y)
     =\ind_{A_k(\bm a-b\bm1)}(\bm x).
\]
The signed smoothing inequality of \cite[Lemma~4.1]{Koike2026} is
valid for every bounded measurable $h$. Apply it to
\[
 \mu=\mathcal L\{\sqrt{1-u}\bm W\},\qquad
 \nu(B)=\int\ind_B(\sqrt{1-u}\bm z)p_W(\bm z)\,d\bm z,
\]
with Gaussian convolution kernel $N(\bm0,u\mathbf\Sigma)$ and
$b=\bar\sigma\sqrt{2u\log(8d)}$. Positive-dilation closure converts
the smoothed bound back to the supremum in the statement.
\[
 \Prob(\|\sqrt u\bm Z\|_\infty>b)\le\tfrac14.
\]
Lemma~\ref{lem:global-order-strips}, at orders $0,2,3$, bounds every
signed boundary and translated-difference term by
\[
 C_k\bar\sigma\sqrt u\left\{
 \frac{L_d}{\sigma_*}
 +\frac{L_d^2}{\sigma_*^3}\|\mathbf\Delta\|_\infty
 +\frac{L_d^{5/2}}{\sigma_*^4}\|\mathsf K_3\|_\infty\right\}.
\]
For the smoothed expectation, put
\[
 m_v=(1-v)^{-1/2},\qquad
 \mathcal H_r(\mathsf V;v)
   =\E\langle\mathsf V,\nabla^rh_v(\bm W)\rangle,
 \qquad \kappa=L_d/\sigma_*^2.
\]
First establish the expansion for $h\in C_b^\infty(\R^d)$.
The scalar approximate-kernel identity is the one in
\cite[Definition~4.1]{Koike2026}. In a leave-one-out term, freeze the
other summands and apply it to the scalar function
\[
 x\longmapsto\langle\mathsf U,
       \nabla^r h(\sqrt{1-v}\,x+\bm c)\rangle.
\]
For terms whose coefficient depends on the same summand, take first
$f\in C_c^\infty(\mathbb R^d)$ and put
\[
 G_f(x)=\langle x^{\otimes2},\nabla^2f(x)\rangle
                     -2x^\top\nabla f(x)+2f(x).
\]
Its gradient has $j$th coordinate
$\sum_{a,b}x_ax_b\partial_{jab}f(x)$, so the scalar identity gives
\begin{align*}
 2\E\langle\mathbf\tau_i\otimes\bm\xi_i,
                            \nabla^3f(\bm\xi_i)\rangle
 ={}&\E\langle\bm\xi_i^{\otimes3}
             -\bm\beta_i\otimes\bm\xi_i^{\otimes2},
                            \nabla^3f(\bm\xi_i)\rangle\\
 &-\E\langle\mathbf\tau_i\otimes\bm\xi_i^{\otimes2},
                            \nabla^4f(\bm\xi_i)\rangle.
\end{align*}
Smooth cutoffs extend this identity to $f\in C_b^4$ by the assumed
moments. These two operations therefore require only the scalar
identity in the proposition.
Define the following random tensors:
\begin{align*}
 \mathsf V_1&=\overline{\mathbf T}^{\otimes2}
                    -\sum_i\mathbf\tau_i^{\otimes2},&
 \mathsf V_2&=\sum_i\mathbf\tau_i\otimes\bm\xi_i^{\otimes2},\\
 \mathsf V_3&=\sum_i\bm\xi_i^{\otimes3}
                    \otimes(\bm\xi_i+\bm B-\bm\beta_i),&
 \mathsf V_4&=\sum_i\bm\xi_i^{\otimes3}
                    \otimes(\overline{\mathbf T}-\mathbf\tau_i),\\
 \mathsf V_5&=\bm B^{\otimes2}-\sum_i\bm\beta_i^{\otimes2},&
 \mathsf V_6&=\sum_i\bm\beta_i
                    \otimes(\overline{\mathbf T}-\mathbf\tau_i),\\
 \mathsf V_7&=\sum_i\bm\beta_i\otimes\bm\xi_i^{\otimes2},&
 \mathsf V_8&=\sum_i\bm\beta_i\otimes\bm\xi_i
                    \otimes(\overline{\mathbf T}-\mathbf\tau_i),\\
 \mathsf V_9&=\sum_i\bm\beta_i\otimes\bm\xi_i
                    \otimes(\bm\xi_i+\bm B-\bm\beta_i).
\end{align*}
H\"older's inequality, with exponents $(5/3,5/2)$,
$(5/2,5,5/2)$ and $(5/4,5)$ for the products of respective types
$\bm\xi^3\mathbf\tau$, $\bm\beta\bm\xi\mathbf\tau$ and
$\bm\beta^2\bm\xi$, shows that every $\mathsf V_\ell$ is integrable.
For $u\le v<1$, set
\begin{align*}
 H_u(v)&=\int_u^v m_s\,ds
       =2\{\sqrt{1-u}-\sqrt{1-v}\},\\
 J_u(v)&=\int_u^v(s-u)m_s\,ds,
 &K_u(v)&=\frac12H_u(v)^2.
\end{align*}
Define
\[
 \bm W(v)=\sqrt{1-v}\bm W+\sqrt v\bm Z,\qquad
 \mathcal E_r(\mathsf V;v)=\E\langle\mathsf V,\nabla^rh(\bm W(v))\rangle.
\]
Gaussian integration by parts and the scalar identities above give
\[
 \E h_u(\bm W)-\E h(\bm Z)
 =\frac12\int_u^1\{\mathcal E_2(\overline{\mathbf T};s)
                         +m_s\mathcal E_1(\bm B;s)\}\,ds.
\]
Expanding both terms by the same interpolation, as in
\cite[Appendix~B.2, equations~(B11)--(B20)]{Koike2026}, uses
\[
 \mathcal E_r(\mathsf V;v)=m_v^r\mathcal H_r(\mathsf V;v).
\]
The three terms involving $\mathsf V_5,\mathsf V_6,\mathsf V_7$ before
this conversion are, respectively,
\begin{align*}
 &\frac14\int_u^1\int_s^1m_sm_v
                            \mathcal E_2(\mathsf V_5;v)\,dv\,ds,\\
 &\frac14\int_u^1\int_s^1(m_s+m_v)
                            \mathcal E_3(\mathsf V_6;v)\,dv\,ds,\\
 &-\frac18\int_u^1\int_s^1m_v
                            \mathcal E_3(\mathsf V_7;v)\,dv\,ds.
\end{align*}
The Gaussian endpoint terms have coefficients
$(1-u)/2$ and $(1-u)^{3/2}/6$ and give the two corrections in
$\int h_up_W$. Changing the order of integration therefore yields
\[
 \E h_u(\bm W)-\int h_u(\bm z)p_W(\bm z)\,d\bm z
 =\sum_{\ell=1}^9 c_\ell\int_u^1
       w_\ell(v)\mathcal H_{r_\ell}(\mathsf V_\ell;v)\,dv,
\]
where the coefficients and time weights are
\[
\begin{array}{c|c|c|c|c}
 \ell&r_\ell&c_\ell&w_\ell(v)&B_\ell\\ \hline
 1&4&1/4&(v-u)m_v^4&|\log u|\\
 2&4&-1/8&(v-u)m_v^4&|\log u|\\
 3&4&1/16&J_u(v)m_v^5&8/3\\
 4&5&1/16&J_u(v)m_v^5&8/3\\
 5&2&1/4&H_u(v)m_v^3&4\\
 6&3&1/4&(v-u)m_v^4+H_u(v)m_v^3&\pi+4\\
 7&3&-1/8&(v-u)m_v^4&\pi\\
 8&4&1/8&K_u(v)m_v^4&2\\
 9&3&1/8&K_u(v)m_v^4&\pi
\end{array}
\]
and the last column satisfies
\[
 \int_u^1 w_\ell(v)
           \left(\frac{1-v}{v}\right)^{r_\ell/2}\,dv
 \le B_\ell.
\]
Indeed,
\[
 H_u(v)\le2v,\qquad K_u(v)\le2v^2,\qquad
 J_u(v)\le v^2\int_0^1\frac{x}{\sqrt{1-vx}}\,dx
          \le\frac43v^2.
\]
The first two bounds in the table use
$\int_u^1(v-u)v^{-2}\,dv\le|\log u|$; the remaining bounds follow from
\[
 \begin{gathered}
 \int_0^1v^{-1/2}\,dv=\int_0^1(1-v)^{-1/2}\,dv=2,
 \\
 \int_0^1\frac{dv}{\sqrt{v(1-v)}}=\pi,
 \quad
 \int_0^1\sqrt{\frac v{1-v}}\,dv=\frac\pi2.
 \end{gathered}
\]
For $h=\ind_{A_k(\bm a)}$, apply the identity to $h_\delta$ and use
\[
 (h_\delta)_v=h_{\delta+v-\delta v},\qquad
 \frac{1-(\delta+v-\delta v)}{\delta+v-\delta v}
 \le\frac{1-v}{v}.
\]
The derivative bound in Lemma~\ref{lem:global-order-strips}, the displayed
time bounds and tensor integrability justify dominated convergence as
$\delta\downarrow0$.

Lemma~\ref{lem:global-order-strips} now gives
\[
 |\mathcal H_r(\mathsf V;v)|
 \le C_{k,r}\E\|\mathsf V\|_\infty
             \left\{\frac{(1-v)\kappa}{v}\right\}^{r/2}.
\]
The coefficients remain inside their joint expectations.
Coordinatewise Cauchy--Schwarz and $2xy\le x^2+y^2$ imply
\begin{align*}
 2\kappa^2\E\left\|\sum_i\mathbf\tau_i
                         \otimes\bm\xi_i^{\otimes2}\right\|_\infty
 &\le\kappa^2\E\left\|\sum_i\mathbf\tau_i^{\otimes2}\right\|_\infty
       +\kappa^2\E\left\|\sum_i\bm\xi_i^{\otimes4}\right\|_\infty,\\
 2\kappa^{3/2}\E\|\overline{\mathbf T}\otimes\bm B\|_\infty
 &\le\kappa^2\E\|\overline{\mathbf T}\|_\infty^2
       +\kappa\E\|\bm B\|_\infty^2,\\
 2\kappa^{3/2}\E\left\|\sum_i\bm\beta_i
                                  \otimes\mathbf\tau_i\right\|_\infty
 &\le\kappa\E\left\|\sum_i\bm\beta_i^{\otimes2}\right\|_\infty
       +\kappa^2\E\left\|\sum_i\mathbf\tau_i^{\otimes2}\right\|_\infty,\\
 2\kappa^{3/2}\E\left\|\sum_i\bm\beta_i
                         \otimes\bm\xi_i^{\otimes2}\right\|_\infty
 &\le\kappa\E\left\|\sum_i\bm\beta_i^{\otimes2}\right\|_\infty
       +\kappa^2\E\left\|\sum_i\bm\xi_i^{\otimes4}\right\|_\infty.
\end{align*}
Expanding $\mathsf V_3,\mathsf V_4,\mathsf V_8,\mathsf V_9$ and using
$|\log u|\ge\log2$ therefore yields
\[
 \left|\E h_u(\bm W)-\int h_up_W\right|
 \le C_k\left\{
 |\log u|\kappa^2a_4+\kappa^{5/2}a_5
       +\kappa^2b_4+\kappa b_2+\kappa^{3/2}b_3\right\}.
\]
Combining this bound with the signed smoothing inequality proves
\eqref{eq:rank-stein-bound}.

For Gaussian $\bm W$, equation~(4.7) of \cite{Koike2026} gives
\[
 \E h_u(\bm W)-\int h_up_W
 =\frac14\int_u^1(v-u)m_v^4
       \E\langle\mathbf\Delta^{\otimes2},
                      \nabla^4h_v(\bm W)\rangle\,dv.
\]
Gaussian integration by parts is valid for singular covariance matrices.
The first row of the table and the same smoothing bounds consequently
prove the Gaussian assertion.
\end{proof}

\begin{lemma}[A common moment event]\label{lem:joint-data-event}\label{lem:koike-concentration}
Suppose Assumption~\ref{ass:data} holds, $n\ge3$, $d\ge1$, and $r_n\le c$
for a sufficiently small fixed $c>0$. There is an event $\mathcal E_n$
with $\Prob(\mathcal E_n^c)\le Cn^{-3}$ on which
\begin{align*}
 \max_{i,j}|X_{ij}|&\le CL,&
 \|\bar{\bm X}\|_\infty&\le C\sqrt{L/n},\\
 \|\bar{\bm X}^{\,2}-\mathbf\Sigma\|_\infty
 +\|\bar{\bm X}^{\,3}-\mathsf M_3\|_\infty
 &\le C\sqrt{L/n},&
 \max_j n^{-1}\sum_i|X_{ij}|^6&\le C.
\end{align*}
For $\bm b_i=\bm X_i-\bar{\bm X}$, the maximum and sixth-moment
bounds remain valid, and
\[
 \left\|n^{-1}\sum_i\bm b_i^{\otimes2}-\mathbf\Sigma\right\|_\infty
 +\left\|n^{-1}\sum_i\bm b_i^{\otimes3}-\bar{\bm X}^{\,3}\right\|_\infty
 \le C\sqrt{L/n}.
\]
If $w_i$ satisfy the bounded alternative of Assumption~\ref{ass:mult},
then, on $\mathcal E_n$, the variables $\bm Y_i=w_i\bm b_i$ satisfy,
with conditional probability at least $1-Cn^{-3}$,
\begin{align*}
 \|\bar{\bm Y}\|_\infty&\le C\sqrt{L/n},\\
 \|\bar{\bm Y}^{\,2}-\mathbf\Sigma_0\|_\infty
 +\|\bar{\bm Y}^{\,3}-\mathsf M_0\|_\infty&\le C\sqrt{L/n},
\end{align*}
where $\mathbf\Sigma_0=n^{-1}\sum_i\bm b_i^{\otimes2}$ and
$\mathsf M_0=\gamma n^{-1}\sum_i\bm b_i^{\otimes3}$.
Their maximum and empirical sixth moments are bounded as above
for every multiplier realization. Constants depend only on $b,\sigma_*$
and, conditionally, $b_w$.
\end{lemma}

\begin{proof}[Proof of Lemma~\ref{lem:joint-data-event}]
The exponential tail bound and a union bound give
\[
 \Prob(\max_{i,j}|X_{ij}|>CL)\le Cn^{-4}
\]
after increasing $C$. Truncate each coordinate at this threshold.
Products of $q\le6$ coordinates have bounded variance and magnitude
at most $(CL)^q$.
Bernstein's inequality and a union bound give the deviation
\[
 C\{\sqrt{L/n}+L^{q+1}/n\}
\]
with failure probability $Cn^{-3}$; the truncated tail expectations
are at most $Cn^{-4}$. Since $n\ge c^{-1}L^{12}$,
\[
 L^{q+1}/n\le C\sqrt{L/n}\quad(q\le3),
 \qquad L^7/n\le C.
\]
The centering identities and H\"older's inequality give the assertions
for $\bm b_i$. For example,
\begin{align*}
 n^{-1}\sum_i\bm b_i^{\otimes2}
 &=\bar{\bm X}^{\,2}-\bar{\bm X}^{\otimes2},\\
 n^{-1}\sum_i\bm b_i^{\otimes3}
 &=\bar{\bm X}^{\,3}
   -\operatorname{sym}(\bar{\bm X}\otimes\bar{\bm X}^{\,2})
   +2\bar{\bm X}^{\otimes3},
\end{align*}
where $\operatorname{sym}$ sums the three placements of the vector.

Conditionally, for each product $V_i$ of $q\le3$ coordinates of
$w_i\bm b_i$, the empirical sixth-moment bound gives
\[
 n^{-1}\sum_i\Var^*(V_i)\le C,\qquad
 \max_i|V_i|\le(CL)^q.
\]
Apply the same Bernstein calculation. Bounded weights give the remaining
assertions.
\end{proof}

\subsection{The projected local input}
\begin{proposition}[Projected absolute expansions]\label{prop:projected-local}
Suppose Assumption~\ref{ass:data} holds and $L_n^3\le n$. For every nonempty $I\subset[d]$,
\begin{equation*}
 \sup_{t\in\R}\left|
 \Prob\{\bm S_{n,I}\in(t,\infty)^{|I|}\}
       -\int_{(t,\infty)^{|I|}}p_{n,I}\right|\le C\epsn^2.
\end{equation*}
Under Assumption~\ref{ass:mult}, with probability at least $1-C/n$,
the same simultaneous bound holds for the conditional probability and
the raw-moment density $\hat p_{n,\gamma,I}$:
\begin{equation*}
 \sup_{\varnothing\ne I\subset[d]}\sup_{t\in\R}
 \left|\Prob^*\{\bm S_{n,I}^*\in(t,\infty)^{|I|}\}
       -\int_{(t,\infty)^{|I|}}\hat p_{n,\gamma,I}\right|
 \le C\epsn^2.
\end{equation*}
\end{proposition}
\begin{proof}[Proof of Proposition~\ref{prop:projected-local}]
The identity in Assumption~\ref{ass:data}, applied to coordinate tests,
gives $\E\tau_{i,j\ell}=\E(X_{ij}X_{i\ell})$. Hence
\begin{align*}
 |\E\tau_{i,j\ell}|
 &\le (\E X_{ij}^2\E X_{i\ell}^2)^{1/2}\le4b^2,\\
 \|\tau_{i,j\ell}\|_{\psi_{1/2}}
 &\le C\|\tau_{i,j\ell}-\E\tau_{i,j\ell}\|_{\psi_1}
                         +C|\E\tau_{i,j\ell}|
 \le Cb^2.
\end{align*}
Taking gradient tests yields the scalar Stein identity of
\cite[Definition~2.1]{Koike2026}. Thus its conditions~(2.2)--(2.3)
hold after a fixed enlargement of $b$; the reference covariance satisfies
$\lambda_{\min}(\mathbf\Sigma)\ge\sigma_*^2$, and
$\log^3d\le L_n^3\le n$.
Assumption~\ref{ass:mult} supplies exactly the Gaussian or bounded
Stein-kernel multiplier alternatives in Theorem~2.2 of \cite{Koike2026}.

Reflect all coordinates and regard a projected upper orthant as a rectangle
in $\mathbb R^d$, with unrestricted coordinates outside $I$; strict endpoints
follow by limits.
Theorems~2.1--2.2 and equation~(2.7) of \cite{Koike2026} apply
simultaneously over this class. The bootstrap theorem is already stated
with the raw-moment density $\hat p_{n,\gamma}$.
Marginalizing the Gaussian derivatives gives $p_{n,I}$ and
$\hat p_{n,\gamma,I}$, respectively, and proves both assertions for $d\ge3$.

If $d<3$, augment every $\bm X_i$ by $3-d$ independent
$N(\bm0,\sigma_*^2\mathbf I_{3-d})$ coordinates, independently across
$i$ and of the original data. The augmented covariance and Stein kernel are
\[
 \operatorname{diag}(\mathbf\Sigma,\sigma_*^2\mathbf I_{3-d}),
 \qquad
 \operatorname{diag}\{\bm\tau_i(\bm X_i),
                       \sigma_*^2\mathbf I_{3-d}\}.
\]
They satisfy the same lower eigenvalue bound and the required moment
bounds with a fixed enlargement of $b$, since $\sigma_*\le Cb$.

The condition $L_n^3\le n$ implies $\log^3 3\le n$, and
$\log(3n)\le C\log(dn)$. Apply the three-dimensional results with the added coordinates
unrestricted, and then marginalize.

The resulting bootstrap error depends
only on the original data, so its failure probability under their marginal
law is bounded by the failure probability in the augmented experiment.
\end{proof}
The errors in Proposition~\ref{prop:projected-local} are absolute.
The distributional expansion below uses the class $\mathcal A_k$ directly;
finite projected aggregation is recorded separately with its full remainder.

\subsection{Gaussian factorial moments, aggregation, and regularity}

\begin{lemma}[Gaussian threshold scale]\label{lem:t-order}
Assume Assumptions~\ref{ass:data} and \ref{ass:variance}.
For fixed $k,\epsilon$, there are $0<c_1<C_1<\infty$ such that
\[
 c_1\log d\le t^2\le C_1\log d
 \qquad(t\in\Tcal_{k,\epsilon})
\]
for all sufficiently large $d$.
\end{lemma}
\begin{proof}[Proof of Lemma~\ref{lem:t-order}]
Write $\bm Z=\sigma_*\bm G+\bm W$, where $\bm G$ is standard Gaussian
and independent of the centered Gaussian vector $\bm W$.
Partition $km$ coordinates into $k$ groups of size $m=\lfloor d/k\rfloor$.
For the maximum $M_r$ within group $r$, conditional Jensen gives
\[
 \E M_r\ge\sigma_*\E\max_{j\le m}G_j\ge c\sqrt{\log d}.
\]
The last inequality follows from Mills' bounds for the independent maximum.
Each $M_r$ is $\overline\sigma$-Lipschitz as a function of a standard
Gaussian vector. Gaussian concentration, a consequence of Gaussian
isoperimetry, gives
\[
 \Prob\{M_r\le\tfrac12c\sqrt{\log d}\}
 \le \exp\{-c^2\log d/(8\overline\sigma^2)\}.
\]
If every group exceeds the displayed threshold, at least $k$ coordinates
exceed it. A union bound therefore proves the lower quantile bound.
For a sufficiently large fixed $C$,
\[
 \Prob\{T_{\bm Z,[k]}>C\sqrt{\log d}\}
 \le d\,\bar\Phi(C\sqrt{\log d}/\overline\sigma)\longrightarrow0,
\]
which proves the upper bound.
\end{proof}

\begin{lemma}[Gaussian orthant comparison through a finite order]
\label{lem:finite-gauss-factorial}\label{lem:gaussian-factorial}
Let $\bm Z$ be a centered Gaussian vector with marginal standard deviations $\sigma_j\ge\sigma_->0$, and set
\begin{gather*}
r:=\max_{i\ne j}|\Corr(Z_i,Z_j)|,
\qquad p_{\max}(t):=\max_jp_j(t),\\
p_j(t):=\bar\Phi(t/\sigma_j),\qquad
\lambda(t):=\sum_{j=1}^dp_j(t).
\end{gather*}
Fix integers $1\le k\le K<d$, write $L=K+1$, and let $t\ge0$. Suppose
\[
(L-1)r\le\tfrac12,
\qquad L^2\eta(t)\le B,
\qquad \eta(t):=r\{1+t^2/\sigma_-^2\},
\]
where $B<\infty$ is fixed. There is a numerical constant $C$ such that, for every $I\subset[d]$ with $1\le s=|I|\le L$,
\[
e^{-Cs^2\eta(t)}\prod_{j\in I}p_j(t)
\le\pi_I(t)\le
e^{Cs^2\eta(t)}\prod_{j\in I}p_j(t).
\]
Writing $M_{Z,s}(t)=\sum_{|I|=s}\pi_I(t)$, we have
\begin{align*}
M_{Z,s}(t)&\le e^{CB}\frac{\lambda(t)^s}{s!},\\
\left|M_{Z,s}(t)-\frac{\lambda(t)^s}{s!}\right|
&\le C_Bs^2\eta(t)\frac{\lambda(t)^s}{s!}
 +\binom s2\frac{p_{\max}(t)\lambda(t)^{s-1}}{s!}.
\end{align*}
If also $\lambda(t)\le\Lambda$, then
\begin{align*}
\sum_{s=k}^{K}\binom{s-1}{k-1}M_{Z,s}(t)&\le C_{k,\Lambda,B},\\
|G_k(t)-h_k(\lambda(t))|
&\le C_{k,\Lambda,B}\{\eta(t)+p_{\max}(t)\}\\
&\quad +(1+e^{CB})\binom K{k-1}\frac{\Lambda^{K+1}}{(K+1)!},
\end{align*}
where $h_k(x)=e^{-x}\sum_{j=0}^{k-1}x^j/j!$.
\end{lemma}

\begin{proof}[Proof of Lemma~\ref{lem:finite-gauss-factorial}]
For fixed $I$, standardize the coordinates and let $\mathbf R_I$ be their correlation matrix. Put $\theta=(s-1)r$ and $u_j=t/\sigma_j$. Gershgorin's inequality gives
\[
(1-\theta)\mathbf I_s\preceq\mathbf R_I\preceq(1+\theta)\mathbf I_s,
\qquad 0\le\theta\le\tfrac12.
\]
The Gaussian density and determinant bounds imply
\begin{align*}
\pi_I(t)
&\le\left(\frac{1+\theta}{1-\theta}\right)^{s/2}
 \prod_{j\in I}\bar\Phi\left(\frac{u_j}{\sqrt{1+\theta}}\right),\\
\pi_I(t)
&\ge\left(\frac{1-\theta}{1+\theta}\right)^{s/2}
 \prod_{j\in I}\bar\Phi\left(\frac{u_j}{\sqrt{1-\theta}}\right).
\end{align*}
Since $\phi(u)/\bar\Phi(u)\le C(1+u)$ for $u\ge0$, integration of the logarithmic derivative yields
\[
\left|\log\frac{\bar\Phi(u/\sqrt{1\pm\theta})}{\bar\Phi(u)}\right|
\le C\theta(1+u^2),
\quad u\ge0,\quad 0\le\theta\le\tfrac12.
\]
Also $|\log\{(1+\theta)/(1-\theta)\}|\le C\theta$. Multiplication over coordinates proves the orthant comparison.

Let $e_s(t)=\sum_{|I|=s}\prod_{j\in I}p_j(t)$. The comparison gives
\[
|M_{Z,s}(t)-e_s(t)|
\le\{e^{Cs^2\eta(t)}-1\}e_s(t)
\le C_Bs^2\eta(t)\frac{\lambda(t)^s}{s!}.
\]
Draw $J_1,\ldots,J_s$ independently with $\Prob(J_a=j)=p_j(t)/\lambda(t)$. Their collision probability satisfies
\[
0\le1-\frac{s!e_s(t)}{\lambda(t)^s}
\le\sum_{1\le a<b\le s}\Prob(J_a=J_b)
=\binom s2\frac{\sum_jp_j(t)^2}{\lambda(t)^2}
\le\binom s2\frac{p_{\max}(t)}{\lambda(t)}.
\]
The factorial-moment bounds follow. The weighted sums are finite because
\[
\sum_{s=k}^{\infty}\binom{s-1}{k-1}s^2\frac{\Lambda^s}{s!}<\infty,
\qquad
\sum_{s=\max(k,2)}^{\infty}\binom{s-1}{k-1}\binom s2\frac{\Lambda^{s-1}}{s!}<\infty.
\]
Apply Lemma~\ref{lem:finite-ie-corrected} to $N_Z(t)$ and to a Poisson variable with mean $\lambda(t)$. Their respective remainders are bounded by
\[
\binom K{k-1}M_{Z,K+1}(t),
\qquad \binom K{k-1}\frac{\lambda(t)^{K+1}}{(K+1)!}.
\]
Comparison of their finite sums proves the final assertion.
\end{proof}

\begin{lemma}[Central Gaussian density]\label{lem:rank-density-lower}\label{lem:weighted-regularity}
Let $1\le k\le d$ be fixed, let
$\mathbf\Sigma\succeq\sigma_*^2\mathbf I_d$ with $\sigma_*>0$, and set
$\bar\sigma^2=\max_j\Sigma_{jj}<\infty$.
For every $p\in(0,1)$,
\[
 f_k(c^G_{p,k})\ge
 \frac{\phi\{\Phi^{-1}(p)\}}{\bar\sigma}.
\]
For fixed $\epsilon\in(0,1/4)$ and
$p\in[\epsilon/4,1-\epsilon/4]$, this gives
\[
 f_k(c^G_{p,k})\ge c_\epsilon>0,
 \qquad \sup_t|f_k'(t)|\le CL_d,
 \qquad \sup_{\epsilon/4\le\alpha\le1-\epsilon/4}
       \|\mathbf A_\alpha\|_1\le CL_d,
\]
where the constants depend only on $k,\epsilon,\sigma_*,\bar\sigma$.

In particular,
\begin{equation*}
 \sup_{\epsilon/4\le p\le1-\epsilon/4}
       |(G_k^{-1})''(p)|\le CL_d.
\end{equation*}
\end{lemma}

\begin{proof}[Proof of Lemma~\ref{lem:rank-density-lower}]
Set $g(\bm u)=\{\mathbf\Sigma^{1/2}\bm u\}_{[k]}$.
The order statistic is one-Lipschitz in the maximum norm, so
\[
 |g(\bm u)-g(\bm v)|
 \le\|\mathbf\Sigma^{1/2}(\bm u-\bm v)\|_\infty
 \le\bar\sigma\|\bm u-\bm v\|_2.
\]
For $B_t=\{\bm u:g(\bm u)\le t\}$, its Euclidean enlargement satisfies
$B_t^h\subset B_{t+\bar\sigma h}$.
Gaussian isoperimetry, also used in \cite[Remark~5.1]{Koike2026}, yields
\[
 G_k(t+\bar\sigma h)
 \ge\Phi\{\Phi^{-1}(G_k(t))+h\},\qquad h>0.
\]
Divide by $h$ and let $h\downarrow0$.
The global Gaussian derivative lemma gives
\[
 |f_k'(t)|\le\|\mathsf J_{2,k}(t)\|_1\le CL_d,
 \qquad
 \|\mathbf A_\alpha\|_1
 \le\frac{\|\mathsf J_{2,k}(c_\alpha)\|_1}{2c_\epsilon}
 \le CL_d.
\]

The inverse identity
$(G_k^{-1})''=-f_k'(G_k^{-1})/f_k(G_k^{-1})^3$
gives the last assertion.
\end{proof}

\subsection{Factorial-moment and distribution expansions}

\begin{theorem}[Finite transfer of projected expansions]\label{thm:finite-transfer}\label{thm:FM}
Let $\bm Y\in\mathbb R^d$ and let $\bm Z$ be a Gaussian vector. Fix integers $1\le k\le K<d$, a threshold set $\mathcal T\subset\mathbb R$, and signed projected densities $q_I$ for every $I\subset[d]$ with $k\le |I|\le K+1$. Define
\[
 \begin{aligned}
 V_{Y,s}(t)&=\sum_{|I|=s}\Prob(Y_j>t\text{ for all }j\in I),\\
 V_{Z,s}(t)&=\sum_{|I|=s}\pi_I(t),\\
 M_s(t)&=\sum_{|I|=s}\int_{(t,\infty)^s}q_I(\bm u)\,d\bm u,
 \end{aligned}
\]
where $\pi_I(t)=\Prob(Z_j>t\text{ for all }j\in I)$, and put $a_{s,k}=\binom{s-1}{k-1}$. Suppose that nonnegative deterministic functions $r_s(t)$ and $h_s(t)$ satisfy, for all $t\in\mathcal T$ and $|I|=s\in\{k,\ldots,K+1\}$,
\begin{align*}
 \left|\Prob(Y_j>t\text{ for all }j\in I)-\int_{(t,\infty)^s}q_I(\bm u)\,d\bm u\right|&\le r_s(t)\pi_I(t),\\
 \left|\int_{(t,\infty)^s}q_I(\bm u)\,d\bm u-\pi_I(t)\right|&\le h_s(t)\pi_I(t).
\end{align*}
Then, with $F_{Y,k}(t)=\Prob(Y_{[k]}\le t)$, $G_k(t)=\Prob(Z_{[k]}\le t)$ and
\[
 Q_{K,k}(t)=-\sum_{s=k}^{K}(-1)^{s-k}a_{s,k}\{M_s(t)-V_{Z,s}(t)\},
\]
one has the finite bound
\begin{align*}
 &|F_{Y,k}(t)-G_k(t)-Q_{K,k}(t)|\\
 &\quad\le \sum_{s=k}^{K}a_{s,k}r_s(t)V_{Z,s}(t)\\
 &\qquad+a_{K+1,k}\{2+h_{K+1}(t)+r_{K+1}(t)\}V_{Z,K+1}(t),\qquad t\in\mathcal T.
\end{align*}
The same statement holds conditionally, with every quantity and inequality interpreted under a fixed conditional law on the specified event.
\end{theorem}

\begin{proof}[Proof of Theorem~\ref{thm:finite-transfer}]
The finite Bonferroni bound for the exceedance count gives
\[
 \left|F_{Y,k}(t)-\left\{1-\sum_{s=k}^{K}(-1)^{s-k}a_{s,k}V_{Y,s}(t)\right\}\right|
 \le a_{K+1,k}V_{Y,K+1}(t).
\]
Apply the same inequality to $\bm Z$, subtract the two expansions, and insert $M_s(t)$ to obtain
\begin{align*}
 |F_{Y,k}(t)-G_k(t)-Q_{K,k}(t)|
 &\le\sum_{s=k}^{K}a_{s,k}|V_{Y,s}(t)-M_s(t)|\\
 &\quad+a_{K+1,k}\{V_{Y,K+1}(t)+V_{Z,K+1}(t)\}.
\end{align*}
The stated projected bounds imply
\[
 \begin{aligned}
 |V_{Y,s}(t)-M_s(t)|&\le r_s(t)V_{Z,s}(t),\\
 V_{Y,K+1}(t)&\le\{1+h_{K+1}(t)+r_{K+1}(t)\}V_{Z,K+1}(t).
 \end{aligned}
\]
Substitution proves the result. Every step is deterministic after conditioning, which proves the conditional assertion.
\end{proof}

\begin{corollary}[Polynomial projected errors]\label{cor:polynomial-transfer}
Assume the hypotheses of Theorem~\ref{thm:finite-transfer}. Suppose also that, uniformly over $t\in\mathcal T$ and $k\le s\le K+1$,
\[
 V_{Z,s}(t)\le C_0\frac{\Lambda^s}{s!},\qquad
 r_s(t)\le r(1+s)^q,\qquad h_{K+1}(t)\le H,
\]
where $C_0,\Lambda,q,H$ are fixed nonnegative constants, and $k$ is fixed. Then
\[
 \begin{aligned}
 &\sup_{t\in\mathcal T}|F_{Y,k}(t)-G_k(t)-Q_{K,k}(t)|\\
 &\quad\le C_{k,q,\Lambda,C_0}r
 +C_0\binom K{k-1}\{2+H+r(K+2)^q\}\frac{\Lambda^{K+1}}{(K+1)!}.
 \end{aligned}
\]
\end{corollary}

\begin{proof}[Proof of Corollary~\ref{cor:polynomial-transfer}]
Since $k$ and $q$ are fixed,
\[
 \sum_{s=k}^{\infty}\binom{s-1}{k-1}(1+s)^q\frac{\Lambda^s}{s!}<\infty.
\]
Insert this bound and the bound at $s=K+1$ into Theorem~\ref{thm:finite-transfer}.
\end{proof}

\begin{theorem}[Distribution expansion]
\label{thm:distribution}\label{thm:global-order-edgeworth}
Suppose Assumption~\ref{ass:data} holds and $L_n^3\le n$. For fixed $k$,
\begin{equation*}
 \sup_{t\in\R}
 |\Prob(T_{n,[k]}\le t)-G_k(t)-Q_{n,k}(t)|\le C\epsn^2.
\end{equation*}
If Assumption~\ref{ass:mult} also holds, then with probability at least $1-C/n$,
\begin{equation*}
 \sup_{t\in\R}
 |\hat F_{n,k}(t)-G_k(t)-\hat Q_{n,\gamma,k}(t)|\le C\epsn^2.
\end{equation*}
More precisely, both bounds are at most $CL_n^3\log(en)/n$.
\end{theorem}

\begin{proof}[Proof of Theorem~\ref{thm:global-order-edgeworth}]
Use the symmetric scalar kernels $\mathbf T_i$ defined at the start of this appendix.
For the original sum use
\[
 \bm\xi_i=\bm X_i/\sqrt n,\qquad
 \mathbf\tau_i(\bm\xi_i)=\mathbf T_i(\bm X_i)/n,\qquad
 \bm\beta_i=\bm0.
\]
The centered kernel condition and the second-moment identity imply
\[
 \E\mathbf T_i=\E\bm X_i^{\otimes2},\qquad
 \max_{i,j,\ell}\|T_{i,j\ell}(\bm X_i)\|_{\psi_{1/2}}\le Cb^2.
\]
The finite-coordinate maximal inequalities underlying the cited tensor
estimates remain valid when $d<3$: append zero entries to each coefficient
array until it has at least three entries. Thus
$\log(3\vee d^r)\le C_rL_d$ for each fixed tensor order $r$.

Thus the hypotheses for the tensor estimates in
\cite[Theorem~2.1 and its proof, Section~4.2]{Koike2026} hold.
Take $u=n^{-3}$ in Proposition~\ref{prop:global-stein-transfer}.
The estimates in \cite[equations~(4.14)--(4.19)]{Koike2026} give
\[
 \begin{aligned}
 a_4&\le C\{L_n/n+L_n^4/n^2\}\le CL_n/n,\\
 a_5&\le C\{\sqrt{L_n}/n+L_n^5/n^{5/2}\}
      \le C\sqrt{L_n}/n.
 \end{aligned}
\]
Here $\mathbf\Delta=\mathbf0$ and $\|\mathsf K_3\|_\infty\le C/\sqrt n$.
Inserting these bounds into \eqref{eq:rank-stein-bound} gives
\[
 \sup_{A\in\mathcal A_k}
 \left|\Prob(\bm S_n\in A)-\int_Ap_n\right|
 \le C\frac{L_n^3\log(en)}n,
 \qquad \log^3d\le L_n^3\le n.
\]
The tensor estimates depend only on the independent array, whereas the
set class enters through the Gaussian bounds in
Proposition~\ref{prop:global-stein-transfer}.

For bounded multipliers put $\bm b_i=\bm X_i-\bar{\bm X}$.
The conditional kernel and moments are
\[
 \bm\xi_i=w_i\bm b_i/\sqrt n,\qquad
 \mathbf\tau_i(\bm\xi_i)=n^{-1}\E^*\{
           \tau^w(w_i)\bm b_i^{\otimes2}\mid\bm\xi_i\},\qquad
 \bm\beta_i=\bm0,
\]
\[
 \operatorname{Cov}^*(\bm S_n^*)=\bar{\bm X}^{\,2}-\bar{\bm X}^{\otimes2},
 \qquad
 \E^*(\bm S_n^*)^{\otimes3}
 =\frac\gamma{\sqrt n}\left(n^{-1}\sum_i\bm b_i^{\otimes3}\right).
\]
The conditional estimates in
\cite[equations~(4.32)--(4.40)]{Koike2026} give, on one event
$\mathcal F_n$ with $\Prob(\mathcal F_n^c)\le C/n$,
\[
 \begin{gathered}
 a_4\le CL_n/n,\qquad a_5\le C\sqrt{L_n}/n,\\
 \|\mathbf\Delta\|_\infty\le C\sqrt{L_n/n},\qquad
 \|\mathsf K_3\|_\infty\le C/\sqrt n.
 \end{gathered}
\]
Use $u=n^{-3}$ again to obtain the same centered-moment remainder.
Intersect this with the event
\[
 \|\bar{\bm X}\|_\infty\le C\sqrt{L_n/n},\qquad
 \|\bar{\bm X}^{\,2}\|_\infty\le C,
\]
whose failure probability is $C/n$ by Bernstein's inequality and
truncation at $CbL_n$. Expanding the centered third tensor gives
\[
 \left\|n^{-1}\sum_i\bm b_i^{\otimes3}-\bar{\bm X}^{\,3}\right\|_\infty
 \le3\|\bar{\bm X}\|_\infty\|\bar{\bm X}^{\,2}\|_\infty
       +2\|\bar{\bm X}\|_\infty^3
 \le C\sqrt{L_n/n}.
\]
Consequently, the change in the integrated signed density is at most
\[
 CL_d\|\bar{\bm X}\|_\infty^2
 +\frac{CL_d^{3/2}}{\sqrt n}
    \left\|n^{-1}\sum_i\bm b_i^{\otimes3}-\bar{\bm X}^{\,3}\right\|_\infty
 \le C L_n^2/n.
\]
For Gaussian multipliers the conditional sum is Gaussian; apply the
Gaussian branch of Proposition~\ref{prop:global-stein-transfer} and the
same covariance and centering estimates. Both cases yield the assertion.
\end{proof}

\begin{lemma}[Uniform marginal bounds on the moment event]
\label{lem:joint-marginal}
Suppose Assumptions~\ref{ass:data} and \ref{ass:mult} hold,
$k$ is fixed and $r_n\le c$. Then
\[
 \sup_t|\Prob(T_{n,[k]}\le t)-G_k(t)-Q_{n,k}(t)|\le CL^7/n.
\]
On the event in Lemma~\ref{lem:joint-data-event},
\[
 \sup_t|\hat F_{n,k}(t)-G_k(t)-\hat Q_{n,\gamma,k}(t)|\le CL^7/n.
\]
More generally, let a deterministic array $\bm a_i$ obey
\begin{gather*}
 \max_{i,j}|a_{ij}|\le C_0L,\quad
 \|\bar{\bm a}\|_\infty\le C_0\sqrt{L/n},\quad
 \max_jn^{-1}\sum_i|a_{ij}|^6\le C_0,\\
 \|n^{-1}\sum_i\bm a_i^{\otimes2}-\mathbf\Sigma\|_\infty
       \le C_0\sqrt{L/n}.
\end{gather*}
The centered multiplier sum based on this array has the same bound,
with its raw second and third moments replacing those of $\bm X_i$.
No lower eigenvalue condition is imposed on the array's sample covariance.
Constants depend only on $k,b,\sigma_*,b_w,C_0$.
\end{lemma}

\begin{proof}[Proof of Lemma~\ref{lem:joint-marginal}]
Use the global rank-set Stein bound with reference $\mathbf\Sigma$
and smoothing parameter $n^{-2}$.
Take the symmetric scalar kernels $\mathbf T_i$ defined at the start of this appendix.
For $\bm\xi_i=\bm X_i/\sqrt n$, take
$\mathbf\tau_i=\mathbf T_i(\bm X_i)/n$. Bernstein and H\"older give
\begin{gather*}
 \|\overline{\mathbf T}\|_{L^2(\max)}\le CL/\sqrt n,
 \qquad
 \E\left\|\sum_i\mathbf\tau_i^{\otimes2}\right\|_\infty
 +\E\left\|\sum_i\bm\xi_i^{\otimes4}\right\|_\infty\le CL^4/n,\\
 \left\|\sum_i\bm\xi_i^{\otimes3}\right\|_{L^2(\max)}
       \le CL^3/\sqrt n.
\end{gather*}
Here $\overline{\mathbf T}=\sum_i\mathbf\tau_i-\mathbf\Sigma$.
Inserting these bounds into \eqref{eq:rank-stein-bound} gives $CL^7/n$.

For the array, set $\bm b_i=\bm a_i-\bar{\bm a}$.
Both allowed multiplier laws have the exact conditional kernel
\[
 \mathbf\tau_i(\bm\xi_i)
 =n^{-1}\E^*\{\tau^w(w_i)\bm b_i\bm b_i^\top\mid\bm\xi_i\},
 \qquad \bm\xi_i=w_i\bm b_i/\sqrt n.
\]
The covariance drift relative to $\mathbf\Sigma$ is retained.
The empirical sixth-moment bound and conditional Bernstein inequality give
\[
 \|\overline{\mathbf T}\|_{L^2(\max)}
 \le C\{\sqrt{L/n}+L^3/n\}\le CL/\sqrt n;
\]
the preceding fourth- and third-tensor bounds remain valid.
For Gaussian multipliers, $\tau^w=1$ and the kernel is deterministic;
the same bounds follow from the fixed moments of $w_i$.

Replacing centered by raw moments changes the signed integrals by at most
\[
 CL\|\bar{\bm a}\|_\infty^2
 +\frac{CL^{3/2}}{\sqrt n}
   \left\|n^{-1}\sum_i\bm b_i^{\otimes3}
                   -n^{-1}\sum_i\bm a_i^{\otimes3}\right\|_\infty
 \le CL^3/n.
\]
This proves the deterministic-array assertion and its data specialization.
\end{proof}

\subsection{Bootstrap centering and Cornish--Fisher inversion}\label{sec:CF}
\begin{lemma}[Bootstrap centering]\label{lem:centering}
Under Assumption~\ref{ass:data}, the raw-moment correction satisfies
\begin{equation}
 \E\{\hat Q_{n,\gamma,k}(t)\}=\gamma Q_{n,k}(t),\qquad t\in\R.
 \label{eq:centering-main}
\end{equation}
For the centered projected moments,
\begin{align*}
 \E\bar{\bm b}_I^{\,2}-\mathbf\Sigma_{II}
      &=-n^{-1}\mathbf\Sigma_{II},\\
 \E\bar{\bm b}_I^{\,3}
      &=(1-3/n+2/n^2)\E\bar{\bm X}_I^{\,3}.
\end{align*}
\end{lemma}
\begin{proof}[Proof of Lemma~\ref{lem:centering}]
The raw second-moment fluctuation has expectation zero, and the third
tensor in the raw signed density has expectation $\mathsf M_3$.
Taking expectations under the integral proves \eqref{eq:centering-main}.
Independence and centering give
\[
 \E\bar{\bm X}^{\otimes2}=\mathbf\Sigma/n,\qquad
 \E\bar{\bm X}^{\otimes3}=\mathsf M_3/n^2.
\]
Expansion of the centered tensors now gives
\[
 \E\bar{\bm b}^{\,2}=(1-1/n)\mathbf\Sigma,\qquad
 \E\bar{\bm b}^{\,3}=\mathsf M_3-3\mathsf M_3/n+2\mathsf M_3/n^2.
\]
Projection proves the two displayed identities.
\end{proof}

\begin{proposition}[Quantile inversion]
\label{prop:quantile-remainder}
Let $p\in(0,1)$, $G(c)=p$, and $J=[c-h,c+h]$. Suppose
$G\in C^3(J)$, $Q\in C^2(J)$, and
\[
 f=G'(c)>0,\qquad m=\inf_{t\in J}\{G'(t)+Q'(t)\}>0.
\]
Let $H$ be a distribution function satisfying
$\sup_{t\in J}|H(t)-G(t)-Q(t)|\le r$ and
\[
 G(c-h)+Q(c-h)<p-r,\qquad p+r<G(c+h)+Q(c+h).
\]
Define
\[
 a=-\frac{Q(c)}f,\qquad
 b=-\frac{G''(c)Q(c)^2}{2f^3}+\frac{Q'(c)Q(c)}{f^2},
 \qquad t_0=c+a+b,
\]
and assume $t_0\in J$. With $M_3=\sup_J|G'''|$ and $M_Q=\sup_J|Q''|$,
\begin{align*}
 |H^{-1}(p)-t_0|
 &\le\frac1m\biggl\{r+|G''(c)|\left(|ab|+\frac{b^2}2\right)
   +|Q'(c)b|\\
 &\hspace{36mm}+\frac{M_3}{6}|a+b|^3+\frac{M_Q}{2}|a+b|^2\biggr\}.
\end{align*}
\end{proposition}
\begin{proof}[Proof of Proposition~\ref{prop:quantile-remainder}]
Write $S=G+Q$. Since $S'\ge m$ on $J$, the points
$q_-=S^{-1}(p-r)$ and $q_+=S^{-1}(p+r)$ exist in $J$. The quantile
inequalities give $q_-\le H^{-1}(p)\le q_+$. Therefore
\[
 |H^{-1}(p)-t_0|\le\frac{r+|S(t_0)-p|}{m}.
\]
Taylor's formula and the defining cancellations yield
\[
 fa+Q(c)=0,\qquad fb+\frac12G''(c)a^2+Q'(c)a=0,
\]
\begin{align*}
 S(t_0)-p
 &=G''(c)ab+\frac12G''(c)b^2+Q'(c)b+R,\\
 |R|&\le\frac{M_3}{6}|a+b|^3+\frac{M_Q}{2}|a+b|^2.
\end{align*}
Substitution proves the inequality.
\end{proof}

\begin{theorem}[Cornish--Fisher expansion]
\label{thm:CF}\label{prop:quantile-joint}
Suppose Assumptions~\ref{ass:data} and \ref{ass:mult} hold,
$k$ and $\epsilon\in(0,1/4)$ are fixed, and $r_n\le c$
for a sufficiently small constant.
With probability at least $1-Cn^{-3}$,
\begin{equation}
 \sup_{\epsilon\le\alpha\le1-\epsilon}
 \left|\hat c_{1-\alpha,k}-
 \left[c_\alpha-\frac{\hat Q_{n,\gamma,k}(c_\alpha)}{f_\alpha}
                       -C_{n,k}(\alpha)\right]\right|
 \le CL_n^7/n.
 \label{eq:CF-main-corrected}
\end{equation}
On the same event,
\[
 \sup_{\epsilon\le\alpha\le1-\epsilon}
 \left|\hat c_{1-\alpha,k}-c_\alpha+
       \frac{B_{n,k}(c_\alpha)+\gamma Q_{n,k}(c_\alpha)}{f_\alpha}\right|
 \le CL_n^7/n.
\]
\end{theorem}

\begin{proof}[Proof of Theorem~\ref{thm:CF}]
Work on $\mathcal E_n$ and use the central probability interval
$[\epsilon/4,1-\epsilon/4]$.
The global derivatives and Lemma~\ref{lem:joint-data-event} imply
\[
 \|\hat Q_{n,\gamma,k}\|_\infty\le CL^{3/2}/\sqrt n,
 \qquad
 \|\hat Q_{n,\gamma,k}'\|_\infty\le CL^2/\sqrt n.
\]
Set $S=G_k+\hat Q_{n,\gamma,k}$ and $p=1-\alpha$.
The quantile inequalities and continuity of $S$ give
\[
 |S(\hat c_{p,k})-p|\le CL^7/n,
 \qquad
 |\hat c_{p,k}-c^G_{p,k}|\le CL^{3/2}/\sqrt n.
\]
The second bound follows by deterministic quantile brackets and
Lemma~\ref{lem:rank-density-lower}. Taylor expansion now gives
\[
 \hat c_{p,k}=c^G_{p,k}
 -\frac{\hat Q_{n,\gamma,k}(c^G_{p,k})}{f_k(c^G_{p,k})}
 +O(L^7/n),
\]
since the quadratic error is at most
$CL\{L^{3/2}/\sqrt n\}^2+CL^2L^{3/2}/n\le CL^4/n$.

For the quadratic term, the derivative lemma also gives
\[
 \|\hat Q_{n,\gamma,k}''\|_\infty\le CL^{5/2}/\sqrt n,\qquad
 \|G_k^{(3)}\|_\infty\le CL^{3/2},\qquad |C_{n,k}(\alpha)|\le CL^4/n.
\]
Put $a=-\hat Q_{n,\gamma,k}(c_\alpha)/f_\alpha$,
$b=-C_{n,k}(\alpha)$ and
$J=[c_\alpha-h,c_\alpha+h]$, where $h=c_1/L$ and $c_1>0$ is sufficiently small. The density and derivative bounds imply
\[
 \inf_J(G_k'+\hat Q_{n,\gamma,k}')\ge c_\epsilon/4,\qquad
 |a|+|b|+CL^7/n\le h/2
\]
when the constant in $r_n\le c$ is sufficiently small.
Thus Proposition~\ref{prop:quantile-remainder} applies with
$Q=\hat Q_{n,\gamma,k}$, $a=-Q(c_\alpha)/f_\alpha$ and
$b=-C_{n,k}(\alpha)$.
Its input CDF remainder is $CL^7/n$, and its Taylor terms are bounded by
\[
 C\{L^{13/2}/n^{3/2}+L^9/n^2\}\le CL^7/n.
\]
This proves \eqref{eq:CF-main-corrected}.
Moreover,
\[
 \sup_t|\hat Q_{n,\gamma,k}(t)-B_{n,k}(t)-\gamma Q_{n,k}(t)|
 \le\frac{CL^{3/2}}{\sqrt n}
           \|\bar{\bm X}^{\,3}-\mathsf M_3\|_\infty
 \le CL^2/n.
\]
Consequently, on this common event,
\begin{equation}\label{eq:joint-linear-quantile}
 \sup_{\epsilon\le\alpha\le1-\epsilon}
 \left|\hat c_{1-\alpha,k}-c_\alpha+
       \frac{B_{n,k}(c_\alpha)+\gamma Q_{n,k}(c_\alpha)}{f_\alpha}\right|
 \le CL^7/n.
\end{equation}
Its failure probability is $Cn^{-3}$, which proves the stated bound.
\end{proof}

\subsection{Coverage expansion}\label{sec:coverage}
For $\alpha\in[\epsilon/4,1-\epsilon/4]$, write
\[
 c=c_\alpha,\qquad f=f_\alpha,\qquad
 \mathbf\Psi=\mathbf\Psi_\alpha,\qquad
 \mathbf A_\alpha=\mathbf\Psi_\alpha/(2f_\alpha).
\]
The reference covariance throughout remains the population matrix $\mathbf\Sigma$.

\begin{lemma}[Quadratic transformation of a Stein kernel]
\label{lem:joint-exact-kernel}
Let $\bm Y$ be centered, and let the symmetric random matrix
$\mathbf T^Y(\bm Y)$ satisfy
\[
 \E\{\bm Y^\top\nabla h(\bm Y)\}
 =\E\langle\mathbf T^Y(\bm Y),\nabla^2h(\bm Y)\rangle,
 \qquad h\in C_b^2(\R^d).
\]
Suppose the entries of $\bm Y$ and $\mathbf T^Y$ have finite moments
of every order. For a deterministic symmetric matrix $\mathbf A$ and integer
$n\ge1$, define
\[
 \begin{gathered}
 U=\frac{\bm Y^\top\mathbf A\bm Y-
                    \E(\bm Y^\top\mathbf A\bm Y)}{\sqrt n},
 \qquad
 \bm\xi=\frac{\bm Y+U\bm1_d}{\sqrt n},
 \\
 \mathbf D(\bm Y)=\mathbf I_d+
              \frac2{\sqrt n}\bm1_d(\mathbf A\bm Y)^\top.
 \end{gathered}
\]
Then $\bm\xi$ is centered and has the approximate Stein kernel
\begin{align*}
 \mathbf\tau(\bm\xi)
 &=\frac1n\E\{\mathbf D\mathbf T^Y\mathbf D^\top\mid\bm\xi\},\\
 \bm\beta(\bm\xi)
 &=\frac1n\E\left[
   \langle\mathbf A,2\mathbf T^Y-\bm Y^{\otimes2}
                         -\E\bm Y^{\otimes2}\rangle\bm1_d
                         \,\middle|\,\bm\xi\right].
\end{align*}
In particular, $\E\bm\beta=\bm0$, and all components of $\bm\beta$
are the same scalar random variable.
\end{lemma}

\begin{proof}[Proof of Lemma~\ref{lem:joint-exact-kernel}]
Write $u(\bm y)=\{\bm y^\top\mathbf A\bm y-
\E(\bm Y^\top\mathbf A\bm Y)\}/\sqrt n$ and
$\bm g(\bm y)=\bm y+u(\bm y)\bm1_d$.
For $h\in C_b^2(\R^d)$, the chain rule gives
\begin{align*}
 \nabla(h\circ\bm g)&=\mathbf D^\top\nabla h(\bm g),\\
 \nabla^2(h\circ\bm g)
 &=\mathbf D^\top\nabla^2h(\bm g)\mathbf D
       +\frac{2\mathbf A}{\sqrt n}
                         \{\bm1_d^\top\nabla h(\bm g)\}.
\end{align*}
Smooth cutoffs and the moment assumptions justify the Stein identity
for this composition. Since
$\mathbf D\bm Y=\bm Y+2(\bm Y^\top\mathbf A\bm Y)\bm1_d/\sqrt n$,
\begin{align*}
 \E\{\bm g(\bm Y)^\top\nabla h(\bm g(\bm Y))\}
 &=\E\langle\mathbf D\mathbf T^Y\mathbf D^\top,
                                    \nabla^2h(\bm g(\bm Y))\rangle\\
 &\quad+\frac1{\sqrt n}\E\left[
   \langle\mathbf A,2\mathbf T^Y-\bm Y^{\otimes2}
                       -\E\bm Y^{\otimes2}\rangle
                         \bm1_d^\top\nabla h(\bm g(\bm Y))\right].
\end{align*}
Replace $h(\bm z)$ by $h(\bm z/\sqrt n)$ and condition on $\bm\xi$.
Quadratic tests with smooth cutoffs, justified by the moment assumptions,
give $\E\mathbf T^Y=\E\bm Y^{\otimes2}$. Hence $\E\bm\beta=\bm0$.
The full matrix product is
\begin{align*}
 \mathbf D\mathbf T^Y\mathbf D^\top
 &=\mathbf T^Y+\frac2{\sqrt n}
   \{\bm1_d(\mathbf A\bm Y)^\top\mathbf T^Y
                   +\mathbf T^Y\mathbf A\bm Y\bm1_d^\top\}\\
 &\quad+\frac4n
     \{(\mathbf A\bm Y)^\top\mathbf T^Y\mathbf A\bm Y\}
                             \bm1_d\bm1_d^\top.
\end{align*}
\end{proof}

\begin{lemma}[The covariance-corrected array]\label{lem:joint-corrected-array}
Suppose Assumptions~\ref{ass:data} and \ref{ass:mult} hold,
$k$ is fixed, $\epsilon\in(0,1/4)$ and $r_n\le c$.
Consider either $\bm Y_i=\bm X_i$ with the original probability law,
or, on $\mathcal E_n$, $\bm Y_i=w_i\bm b_i$ under the conditional law,
where $\bm b_i=\bm X_i-\bar{\bm X}$ and the multipliers are bounded.
Write $\Prob_Y,\E_Y$ for the selected law and define
\begin{gather*}
 \mathbf\Sigma_0=n^{-1}\sum_i\E_Y\bm Y_i^{\otimes2},\qquad
 \mathsf M_0=n^{-1}\sum_i\E_Y\bm Y_i^{\otimes3},\\
 H_0(t)=\tfrac12\langle\mathbf\Sigma_0-\mathbf\Sigma,
                                \mathsf J_{2,k}(t)\rangle
          -\tfrac1{6\sqrt n}\langle\mathsf M_0,\mathsf J_{3,k}(t)\rangle.
\end{gather*}
For $\alpha\in[\epsilon/4,1-\epsilon/4]$, put
\begin{gather*}
 U_i=n^{-1/2}\langle\bm Y_i^{\otimes2}-\E_Y\bm Y_i^{\otimes2},
                                                    \mathbf A_\alpha\rangle,
 \quad \bm\xi_i=(\bm Y_i+U_i\bm1_d)/\sqrt n,
 \quad \widetilde T_{k,\alpha}=(\sum_i\bm\xi_i)_{[k]},\\
 D_\alpha(t)=\frac1{\sqrt n}
     \langle\mathsf M_0\otimes\bm1_d,
                  \mathbf A_\alpha\otimes\mathsf J_{2,k}(t)\rangle.
\end{gather*}
Then, uniformly over these $\alpha$,
\[
 \sup_t|\Prob_Y(\widetilde T_{k,\alpha}\le t)
                   -G_k(t)-H_0(t)-D_\alpha(t)|\le CL^{11}/n,
\]
and, for every $t\in\R$ and $h\ge0$,
\[
 \Prob_Y(t\le\widetilde T_{k,\alpha}\le t+h)
       \le C\{h\sqrt L+L^{11}/n\}.
\]
Constants depend only on $k,b,\sigma_*,b_w,\epsilon$.
\end{lemma}

\begin{proof}[Proof of Lemma~\ref{lem:joint-corrected-array}]
Translation equivariance gives the exact identity
\[
 \widetilde T_{k,\alpha}
 =(n^{-1/2}\sum_i\bm Y_i)_{[k]}
  +\frac{\langle n^{-1}\sum_i
                (\bm Y_i^{\otimes2}-\E_Y\bm Y_i^{\otimes2}),
                    \mathbf\Psi_\alpha\rangle}{2f_\alpha}.
\]
In the conditional case an exact kernel for $\bm Y_i$ is
\[
 \mathbf T_i^Y(\bm Y_i)
  =\E^*\{\tau^w(w_i)\bm b_i\bm b_i^\top\mid\bm Y_i\}.
\]
Apply Lemma~\ref{lem:joint-exact-kernel}, with
$\mathbf D_i=\mathbf I_d+2n^{-1/2}\bm1_d(\mathbf A_\alpha\bm Y_i)^\top$.
Thus
\begin{align*}
 \mathbf\tau_i(\bm\xi_i)
 &=n^{-1}\E_Y(\mathbf D_i\mathbf T_i^Y\mathbf D_i^\top\mid\bm\xi_i),\\
 \bm\beta_i(\bm\xi_i)
 &=n^{-1}\E_Y\left[
   \langle\mathbf A_\alpha,2\mathbf T_i^Y-\bm Y_i^{\otimes2}
                      -\E_Y\bm Y_i^{\otimes2}\rangle\bm1_d
                                         \,\middle|\,\bm\xi_i\right].
\end{align*}
All these variables have finite moments. Put
\[
 a=\|\mathbf A_\alpha\|_1\le CL,\qquad
 V_i=\|\bm Y_i\|_\infty,\qquad K_i=\|\mathbf T_i^Y\|_\infty.
\]
The data envelopes, or the bounded conditional multipliers, give for
every fixed $q\ge2$
\[
 \|V_i\|_{L^q}\le C_qL,\qquad \|K_i\|_{L^q}\le C_qL^2,
 \qquad \|U_i\|_{L^q}\le C_qL^3/\sqrt n.
\]
The full Jacobian product satisfies
\[
 \|\mathbf D_i\mathbf T_i^Y\mathbf D_i^\top-\mathbf T_i^Y\|_\infty
 \le \frac{4aV_iK_i}{\sqrt n}+\frac{4a^2V_i^2K_i}{n}.
\]
Conditional Jensen and H\"older therefore imply
\begin{equation}\label{eq:joint-kernel-moments}
\begin{gathered}
 \|\bm\xi_i\|_{L^q(\max)}\le C_qL/\sqrt n,
 \quad \|\mathbf\tau_i\|_{L^q(\max)}\le C_qL^2/n,
 \quad \|\bm\beta_i\|_{L^q(\max)}\le C_qL^3/n.
\end{gathered}
\end{equation}

Put $\mathscr H=\sigma(\bm\xi_1,\ldots,\bm\xi_n)$ and
\[
 \overline{\mathbf T}=\sum_i\mathbf\tau_i-\mathbf\Sigma,
 \qquad \bm B=\sum_i\bm\beta_i.
\]
Independence and the definition of each conditional kernel give
\[
 \overline{\mathbf T}
 =\E_Y\left[
   n^{-1}\sum_i\mathbf D_i\mathbf T_i^Y\mathbf D_i^\top
               -\mathbf\Sigma\,\middle|\,\mathscr H\right].
\]
For the untransformed kernels, Bernstein's inequality gives
\[
 \left\|n^{-1}\sum_i
       (\mathbf T_i^Y-\E_Y\mathbf T_i^Y)\right\|_{L^2(\max)}
 \le C\{\sqrt{L/n}+L^3/n\}.
\]
For fixed data in $\mathcal E_n$, this uses
\[
 \max_i\|\mathbf T_i^Y\|_\infty\le CL^2,
 \qquad
 \max_{j,\ell}n^{-1}\sum_i
       \E_Y|(\mathbf T_i^Y)_{j\ell}|^2\le C,
\]
which follow from the bounded multiplier kernel and the empirical
sixth-moment bound. For the original array, the Bernstein bound follows
directly from the centered exponential kernel envelope in
Assumption~\ref{ass:data}. Since
$n^{-1}\sum_i\E_Y\mathbf T_i^Y=\mathbf\Sigma_0$, conditional Jensen yields
\begin{align*}
 \|\overline{\mathbf T}\|_{L^2(\max)}
 &\le C\{\sqrt{L/n}+L^3/n\}
     +\|\mathbf\Sigma_0-\mathbf\Sigma\|_\infty
     +C\{L^4/\sqrt n+L^6/n\}\\
 &\le CL^4/\sqrt n.
\end{align*}

All coordinates of $\bm\beta_i$ coincide. The variables $\beta_{i1}$
are independent and centered under $\Prob_Y$, so
\[
 \E_Y\|\bm B\|_\infty^2
 =\sum_i\E_Y|\beta_{i1}|^2\le CL^6/n.
\]
Minkowski's and H\"older's inequalities also give
\begin{align*}
 \left\|\sum_i\bm\xi_i^{\otimes3}\right\|_{L^2(\max)}
 &\le\sum_i\|\bm\xi_i\|_{L^6(\max)}^3
 \le CL^3/\sqrt n,\\
 \left\|\sum_i\bm\beta_i\otimes\bm\xi_i\right\|_{L^2(\max)}
 &\le\sum_i\|\bm\beta_i\|_{L^4(\max)}
                 \|\bm\xi_i\|_{L^4(\max)}
 \le CL^4/\sqrt n.
\end{align*}
Use the coefficients $a_4,a_5,b_4,b_2,b_3$ defined before
Proposition~\ref{prop:global-stein-transfer}, with expectation $\E_Y$.
H\"older's inequality and \eqref{eq:joint-kernel-moments} give
\[
\begin{array}{c|ccccc}
 \text{group}&a_4&a_5&b_4&b_2&b_3\\\hline
 \text{bound}&CL^8/n&CL^7/n&CL^8/n&CL^6/n&CL^7/n\\
 \text{analytic factor}&CL^3&CL^{5/2}&CL^2&CL&CL^{3/2}
\end{array}
\]
For example,
\[
 \E_Y\|\overline{\mathbf T}\otimes
           \sum_i\bm\beta_i\otimes\bm\xi_i\|_\infty
 \le(CL^4/\sqrt n)(CL^4/\sqrt n)=CL^8/n.
\]
For the smoothing term, put
\[
 \mathbf\Delta=\operatorname{Cov}_Y(\sum_i\bm\xi_i)-\mathbf\Sigma,
 \qquad \mathsf K_3=\sum_i\E_Y\bm\xi_i^{\otimes3}.
\]
The definition of the centered $U_i$ gives
\begin{align*}
 \|\mathbf\Delta\|_\infty
 &\le\|\mathbf\Sigma_0-\mathbf\Sigma\|_\infty
     +\frac2n\sum_i\|\E_Y(\bm Y_iU_i)\|_\infty
     +\frac1n\sum_i\E_YU_i^2\\
 &\le C\{\sqrt{L/n}+L^4/\sqrt n+L^6/n\}
 \le CL^4/\sqrt n,\\
 \|\mathsf K_3\|_\infty
 &\le\sum_i\E_Y\|\bm\xi_i\|_\infty^3
 \le CL^3/\sqrt n.
\end{align*}
Insert these bounds and the five groups into
\eqref{eq:rank-stein-bound} with smoothing parameter $n^{-2}$.
The five groups contribute at most $CL^{11}/n$, and the smoothing
term is at most
\[
 \frac Cn\left\{L+L^2\frac{L^4}{\sqrt n}
                    +L^{5/2}\frac{L^3}{\sqrt n}\right\}
 \le CL^{11}/n.
\]
Define
\[
 \bm m_i=\E_Y(\bm Y_iU_i),\quad v_i=\E_YU_i^2,\quad
 \bar{\bm m}=n^{-1}\sum_i\bm m_i,\quad \bar v=n^{-1}\sum_i v_i.
\]
The covariance identity is
\begin{gather*}
 \operatorname{Cov}_Y(\sum_i\bm\xi_i)-\mathbf\Sigma
 =\mathbf\Sigma_0-\mathbf\Sigma
     +\bar{\bm m}\bm1_d^\top+\bm1_d\bar{\bm m}^{\top}
     +\bar v\bm1_d\bm1_d^\top,\\
 \bar{\bm m}=n^{-1/2}\mathsf M_0:\mathbf A_\alpha,
 \qquad \bar v\le CL^6/n.
\end{gather*}
Here the contraction is
\[
 (\mathsf M_0:\mathbf A_\alpha)_j
 =\sum_{a,b=1}^d(\mathsf M_0)_{jab}(\mathbf A_\alpha)_{ab}.
\]
The two cross terms give exactly $D_\alpha(t)$ in the signed integral.
The remaining third-tensor change is bounded by
\[
 \left\|n^{-3/2}\sum_i\E_Y
       \{(\bm Y_i+U_i\bm1_d)^{\otimes3}-\bm Y_i^{\otimes3}\}\right\|_\infty
 \le C\{L^5/n+L^7/n^{3/2}+L^9/n^2\}.
\]
Contraction with the global derivative tensors absorbs this change
and the $\bar v$ term into $CL^{11}/n$. This proves the signed CDF expansion.

Both cases have $\|\mathsf M_0\|_\infty\le C$ and
$\|\mathbf\Sigma_0-\mathbf\Sigma\|_\infty\le C\sqrt{L/n}$.
Consequently,
\[
 \sup_t\{|f_k(t)|+|H_0'(t)|+|D_\alpha'(t)|\}
 \le C\{\sqrt L+L^2/\sqrt n+L^{5/2}/\sqrt n\}
 \le C\sqrt L.
\]
Apply the continuous signed approximation at both endpoints, including
the left limit at $t$, to obtain the anti-concentration assertion.
\end{proof}

\begin{lemma}[Moments and concentration of the dependence term]
\label{lem:joint-R-control}
Suppose Assumption~\ref{ass:data} holds, $k$ is fixed,
$\epsilon\in(0,1/4)$ and $r_n\le c$.
Let $\bar R(\alpha)=\E R_{n,k}(\alpha)$.
Uniformly for $\alpha\in[\epsilon/4,1-\epsilon/4]$,
\[
 \E|R_{n,k}(\alpha)|\le CL_d^2/\sqrt n,
 \qquad |\bar R'(\alpha)|\le CL_d^3/\sqrt n.
\]
On $\mathcal E_n$,
\[
 \sup_{\epsilon/4\le\alpha\le1-\epsilon/4}
       |R_{n,k}(\alpha)-\bar R(\alpha)|\le CL^{5/2}/n.
\]
Constants depend only on $k,b,\sigma_*,\epsilon$.
\end{lemma}

\begin{proof}[Proof of Lemma~\ref{lem:joint-R-control}]
H\"older's inequality gives
$\max_{a,b,c}\E|\bar X^{\,3}_{abc}|\le Cb^3$.
Hence the bound, without restriction to a concentration event, is
\[
 \E|R_{n,k}(\alpha)|
 \le\frac{Cb^3}{2\sqrt n f_\alpha}\|\mathbf\Psi_\alpha\|_1^2
 \le CL_d^2/\sqrt n.
\]
The moment event gives
\[
 |R_{n,k}(\alpha)-\bar R(\alpha)|
 \le\frac{\|\bar{\bm X}^{\,3}-\mathsf M_3\|_\infty
                 \|\mathbf\Psi_\alpha\|_1^2}{2\sqrt n f_\alpha}
 \le CL^{5/2}/n.
\]
For derivatives with respect to $\alpha$,
\[
 |c_\alpha'|=f_\alpha^{-1}\le C,
 \quad \|\mathbf\Psi_\alpha'\|_1\le CL_d^{3/2},
 \quad |f_\alpha'|\le CL_d.
\]
Differentiating the finite deterministic contraction for $\bar R$ gives
\[
 |\bar R'(\alpha)|
 \le\frac C{\sqrt n}
        \{L_dL_d^{3/2}+L_d^2L_d\}
 \le CL_d^3/\sqrt n.
\]
\end{proof}

\begin{proof}[Proof of Theorem~\ref{thm:coverage-main}]
Use Lemma~\ref{lem:joint-corrected-array} with $\bm Y_i=\bm X_i$.
Then $H_0=Q_{n,k}$ and
\[
 \widetilde T_{k,\alpha}=T_{n,[k]}+B_{n,k}(c_\alpha)/f_\alpha,
 \qquad D_\alpha(c_\alpha)=\bar R(\alpha).
\]
On $\mathcal E_n$, equation~\eqref{eq:joint-linear-quantile} brackets
the rejection event by the corrected statistic at the deterministic level
\[
 t_\alpha=c_\alpha-\gamma Q_{n,k}(c_\alpha)/f_\alpha
\]
within distance $CL^7/n$.
The corrected anti-concentration bound therefore gives
\[
 |\Prob(T_{n,[k]}>\hat c_{1-\alpha,k})
                -\Prob(\widetilde T_{k,\alpha}>t_\alpha)|
 \le C\{L^{15/2}/n+L^{11}/n+n^{-3}\}.
\]
The global derivative bounds imply
\begin{align*}
 G_k(t_\alpha)&=1-\alpha-\gamma Q_{n,k}(c_\alpha)+O(L^4/n),\\
 Q_{n,k}(t_\alpha)&=Q_{n,k}(c_\alpha)+O(L^{7/2}/n),\\
 D_\alpha(t_\alpha)&=\bar R(\alpha)+O(L^4/n).
\end{align*}
Insert these identities into the signed CDF expansion and take complements.
Each error is at most $Cr_n$.

The same event also gives the inclusion
\[
 \{T_{n,[k]}=\hat c_{1-\alpha,k}\}\cap\mathcal E_n
 \subseteq\{|\widetilde T_{k,\alpha}-t_\alpha|\le CL^7/n\}.
\]
The same anti-concentration bound proves that its probability is $O(r_n)$.
Thus both rejection conventions satisfy the theorem.
The expectation assertion follows from Lemma~\ref{lem:joint-R-control}.
\end{proof}

\begin{lemma}[Central order-statistic thresholds]\label{lem:spectral-threshold}
Let $\bm Z\sim N(\bm0,\mathbf\Sigma)$ with $\Sigma_{jj}=1$ and
$s_*^2\mathbf I_d\preceq\mathbf\Sigma\preceq K\mathbf I_d$, where
$s_*>0$ and $K<\infty$ are fixed. For fixed $k$, $\epsilon\in(0,1/2)$ and
$p\in[\epsilon,1-\epsilon]$, write $c=c^G_{p,k}$. For all sufficiently
large $d$,
\[
 1<c\le C\sqrt{\log d},\qquad
 \phi(c)\le C d^{-1}\sqrt{\log d}.
\]
\end{lemma}
\begin{proof}[Proof of Lemma~\ref{lem:spectral-threshold}]
The two-dimensional covariance minors give
\[
 \rho:=\max_{j\ne\ell}|\Sigma_{j\ell}|\le1-s_*^2<1.
\]
Choose fixed $\Lambda>\max\{2k,32/\epsilon\}$ and set
\[
 u=\Phi^{-1}(1-\Lambda/d),\qquad
 N(u)=\sum_j\ind\{Z_j>u\}.
\]
For a bivariate standard normal pair with correlation $r$,
\[
 \operatorname{Cov}(\mathbf1\{Z_j>u\},\mathbf1\{Z_\ell>u\})
 =\int_0^r\frac{\exp\{-u^2/(1+s)\}}{2\pi\sqrt{1-s^2}}\,ds.
\]
This identity follows by differentiating the bivariate Gaussian CDF
with respect to its correlation. Split the pairs at $|\Sigma_{j\ell}|=u^{-2}$ and use
\[
 |r|\ind\{|r|>u^{-2}\}\le u^2r^2.
\]
Since
\[
 \sum_{j,\ell}\Sigma_{j\ell}^2\le dK^2,
 \qquad \sum_{j,\ell}|\Sigma_{j\ell}|\le Kd^{3/2},
\]
we obtain
\[
 \operatorname{Var}N(u)
 \le\Lambda+C\{dK^2u^2e^{-u^2/(1+\rho)}+Kd^{3/2}e^{-u^2}\}.
\]
Mills' inequalities give $u\asymp\sqrt{\log d}$ and
$e^{-u^2/2}\le C\Lambda\sqrt{\log d}/d$. Thus the last two terms are
\[
 O\!\left(d^{-(1-\rho)/(1+\rho)}(\log d)^{1+1/(1+\rho)}
                  +d^{-1/2}\log d\right)=o(1).
\]
As $\mathbb EN(u)=\Lambda$, Chebyshev's inequality yields
\[
 G_k(u)=\mathbb P\{N(u)\le k-1\}
 \le\frac{\Lambda+o(1)}{(\Lambda-k+1)^2}<\epsilon/2.
\]
Therefore $c>u$ and $\phi(c)\le\phi(u)\le C\sqrt{\log d}/d$.
For $t=\sqrt{2\log\{d/(k\epsilon)\}}$, Markov's inequality gives
$\mathbb P(T_{\bm Z,[k]}>t)\le d\bar\Phi(t)/k\le\epsilon$.
This proves the upper bound for $c$.
\end{proof}

\begin{lemma}[Rows of the rank-set Hessian]\label{lem:spectral-hessian-rows}
Let $\Sigma_{jj}=1$ and
$s_*^2\mathbf I_d\preceq\mathbf\Sigma\preceq K\mathbf I_d$
with fixed $s_*>0$ and $K<\infty$. Fix $k$, $\epsilon\in(0,1/2)$
and $p\in[\epsilon,1-\epsilon]$, set $c=c^G_{p,k}$, and put
$\mathbf\Psi=\int_{A_k(c\bm1_d)}\nabla^2\phi_{\mathbf\Sigma}$.
Then, uniformly for $p\in[\epsilon,1-\epsilon]$,
\[
 \max_j\sum_{\ell=1}^d|\Psi_{j\ell}|
       \le C\frac{\log(ed)}d,
 \qquad\|\mathbf\Psi\bm1_d\|_2
       \le C\frac{\log(ed)}{\sqrt d}.
\]
\end{lemma}
\begin{proof}[Proof of Lemma~\ref{lem:spectral-hessian-rows}]
Put
\[
 \rho=\max_{j\ne\ell}|\Sigma_{j\ell}|\le1-s_*^2<1.
\]
Let $N_{-j}$ count exceedances excluding coordinate $j$.
Integrating first in $z_j$ gives
\[
 \Psi_{jj}=-\phi(c)\mathbb E[
       (\mathbf\Sigma^{-1}\bm Z)_j\mathbf1\{N_{-j}(c)=k-1\}
                                               \mid Z_j=c].
\]
Conditionally on $Z_j=c$, the scalar
$(\mathbf\Sigma^{-1}\bm Z)_j$ is Gaussian with mean $c$ and variance
$(\mathbf\Sigma^{-1})_{jj}-1\le s_*^{-2}-1$. Consequently
\[
 |\Psi_{jj}|\le C(1+c)\phi(c)\le C\log(ed)/d.
\]
For $j\ne\ell$, let $N_{-j,-\ell}$ omit those two coordinates.
Integrating in $z_j,z_\ell$ gives
\begin{align*}
 \Psi_{j\ell}
 &=\phi_{\mathbf\Sigma[ j,\ell ]}(c,c)
 \bigl[\mathbb P\{N_{-j,-\ell}(c)=k-1\mid Z_j=Z_\ell=c\}\\
 &\hspace{35mm}-\mathbb P\{N_{-j,-\ell}(c)=k-2\mid Z_j=Z_\ell=c\}\bigr].
\end{align*}
Here $\mathbf\Sigma[j,\ell]$ is the corresponding two-dimensional principal covariance matrix. A probability with a negative count is zero. Therefore
\[
 |\Psi_{j\ell}|\le C\exp\{-c^2/(1+|\Sigma_{j\ell}|)\}.
\]
The set $I_j=\{\ell\ne j:|\Sigma_{j\ell}|>c^{-2}\}$ has cardinality
at most $K^2c^4$. The threshold bound gives
\begin{align*}
 \sum_{\ell\ne j}|\Psi_{j\ell}|
 &\le C\{K^2c^4e^{-c^2/(1+\rho)}+de^{-c^2}\}\\
 &\le C\{d^{-2/(1+\rho)}(\log d)^{2+1/(1+\rho)}
                         +d^{-1}\log d\}
 \le C d^{-1}\log(ed).
\end{align*}
The finitely many smaller dimensions are absorbed into $C$ by the
fixed eigenvalue bounds and the global derivative inequality.
Summing the squared row bounds proves the second assertion.
\end{proof}

\begin{corollary}[The dependence correction under bounded spectrum]
\label{cor:bounded-spectrum-R}
Suppose the observations are independent and centered,
$\max_{i,j}\|X_{ij}\|_{\psi_1}\le b$, $\Sigma_{jj}=1$ for every $j$, and
$s_*^2\mathbf I_d\preceq\mathbf\Sigma\preceq K\mathbf I_d$.
Then, for fixed $k$ and $\epsilon\in(0,1/2)$,
\[
 \sup_{\epsilon\le\alpha\le1-\epsilon}
 \mathbb E|R_{n,k}(\alpha)|
     \le C\frac{\log^2(ed)}{\sqrt{dn}}.
\]
The constant depends only on $k,\epsilon,b,s_*,K$.
\end{corollary}
\begin{proof}[Proof of Corollary~\ref{cor:bounded-spectrum-R}]
Write $c=c^G_{1-\alpha,k}$, $f=f_k(c)$ and $\bm v=\mathbf\Psi\bm1_d$. Then
\[
 R_{n,k}=\frac1{2\sqrt n f}\frac1n\sum_i
           (\bm X_i^\top\mathbf\Psi\bm X_i)(\bm v^\top\bm X_i).
\]
The moment envelope, the global Hessian bound and the covariance identity
imply
\begin{align*}
 n^{-1}\sum_i\mathbb E(\bm X_i^\top\mathbf\Psi\bm X_i)^2
      &\le Cb^4\|\mathbf\Psi\|_1^2\le C\log^2(ed),\\
 n^{-1}\sum_i\mathbb E(\bm v^\top\bm X_i)^2
      &=\bm v^\top\mathbf\Sigma\bm v
        \le K\|\bm v\|_2^2\le C\log^2(ed)/d.
\end{align*}
Apply Cauchy--Schwarz to the product measure consisting of uniform
averaging over $i$ and expectation, and use the central density lower
bound $f\ge c_\epsilon$. The same argument applies after rescaling any
common fixed positive marginal variance to one.
\end{proof}

\begin{lemma}[Independent Gaussian quantile regularity]
\label{lem:iid-gaussian-regularity}
Let $Z_1,\ldots,Z_d$ be independent standard normal variables, let $k\ge1$ be fixed, and let $G_{d,k}$ and $f_{d,k}$ be the distribution function and density of their $k$th largest coordinate. Fix $\epsilon\in(0,1/2)$, put $a_d=\sqrt{\log(ed)}$, and define
\[
 \mathcal T_{d,k,\epsilon}
 =\{G_{d,k}^{-1}(p):\epsilon\le p\le1-\epsilon\}.
\]
For all sufficiently large $d$, uniformly for $t\in\mathcal T_{d,k,\epsilon}$,
\[
 c_{k,\epsilon}a_d\le t\le C_{k,\epsilon}a_d,
 \qquad c_{k,\epsilon}a_d\le f_{d,k}(t)\le C_{k,\epsilon}a_d,
\]
\[
 |f_{d,k}'(t)|\le C_{k,\epsilon}a_d^2,
 \qquad |f_{d,k}''(t)|\le C_{k,\epsilon}a_d^3.
\]
Consequently,
\[
 \sup_{\epsilon\le p\le1-\epsilon}
 \left|\frac{d^r}{dp^r}G_{d,k}^{-1}(p)\right|
 \le C_{k,\epsilon}a_d^{-1},\qquad r=1,2,3.
\]
\end{lemma}

\begin{proof}[Proof of Lemma~\ref{lem:iid-gaussian-regularity}]
Write $p_t=\bar\Phi(t)$, $q_t=\Phi(t)$, and $\lambda_t=dp_t$. Then
\[
 G_{d,k}(t)=\Prob\{\operatorname{Bin}(d,p_t)\le k-1\},
 \qquad
 f_{d,k}(t)=d\binom{d-1}{k-1}\phi(t)p_t^{k-1}q_t^{d-k}.
\]
Let $N\sim\operatorname{Bin}(d,p_t)$. On the quantile window, Markov's
inequality gives
\[
 \epsilon\le\Prob(N\ge k)\le\lambda_t/k.
\]
If $\lambda_t\ge2(k-1)$, Chebyshev's inequality gives
\[
 \epsilon\le\Prob(N\le k-1)
 \le\frac{\lambda_t}{(\lambda_t-k+1)^2}\le\frac4{\lambda_t}.
\]
Thus Mills' inequalities yield
\[
 k\epsilon\le\lambda_t\le\Lambda^{\mathrm{ind}}_{k,\epsilon}:=\max\{2(k-1),4/\epsilon\},
 \qquad t^2=2\log d-\log\log d+O_{k,\epsilon}(1).
\]

For $d\ge2k$, set
\[
 b_{d,k}(t)=\binom{d-1}{k-1}p_t^{k-1}q_t^{d-k}.
\]
The compact bounds on $\lambda_t$ give $0<c_{k,\epsilon}\le b_{d,k}(t)\le C_{k,\epsilon}$. Since $\phi(t)/\bar\Phi(t)\asymp1+t$,
\[
 f_{d,k}(t)=\lambda_t\frac{\phi(t)}{\bar\Phi(t)}b_{d,k}(t)\asymp a_d.
\]
Let $h(t)=\phi(t)/\bar\Phi(t)$, $v(t)=\phi(t)/\Phi(t)$, and $\ell(t)=\log f_{d,k}(t)$. Differentiation gives
\[
 \ell'(t)=-t-(k-1)h(t)+(d-k)v(t),
\]
\[
 \begin{gathered}
 \ell''(t)=-1-(k-1)h'(t)+(d-k)v'(t),
 \\ h'(t)=h(t)\{h(t)-t\},\quad
 v'(t)=-tv(t)-v(t)^2.
 \end{gathered}
\]
For $t>0$, $h(t)>t$, and Gaussian integration by parts gives
\[
 \Var(Z_1\mid Z_1>t)=1+th(t)-h(t)^2=1-h'(t)>0.
\]
Thus $0<h'(t)<1$. Also $dv(t)=O(a_d)$ and $d|v'(t)|=O(a_d^2)$. Hence
\[
 |\ell'(t)|\le C_{k,\epsilon}a_d,
 \qquad |\ell''(t)|\le C_{k,\epsilon}a_d^2.
\]
Use $f'=f\ell'$ and $f''=f\{(\ell')^2+\ell''\}$. Finally, if $c_p=G_{d,k}^{-1}(p)$,
\[
 c_p'=\frac1{f(c_p)},\qquad
 c_p''=-\frac{f'(c_p)}{f(c_p)^3},\qquad
 c_p'''=\frac{3f'(c_p)^2}{f(c_p)^5}-\frac{f''(c_p)}{f(c_p)^4}.
\]
These estimates prove the inverse bounds. For fixed smaller $d$, the density and derivative bounds follow by compactness of the central quantile interval.
\end{proof}

\begin{lemma}[The independent-coordinate dependence tensor]
\label{lem:iid-dependence-tensor}
Let $\bm X_1,\ldots,\bm X_n$ be independent centered vectors in $\R^d$. Assume that the coordinates within each $\bm X_i$ are independent,
\[
 n^{-1}\sum_{i=1}^n\E(\bm X_i\bm X_i^\top)=\mathbf I_d,
 \qquad \max_{i,j}\E|X_{ij}|^8\le M_8<\infty.
\]
Let $k\ge1$ and $\epsilon\in(0,1/2)$ be fixed. For $p\in[\epsilon,1-\epsilon]$, put $c=G_{d,k}^{-1}(p)$, $f=f_{d,k}(c)$, $f'=f_{d,k}'(c)$, and
\[
 A_{k,c}=\{\bm z\in\R^d:z_{[k]}\le c\},\qquad
 \mathbf\Psi=\int_{A_{k,c}}\nabla^2\phi_{\mathbf I_d}(\bm z)\,d\bm z,
\]
\[
 \begin{gathered}
 \mathsf M_{3,n}=\frac1n\sum_{i=1}^n\bm X_i^{\otimes3},\qquad
 R_{n,k}(1-p)=\frac{\langle\mathsf M_{3,n}\otimes\bm1_d,
                     \mathbf\Psi\otimes\mathbf\Psi\rangle}
                    {2\sqrt n\,f},\\
 \bar\kappa_n=\frac1{nd}\sum_{i=1}^n\sum_{j=1}^d\E X_{ij}^3.
 \end{gathered}
\]
For $d\ge\max\{2,k\}$, the following identity is exact:
\[
 \E R_{n,k}(1-p)=-\frac{\bar\kappa_n\,c\,f'}{2d\sqrt n}.
\]
For all sufficiently large $d$, uniformly over $p\in[\epsilon,1-\epsilon]$,
\[
 \E|R_{n,k}(1-p)|
 \le C_{k,\epsilon,M_8}\{\log(ed)\}^{3/2}
 \left(\frac1{d\sqrt n}+\frac1{n\sqrt d}\right).
\]
In particular, if $d\ge n$, this is bounded by
$C_{k,\epsilon,M_8}\{\log(ed)\}^{3/2}n^{-3/2}$.
\end{lemma}

\begin{proof}[Proof of Lemma~\ref{lem:iid-dependence-tensor}]
Permutation symmetry makes all diagonal entries of $\mathbf\Psi$ equal to some $A$ and all off-diagonal entries equal to some $B$. Conditional integration in one coordinate gives
\begin{align*}
 A&=\phi'(c)\Prob\{\operatorname{Bin}(d-1,\bar\Phi(c))=k-1\}
   =-\frac{cf}{d},\\
 \sum_{a,b=1}^d\Psi_{ab}&=f',\qquad
 A+(d-1)B=\frac{f'}d,\qquad
 B=\frac{f'+cf}{d(d-1)}.
\end{align*}
The sum identity follows by differentiating
$\Prob\{\bm Z+u\bm1_d\in A_{k,c}\}=G_{d,k}(c-u)$ twice at $u=0$.

Set $S_i=\sum_jX_{ij}$ and $Q_i=\sum_jX_{ij}^2$. Since every row sum of $\mathbf\Psi$ is $f'/d$,
\[
 R_{n,k}(1-p)
 =\frac{f'}{2d\sqrt n\,f}\frac1n\sum_{i=1}^nY_i,
 \qquad
 Y_i=\{(A-B)Q_i+B S_i^2\}S_i.
\]
Coordinate independence and centering imply
\[
 \E(Q_iS_i)=\E(S_i^3)=\sum_{j=1}^d\E X_{ij}^3,
 \qquad \E Y_i=A\sum_{j=1}^d\E X_{ij}^3.
\]
Substitution yields $\E R_{n,k}(1-p)=\bar\kappa_n A f'/(2\sqrt n f)$, proving the identity.

Write $a_d=\sqrt{\log(ed)}$. Lemma~\ref{lem:iid-gaussian-regularity} gives
\[
 |A|+|A-B|\le C a_d^2/d,\quad |B|\le C a_d^2/d^2,
 \quad |f'/f|\le C a_d.
\]
The eighth-moment bound and the expansion of moments of independent centered sums imply
\[
 \|Q_i\|_{L^4}\le Cd,
 \qquad \|S_i\|_{L^4}+\|S_i\|_{L^6}\le C\sqrt d.
\]
Indeed, in the expansion of $\E S_i^r$ for even $r\le6$, every nonzero summand contains at most $r/2$ distinct indices. H\"older's inequality therefore gives
\[
 \|Q_iS_i\|_{L^2}\le\|Q_i\|_{L^4}\|S_i\|_{L^4}\le Cd^{3/2},
 \qquad \|S_i^3\|_{L^2}\le Cd^{3/2},
 \qquad \E Y_i^2\le C a_d^4d.
\]
Independence across $i$ now yields
\[
 \|R_{n,k}(1-p)-\E R_{n,k}(1-p)\|_{L^2}
 \le\frac{|f'|}{2d\sqrt n f}
      \left(\frac1{n^2}\sum_i\Var(Y_i)\right)^{1/2}
 \le\frac{Ca_d^3}{n\sqrt d}.
\]
Finally, $|\bar\kappa_n|\le M_8^{3/8}$ and the exact expectation formula give
\[
 |\E R_{n,k}(1-p)|\le\frac{Ca_d^3}{d\sqrt n}.
\]
The triangle inequality and Cauchy--Schwarz prove the $L^1$ bound.
\end{proof}

\begin{corollary}[Isotropic data with dependent coordinates]
\label{cor:isotropic-dependence-tensor}
Let $\bm X_1,\ldots,\bm X_n$ be independent centered vectors in $\R^d$ satisfying
\[
 n^{-1}\sum_{i=1}^n\E(\bm X_i\bm X_i^\top)=\mathbf I_d,
 \qquad \max_{i,j}\|X_{ij}\|_{\psi_1}\le b.
\]
No independence between coordinates is assumed. Let $k\ge1$ and $\epsilon\in(0,1/2)$ be fixed, $p\in[\epsilon,1-\epsilon]$, $c=G_{d,k}^{-1}(p)$, $f=f_{d,k}(c)$, and
\[
 \begin{gathered}
 \mathbf\Psi=\int_{\{\bm z:z_{[k]}\le c\}}\nabla^2\phi_{\mathbf I_d}(\bm z)\,d\bm z,
 \quad \mathsf M_{3,n}=n^{-1}\sum_i\bm X_i^{\otimes3},
 \\ R_{n,k}(1-p)=\frac{\langle\mathsf M_{3,n}\otimes\bm1_d,
               \mathbf\Psi\otimes\mathbf\Psi\rangle}{2\sqrt n\,f}.
 \end{gathered}
\]
For all sufficiently large $d$, uniformly in $p\in[\epsilon,1-\epsilon]$,
\[
 \E|R_{n,k}(1-p)|
 \le \frac{C_{k,\epsilon}b^2\{\log(ed)\}^{3/2}}{\sqrt{dn}}.
\]
If $d\ge n$, the right-hand side is at most $C_{k,\epsilon}b^2\{\log(ed)\}^{3/2}/n$.
\end{corollary}

\begin{proof}[Proof of Corollary~\ref{cor:isotropic-dependence-tensor}]
Put $a_d=\sqrt{\log(ed)}$ and $f'=f_{d,k}'(c)$. The Gaussian matrix identities in the proof of Lemma~\ref{lem:iid-dependence-tensor} hold independently of the distribution of the observations. Therefore every row sum of $\mathbf\Psi$ is $f'/d$, and
\[
 \|\mathbf\Psi\|_1
 =|cf|+|f'+cf|\le |f'|+2|cf|\le C_{k,\epsilon}a_d^2.
\]
The tensor contraction is exactly
\[
 R_{n,k}(1-p)=\frac1{2\sqrt n\,f}\frac1n\sum_{i=1}^n
 (\bm X_i^\top\mathbf\Psi\bm X_i)
 \left(\frac{f'}d\bm1_d^\top\bm X_i\right).
\]
The sub-exponential envelope gives $\max_{i,j}\E X_{ij}^4\le48b^4$. Minkowski's and Cauchy--Schwarz inequalities imply
\[
 \|\bm X_i^\top\mathbf\Psi\bm X_i\|_{L^2}
 \le\sum_{a,b}|\Psi_{ab}|\|X_{ia}X_{ib}\|_{L^2}
 \le\sqrt{48}\,b^2\|\mathbf\Psi\|_1.
\]
The average covariance identity gives
\[
 \frac1n\sum_{i=1}^n\E\left[
       \left(\frac{f'}d\bm1_d^\top\bm X_i\right)^2\right]
 =\frac{(f')^2}{d^2}\bm1_d^\top\mathbf I_d\bm1_d
 =\frac{(f')^2}d.
\]
Apply Cauchy--Schwarz to the probability measure that first selects $i$ uniformly and then samples $\bm X_i$. This yields
\begin{align*}
 &\E|R_{n,k}(1-p)|\\
 &\quad\le\frac1{2\sqrt n\,f}
 \left\{\frac1n\sum_i\E\left[(\bm X_i^\top\mathbf\Psi\bm X_i)^2\right]\right\}^{1/2}
 \left\{\frac1n\sum_i\E\left[\left(\frac{f'}d\bm1_d^\top\bm X_i\right)^2\right]\right\}^{1/2}\\
 &\quad\le\frac{C b^2\|\mathbf\Psi\|_1|f'|}{\sqrt{dn}\,f}
 \le\frac{C_{k,\epsilon}b^2a_d^3}{\sqrt{dn}},
\end{align*}
where Lemma~\ref{lem:iid-gaussian-regularity} supplies $|f'|\le Ca_d^2$ and $f\ge ca_d$.
\end{proof}

\begin{proof}[Proof of Corollary~\ref{cor:thirdmatch}]
Set $\bm Y_i=\bm X_i/\sigma$. With subscripts indicating the data scale,
\[
 c_X=\sigma c_Y,\qquad f_X(c_X)=\sigma^{-1}f_Y(c_Y),\qquad
 \mathbf\Psi_X=\sigma^{-2}\mathbf\Psi_Y,\qquad
 \bar{\bm X}^{\,3}=\sigma^3\bar{\bm Y}^{\,3}.
\]
Consequently $R_X=R_Y$. Apply Corollary~\ref{cor:bounded-spectrum-R}
to the normalized covariance. If it is the identity, apply
Corollary~\ref{cor:isotropic-dependence-tensor} for the sharper bound.
\[
 \gamma=1\quad\Longrightarrow\quad
 (1-\gamma)Q_{n,k}=0.
\]
Insert the expectation estimates into Theorem~\ref{thm:coverage-main}.
For $d\ge n$, $L_d^2/\sqrt{dn}\le L_n^{12}/n$, proving the last assertion.
\end{proof}

\begin{proof}[Proof of Corollary~\ref{cor:firstorderpersist}]
The normalization and correction bound used in the proof of
Corollary~\ref{cor:thirdmatch} do not require $\gamma=1$. Hence
Corollary~\ref{cor:bounded-spectrum-R} gives
\[
 \sup_{\epsilon<\alpha<1-\epsilon}\E|R_{n,k}(\alpha)|
 \le C\frac{L_d^2}{\sqrt{dn}}.
\]

By Theorem~\ref{thm:coverage-main}, there are deterministic errors
$e_{n,k}(\alpha)$ such that
\[
\begin{split}
 \Prob(T_{n,[k]}\ge\hat c_{1-\alpha,k})-\alpha
  +(1-\gamma)Q_{n,k}(c^G_{1-\alpha,k})
  &=-\E R_{n,k}(\alpha)+e_{n,k}(\alpha),\\
 \sup_{\epsilon<\alpha<1-\epsilon}|e_{n,k}(\alpha)|&\le Cr_n.
\end{split}
\]
The triangle inequality proves the first asserted bound. At the
specified level $\alpha_0$, the reverse triangle inequality gives
\begin{align*}
 &\left|\Prob(T_{n,[k]}\ge\hat c_{1-\alpha_0,k})-\alpha_0\right|\\
 &\quad\ge |1-\gamma|\,
       |Q_{n,k}(c^G_{1-\alpha_0,k})|
       -\E|R_{n,k}(\alpha_0)|-|e_{n,k}(\alpha_0)|\\
 &\quad\ge |1-\gamma|c_0\zeta_{n,d}
       -C\left\{r_n+\frac{L_d^2}{\sqrt{dn}}\right\}.
\end{align*}
All constants depend
only on the constants in the stated assumptions and on
$k,\epsilon,\sigma,K$.
\end{proof}

\subsection{Deterministic conditional theorem and double bootstrap}
\begin{theorem}[Deterministic-array conditional theorem]\label{thm:deterministic-array}
Suppose $k$ is fixed, $r_n\le c$, and
$\mathbf\Sigma\succeq\sigma_*^2\mathbf I_d$.
Let deterministic vectors $\bm a_1,\ldots,\bm a_n$ satisfy
\begin{align}
 \max_{i,j}|a_{ij}|&\le C_0L_n,\qquad
 \|\bar{\bm a}\|_\infty\le C_0\sqrt{L_n/n},
       \label{eq:det-array-cond1}\\
 \max_j n^{-1}\sum_i|a_{ij}|^6&\le C_0,
       \label{eq:det-array-cond2}\\
 \|n^{-1}\sum_i\bm a_i^{\otimes2}-\mathbf\Sigma\|_\infty
       &\le C_0\sqrt{L_n/n}.\label{eq:det-array-cond3}
\end{align}
Let $v_i$ obey Assumption~\ref{ass:mult}, set
$\gamma_v=\E v_i^3$ and
$\bm S_n(\bm a)=n^{-1/2}\sum_i v_i(\bm a_i-\bar{\bm a})$.
Define
\[
 H_{\bm a}(t)=\frac12\left\langle
    n^{-1}\sum_i\bm a_i^{\otimes2}-\mathbf\Sigma,\mathsf J_{2,k}(t)\right\rangle
 -\frac{\gamma_v}{6\sqrt n}\left\langle
    n^{-1}\sum_i\bm a_i^{\otimes3},\mathsf J_{3,k}(t)\right\rangle.
\]
Then
\[
 \sup_t|\Prob_v\{S_{n,[k]}(\bm a)\le t\}-G_k(t)-H_{\bm a}(t)|
 \le CL_n^7/n.
\]
The constants depend only on $k,\sigma_*,C_0,b_v$.
No eigenvalue lower bound on the empirical covariance is required.
\end{theorem}
\begin{proof}[Proof of Theorem~\ref{thm:deterministic-array}]
The deterministic-array part of Lemma~\ref{lem:joint-marginal} applies
with coefficients $\bm a_i$. Its exact kernel is
\[
 n^{-1}\E_v\{\tau^v(v_i)(\bm a_i-\bar{\bm a})
                   (\bm a_i-\bar{\bm a})^\top\mid
                   v_i(\bm a_i-\bar{\bm a})/\sqrt n\}.
\]
The proof of that lemma retains the covariance discrepancy relative to
$\mathbf\Sigma$ and bounds the raw-to-centered replacement by $CL_n^3/n$.
This proves the stated expansion.
\end{proof}

\begin{lemma}[The first-level bootstrap array satisfies the deterministic conditions]
\label{lem:bootstrap-array}
Suppose Assumptions~\ref{ass:data} and \ref{ass:mult} hold, the multipliers
are bounded, and $r_n\le c$. On $\mathcal E_n$ from
Lemma~\ref{lem:joint-data-event}, with conditional probability at least
$1-Cn^{-3}$ the array
\[
 \bm a_i=w_i(\bm X_i-\bar{\bm X}),
 \qquad \bar{\bm X}^*=n^{-1}\sum_iw_i(\bm X_i-\bar{\bm X})
\]
satisfies \eqref{eq:det-array-cond1}--\eqref{eq:det-array-cond3}.
\end{lemma}
\begin{proof}[Proof of Lemma~\ref{lem:bootstrap-array}]
Set $\bm Y_i=w_i\bm b_i$ and use the conditional event
$\mathcal E_n^*$ in Lemma~\ref{lem:joint-data-event}. Bounded weights give
\[
 \max_{i,j}|Y_{ij}|\le CL_n,\qquad
 \max_jn^{-1}\sum_i|Y_{ij}|^6\le C,\qquad
 \|\bar{\bm Y}\|_\infty\le C\sqrt{L_n/n}.
\]
Since $\bm a_i=\bm Y_i$, the first two deterministic conditions
follow. For the covariance,
\[
 \left\|n^{-1}\sum_i\bm a_i^{\otimes2}-\mathbf\Sigma\right\|_\infty
 \le\|\bar{\bm Y}^{\,2}-\mathbf\Sigma_0\|_\infty
    +\|\mathbf\Sigma_0-\mathbf\Sigma\|_\infty
 \le C\sqrt{L_n/n}.
\]
This proves the last condition with the stated conditional probability.
\end{proof}

\begin{lemma}[Conditional inner-bootstrap coverage]
\label{lem:joint-inner-coverage}
Suppose Assumption~\ref{ass:data} holds, $k$ is fixed,
$\epsilon\in(0,1/4)$ and $r_n\le c$.
At both bootstrap levels suppose the bounded Stein-kernel alternative
in Assumption~\ref{ass:mult} holds, with
$\E w_1^3=\E v_1^3=1$.
Let $\hat c^{**}_{u,k}=(\hat F^*_{n,k})^{-1}(u)$.
On $\mathcal E_n$,
\[
 \sup_{\epsilon/2\le u\le1-\epsilon/2}
 \left|\Prob^*(T_{n,[k]}^*>\hat c^{**}_{u,k})
                   -\{1-u-R_{n,k}(1-u)\}\right|\le Cr_n.
\]
The same bound holds for $\ge$; $C$ depends only on
$k,b,\sigma_*,b_w,b_v,\epsilon$.
\end{lemma}

\begin{proof}[Proof of Lemma~\ref{lem:joint-inner-coverage}]
Condition on a data array in $\mathcal E_n$ and set $\bm Y_i=w_i\bm b_i$.
Then
\[
 \mathbf\Sigma_0=n^{-1}\sum_i\bm b_i^{\otimes2},
 \qquad \mathsf M_0=n^{-1}\sum_i\bm b_i^{\otimes3}.
\]
The conditional event $\mathcal E_n^*$ of
Lemma~\ref{lem:joint-data-event} satisfies
\[
 \Prob^*((\mathcal E_n^*)^c)\le Cn^{-3},\qquad
 \|\bar{\bm Y}\|_\infty\le C\sqrt{L/n}\quad\text{on }\mathcal E_n^*.
\]
On $\mathcal E_n^*$, Lemma~\ref{lem:bootstrap-array} verifies the
deterministic conditions for the coefficient array $\bm a_i=\bm Y_i$.
Theorem~\ref{thm:deterministic-array}, with $\gamma_v=1$, gives
\[
 \sup_t|\hat F^*_{n,k}(t)-G_k(t)-H_Y(t)|\le CL^7/n,
\]
where the raw-moment correction decomposes as
\begin{align*}
 H_Y(t)&=H_0(t)+W_Y(t)+\mathcal R_Y(t),\\
 W_Y(t)&=\tfrac12\langle\bar{\bm Y}^{\,2}-\mathbf\Sigma_0,
                                      \mathsf J_{2,k}(t)\rangle,\\
 \mathcal R_Y(t)&=-\frac1{6\sqrt n}
              \langle\bar{\bm Y}^{\,3}-\mathsf M_0,
                                      \mathsf J_{3,k}(t)\rangle.
\end{align*}
The global derivative bounds and $\mathcal E_n^*$ imply
\[
 \|\mathcal R_Y\|_\infty\le CL^2/n,\qquad
 \|H_Y\|_\infty\le CL^{3/2}/\sqrt n,\qquad
 \|H_Y'\|_\infty\le CL^2/\sqrt n.
\]
Apply Proposition~\ref{prop:quantile-remainder} and the central density
bound as in \eqref{eq:joint-linear-quantile}. Uniformly in $u$,
\begin{align*}
 \hat c^{**}_{u,k}
 &=c^G_{u,k}-\frac{H_0(c^G_{u,k})}{f_k(c^G_{u,k})}\\
 &\quad-
 \frac{\left\langle n^{-1}\sum_i
       (\bm Y_i^{\otimes2}-\E^*\bm Y_i^{\otimes2}),
                  \mathbf\Psi_{1-u}\right\rangle}{2f_k(c^G_{u,k})}
       +O(L^7/n).
\end{align*}
Apply the conditional case of Lemma~\ref{lem:joint-corrected-array}
at $c=c^G_{u,k}$. Its signed CDF satisfies
\[
 G_k\{c-H_0(c)/f_k(c)\}
       +H_0\{c-H_0(c)/f_k(c)\}=u+O(L^4/n).
\]
Writing $t_u=c-H_0(c)/f_k(c)$, the same derivative bounds give
\[
 |t_u-c|\le CL^{3/2}/\sqrt n,\qquad
 \sup_t|D_{1-u}'(t)|\le CL^{5/2}/\sqrt n,
\]
and therefore $|D_{1-u}(t_u)-D_{1-u}(c)|\le CL^4/n$.
Thus the covariance drift $\mathbf\Sigma_0-\mathbf\Sigma$ and
third-moment correction cancel, while the retained term is
\[
 R^{(b)}_{n,k}(1-u)
 =\frac{\langle\mathsf M_0\otimes\bm1_d,
                         \mathbf\Psi_{1-u}^{\otimes2}\rangle}
           {2\sqrt n f_k(c^G_{u,k})}.
\]
The corrected interval bound controls the quantile error and ties by
\[
 C\{L^{15/2}/n+L^{11}/n+n^{-3}\}\le Cr_n,
\]
including the conditional failure probability of $\mathcal E_n^*$.
Finally,
\[
 \sup_u|R^{(b)}_{n,k}(1-u)-R_{n,k}(1-u)|
 \le\frac{CL_d^2}{\sqrt n}
       \left\|n^{-1}\sum_i\bm b_i^{\otimes3}
                         -\bar{\bm X}^{\,3}\right\|_\infty
 \le CL^{5/2}/n.
\]
All expansions use the fixed positive definite reference covariance
$\mathbf\Sigma$; no inversion of $\mathbf\Sigma_0$ is required.
\end{proof}

Write
\[
 \hat H_k(u)=\Prob^*\{\hat F^*_{n,k}(T_{n,[k]}^*)\le u\},\qquad
 \hat\beta_{\alpha,k}=\hat H_k^{-1}(1-\alpha).
\]

\begin{lemma}[The prepivoted distribution]\label{lem:joint-prepivot-cdf}
Suppose Assumption~\ref{ass:data} holds, $k$ is fixed,
$\epsilon\in(0,1/4)$ and $r_n\le c$.
At both levels assume bounded Stein-kernel multipliers with
$\E w_1^3=\E v_1^3=1$, as in Lemma~\ref{lem:joint-inner-coverage}.
Define $H^{\rm piv}(u)=u+\bar R(1-u)$ on
$[\epsilon/2,1-\epsilon/2]$.
For sufficiently small $c$, this function is strictly increasing, and
on $\mathcal E_n$,
\[
 \sup_{\epsilon\le u\le1-\epsilon}
       |\hat H_k(u)-H^{\rm piv}(u)|\le Cr_n.
\]
Constants depend only on $k,b,\sigma_*,b_w,b_v$ and $\epsilon$.
\end{lemma}

\begin{proof}[Proof of Lemma~\ref{lem:joint-prepivot-cdf}]
For any CDF $F$, $u\in(0,1)$ and $\delta>0$ with $u+\delta<1$,
\[
 \{x<F^{-1}(u)\}\subseteq\{F(x)\le u\}
                  \subseteq\{x\le F^{-1}(u+\delta)\}.
\]
Apply these inclusions pointwise with $F=\hat F^*_{n,k}$,
$x=T_{n,[k]}^*$ and $\delta=r_n$.
Lemma~\ref{lem:joint-inner-coverage} applies to both endpoints because
$u\in[\epsilon,1-\epsilon]$ and $r_n<\epsilon/2$.

Lemma~\ref{lem:joint-R-control} allows $R_{n,k}$ to be replaced by
$\bar R$ with error $CL^{5/2}/n$. It also gives
\[
 |\bar R(1-u-r_n)-\bar R(1-u)|
     \le Cr_nL_d^3/\sqrt n\le Cr_n.
\]
The inclusions therefore prove the CDF bound without a conditional
continuity assumption. Finally,
\[
 (H^{\rm piv})'(u)=1-\bar R'(1-u)\ge1-CL_d^3/\sqrt n>\tfrac12
\]
when the constant in $r_n\le c$ is sufficiently small.
\end{proof}

\begin{proof}[Proof of Theorem~\ref{thm:doublewild}]
Let $C_0$ be the error constant in Lemma~\ref{lem:joint-prepivot-cdf}.
Put $p=1-\alpha$, $\delta=2C_0r_n$, and define
\[
 b_-=(H^{\rm piv})^{-1}(p-\delta),\qquad
 b_+=(H^{\rm piv})^{-1}(p+\delta).
\]
For $p\in[2\epsilon,1-2\epsilon]$ and sufficiently small $r_n$,
both levels belong to $[\epsilon,1-\epsilon]$.
The uniform conditional CDF bound gives, on $\mathcal E_n$,
\[
 b_-\le\hat\beta_{\alpha,k}\le b_+,
 \qquad
 \hat c_{b_-,k}\le\hat c_{\hat\beta_{\alpha,k},k}
                                  \le\hat c_{b_+,k}.
\]
Theorem~\ref{thm:coverage-main}, with $\gamma=1$, yields
\begin{align*}
 \Prob(T_{n,[k]}\ge\hat c_{b_\pm,k})
 &=1-b_\pm-\bar R(1-b_\pm)+O(r_n)\\
 &=\alpha\mp\delta+O(r_n).
\end{align*}
Monotonicity brackets the desired rejection probability between these
two probabilities, up to $\Prob(\mathcal E_n^c)\le Cn^{-3}$.
The strict version follows from the strict version of
Theorem~\ref{thm:coverage-main} and the same event inclusions.
\end{proof}

\subsection{An exponential example}
\begin{proposition}[Skewness and the dependence correction]
\label{prop:exponential-global}
Let $X_{ij}=E_{ij}-1$, where all $E_{ij}$ are independent standard
exponential variables, and let $d=\lfloor n^a\rfloor$ for fixed $a>0$.
Gaussian multipliers and $w=4B-1$, $B\sim\operatorname{Beta}(1/2,3/2)$,
satisfy the multiplier assumptions, with $\gamma=0$ and $\gamma=1$,
respectively. For fixed $k\ge1$, $\alpha\in(0,1)$ and sufficiently large $n$, put $c=c^G_{1-\alpha,k}$. Then
\[
 Q_{n,k}(c)=-\frac{(c^2-1)f_k(c)}{3\sqrt n},\qquad
 \E R_{n,k}(\alpha)=-\frac{cf_k'(c)}{d\sqrt n}.
\]
Let $\lambda>0$ solve $h_k(\lambda)=1-\alpha$, and put
$A_{k,\lambda}=\lambda^ke^{-\lambda}/(k-1)!$. Then
\[
 c\sim\sqrt{2\log d},\quad
 f_k(c)\sim cA_{k,\lambda},\quad
 \frac{f_k'(c)}{cf_k(c)}\longrightarrow\lambda-k.
\]
For Gaussian multipliers,
\[
 \Prob(T_{n,[k]}>\hat c_{1-\alpha,k})-\alpha
 \sim \frac{A_{k,\lambda}(2\log d)^{3/2}}{3\sqrt n}.
\]
For the matched Beta multipliers, if $0<a<1/2$ and $\lambda\ne k$,
\[
 \Prob(T_{n,[k]}>\hat c_{1-\alpha,k})-\alpha
 \sim \frac{(\lambda-k)A_{k,\lambda}(2\log d)^{3/2}}{d\sqrt n}.
\]
\end{proposition}

\begin{proof}[Proof of Proposition~\ref{prop:exponential-global}]
Write
\[
 L=L_n=\log(end),\qquad c=G_{d,k}^{-1}(1-\alpha),\qquad
 \kappa=\E X_{ij}^3=2.
\]
The observations have mean zero and covariance $\mathbf I_d$. Their
Stein kernels are
\[
 \mathbf T_i(\bm X_i)=\operatorname{diag}(X_{i1}+1,\ldots,X_{id}+1),
 \qquad \max_{i,j}\bigl(\|X_{ij}\|_{\psi_1}
             +\|T_{i,jj}-1\|_{\psi_1}\bigr)\le C.
\]
For the stated Beta weights, direct moment evaluation and the Beta Stein
identity give
\[
 \begin{gathered}
 \E w_i=0,\qquad\E w_i^2=\E w_i^3=1,
 \qquad -1\le w_i\le3,
 \\\tau^w(w)=\frac{(w+1)(3-w)}2\in[0,2].
 \end{gathered}
\]
Gaussian weights satisfy the alternative multiplier condition.
The hypotheses of Theorem~\ref{thm:coverage-main} hold uniformly, since
\[
 L^{12}/n=O\{(\log n)^{12}/n\}\longrightarrow0.
\]
Coordinate independence gives
\[
 A_{k,t}=\{\bm z:z_{[k]}\le t\},\qquad
 (\mathsf M_3)_{abc}=\kappa\ind\{a=b=c\}.
\]
Integration in coordinate $j$ gives
\begin{align*}
 \int_{A_{k,t}}\partial_j^3\phi_{\mathbf I_d}(\bm z)\,d\bm z
 &=\phi''(t)\Prob\{\operatorname{Bin}(d-1,\bar\Phi(t))=k-1\}\\
 &=\frac{(t^2-1)f_{d,k}(t)}d.
\end{align*}
The deterministic first Edgeworth correction is therefore exactly
\[
 Q_{n,k}(t)
 =-\frac{\kappa}{6\sqrt n}
        \sum_{j=1}^d\int_{A_{k,t}}\partial_j^3\phi_{\mathbf I_d}
 =-\frac{\kappa(t^2-1)f_{d,k}(t)}{6\sqrt n}.
\]
Lemma~\ref{lem:iid-dependence-tensor} gives the expectation of the
same-observation dependence correction:
\[
 \E R_{n,k}(\alpha)=-\frac{\kappa c f_{d,k}'(c)}{2d\sqrt n}.
\]
Insert these two identities into
\[
 \Prob(T_{n,[k]}>\hat c_{1-\alpha,k})-\alpha
 =-(1-\gamma)Q_{n,k}(c)-\E R_{n,k}(\alpha)
   +O(L^{12}/n),
\]
which is Theorem~\ref{thm:coverage-main} for this model.

Put $A_{k,\lambda}=\lambda^ke^{-\lambda}/(k-1)!$ and write $p_c=\bar\Phi(c)$, $q_c=\Phi(c)$, and $\lambda_d=dp_c$.
Lemma~\ref{lem:iid-gaussian-regularity} bounds $\lambda_d$ in a compact
subset of $(0,\infty)$, on which the binomial formula gives
\[
 1-\alpha=\sum_{j=0}^{k-1}\binom dj
      (\lambda_d/d)^j(1-\lambda_d/d)^{d-j}
 =e^{-\lambda_d}\sum_{j=0}^{k-1}\frac{\lambda_d^j}{j!}+O(d^{-1}).
\]
Strict monotonicity of $h_k$ and Mills' ratio now imply
\[
 \lambda_d\longrightarrow\lambda,
 \qquad c\sim\sqrt{2\log d},
 \qquad\phi(c)\sim\frac{c\lambda}{d}.
\]
Consequently,
\begin{align*}
 f_{d,k}(c)
 &=d\binom{d-1}{k-1}\phi(c)p_c^{k-1}q_c^{d-k}
   \sim c A_{k,\lambda},\\
 \frac{f_{d,k}'(c)}{c f_{d,k}(c)}
 &= -1-(k-1)\frac{\phi(c)}{c p_c}
       +(d-k)\frac{\phi(c)}{c q_c}
   \longrightarrow\lambda-k.
\end{align*}
For $\gamma=0$, the dependence term is smaller than the skewness term by
a factor of order $d^{-1}$, and
\[
 \frac{n^{-1}(\log n)^{12}}
      {n^{-1/2}(\log n)^{3/2}}
 =n^{-1/2}(\log n)^{21/2}\longrightarrow0.
\]
For $\gamma=1$, $0<a<1/2$, and $\lambda\ne k$, the first term vanishes
and the dependence term dominates the remainder because
\[
 \frac{n^{-1}(\log n)^{12}}
      {n^{-a-1/2}(\log n)^{3/2}}
 =n^{a-1/2}(\log n)^{21/2}\longrightarrow0.
\]
These comparisons prove all asserted asymptotic statements.
\end{proof}

\subsection{Regularity of the simulation models}
\label{app:simulation-regularity}

For the two gamma-copula designs in Section~\ref{sec:sim},
$\theta\in\{1/2,1\}$ and $\rho\in\{0.2,0.8\}$ are fixed.
We verify the distributional conditions uniformly in $d$.
Let $f_\theta$ denote the gamma density and write
$g_\theta=F_\theta^{-1}\circ\Phi-\theta$.

For $x>0$,
\[
 f_\theta(x)=\frac{x^{\theta-1}e^{-x}}{\Gamma(\theta)},
 \quad
 \frac{f_\theta(x)}{F_\theta(x)}\sim\frac\theta x
       \quad(x\downarrow0),
 \quad
 \frac{f_\theta(x)}{1-F_\theta(x)}\longrightarrow1
       \quad(x\to\infty).
\]
Consequently, the one-dimensional Cheeger constant satisfies
\[
 \kappa_\theta
 :=\inf_{x>0}\frac{f_\theta(x)}
                       {\min\{F_\theta(x),1-F_\theta(x)\}}>0.
\]
The density and both tail probabilities are continuous and positive
on compact subintervals of $(0,\infty)$.
The boundary inequality for finite unions of intervals, followed by
approximation, gives condition~(2.5) of \cite{Koike2026}.

Let $\bm G_R\sim N(\bm0,\mathbf R)$ be the latent Gaussian vector and set
$\bm V=(g_\theta(G_{R,1}),\ldots,g_\theta(G_{R,d}))^\top$. Then
\cite[Proposition~2.1 and its proof in Section~4.4]{Koike2026} gives
\[
 \max_j\|V_j\|_{\psi_1}\le C_\theta,
 \qquad
 \max_{j,\ell}\|\tau^V_{j\ell}(\bm V)\|_{\psi_1}\le C_\theta.
\]
Its Gaussian integration-by-parts construction gives the componentwise identity
\[
 \E\{V_jh(\bm V)\}
   =\sum_{\ell=1}^d
      \E\{\tau^V_{j\ell}(\bm V)\partial_\ell h(\bm V)\},
\]
which implies the vector-test identity in Assumption~\ref{ass:data}.
Centering the kernel entries preserves their $\psi_1$ bounds.
For an independent copy $\bm V'$, the difference has the kernel
\[
 \mathbf\tau^{\rm sym}(\bm X)
 =\E\{\mathbf\tau^V(\bm V)+\mathbf\tau^V(\bm V')
                                  \mid\bm V-\bm V'=\bm X\}.
\]
Conditional Jensen's inequality preserves the same bounds, up to a fixed constant.

To verify the population eigenvalue condition, let $G\sim N(0,1)$ and
let $H_m$ be the probabilists' Hermite polynomials. Expand
\[
 g_\theta(z)=\sum_{m\ge1}a_{\theta,m}
                       \frac{H_m(z)}{\sqrt{m!}},
 \qquad
 \sum_{m\ge1}a_{\theta,m}^2=\theta.
\]
Strict monotonicity gives $a_{\theta,1}=\E\{Gg_\theta(G)\}>0$.
Writing $\mathbf C_\theta=\operatorname{Cov}(\bm V)$, the Hermite
identity and the Schur product theorem yield
\[
 \mathbf C_\theta
   =\sum_{m\ge1}a_{\theta,m}^2\mathbf R^{\circ m}
   \succeq a_{\theta,1}^2\mathbf R.
\]
The actual covariance is $\mathbf\Sigma=\mathbf C_1$ in the asymmetric
case and $\mathbf\Sigma=2\mathbf C_{1/2}$ in the symmetric case.
Since
\[
 \lambda_{\min}(\mathbf R)\ge1-\rho
       \quad\text{in Design I},
 \qquad
 \lambda_{\min}(\mathbf R)\ge\frac{1-\rho}{1+\rho}
       \quad\text{in Design II},
\]
both designs satisfy Assumption~\ref{ass:data}(i)--(ii), with constants
independent of $d$, and both have $\Sigma_{jj}=1$.
The growth restrictions in Assumption~\ref{ass:data}(iii) and
Theorems~\ref{thm:coverage-main}--\ref{thm:doublewild}
are additional to these distributional conditions.

Put $s_1=1$, $s_{1/2}=2$ and
$c_\theta(r)=\sum_{m\ge1}a_{\theta,m}^2r^m$.
In Design II,
\[
 0\le\Sigma_{j\ell}
 =s_\theta c_\theta(\rho^{|j-\ell|})
 \le \rho^{|j-\ell|},
 \qquad
 \lambda_{\max}(\mathbf\Sigma)
 \le\frac{1+\rho}{1-\rho}.
\]
In Design I,
\[
 \lambda_{\max}(\mathbf\Sigma)
   =1+(d-1)s_\theta c_\theta(\rho),
 \qquad c_\theta(\rho)>0.
\]
Thus the bounded-largest-eigenvalue condition of
Corollary~\ref{cor:thirdmatch} holds in Design II and fails in Design I.
For both designs, $\rho_d=s_\theta c_\theta(\rho)>0$ is fixed, so
Assumption~\ref{ass:weakcorr} does not hold as $d\to\infty$;
it is not required by Theorem~\ref{thm:coverage-main}.

For the beta weights, write $\nu=\alpha_\nu+\beta_\nu$,
$\mu=\alpha_\nu/\nu$ and
$s^2=\alpha_\nu\beta_\nu/\{\nu^2(\nu+1)\}$.
The standardized weight $w=(\eta-\mu)/s$ has the scalar Stein kernel
\[
 \tau^w(w)=\frac{\eta(1-\eta)}{\nu s^2},
 \qquad
 |w|\le\frac{\max\{\mu,1-\mu\}}s,
 \qquad
 0\le\tau^w(w)\le\frac1{4\nu s^2}.
\]
These are fixed finite bounds because $\nu=0.1$ and
$\alpha_\nu,\beta_\nu>0$. The chosen parameters give
\[
 \E w=0,
 \qquad \E w^2=1,
 \qquad
 \E w^3=
 \frac{2(\beta_\nu-\alpha_\nu)\sqrt{\nu+1}}
      {(\nu+2)\sqrt{\alpha_\nu\beta_\nu}}=1.
\]
Using this law independently at both levels satisfies the multiplier
conditions of Theorem~\ref{thm:doublewild}.
Gaussian weights satisfy the other alternative of Assumption~\ref{ass:mult}.
The empirical, Mammen and Rademacher procedures are comparison methods;
the present Stein-kernel theorems do not assert their second-order accuracy.

\section{Proofs for the stationary exponential-mixing alternative}
\label{supp-sec:mixing-alt}\label{app:mixing-alt}
Throughout this appendix, use Assumptions~\ref{ass:data},
\ref{ass:mult}, \ref{ass:variance} and \ref{ass:mixing-alt}, and the block
parameters defined in Section~\ref{subsec:mixing-alt}.
The Gaussian Poisson error is estimated explicitly below. The
distributional and joint bootstrap expansions use the bounds in Appendix~A.

\subsection*{B.1. Correlation decay and Gaussian cluster tails}

\begin{lemma}
\label{lem:mix-rem-corr}
Under Assumptions~\ref{ass:data} and \ref{ass:mixing-alt}, for every $h\ge1$,
\begin{equation}
|\rho(h)|
\le
\sin\{2\pi\alpha_Z(h)\}
\le
2\pi C_\alpha e^{-a_\alpha h}.
\label{eq:mix-rem-corr-decay}
\end{equation}
Consequently,
\begin{equation}
\sum_{h=1}^{\infty}|\rho(h)|
\le
\frac{2\pi C_\alpha e^{-a_\alpha}}{1-e^{-a_\alpha}}
<\infty.
\label{eq:mix-rem-corr-sum}
\end{equation}
Moreover, for every integer $m\ge2$, every index set $I\subset[d]$ with $|I|=m$, and every $t>0$,
\begin{equation}\begin{aligned}
\Prob(Z_j>t,\ \forall j\in I)
&\le
\bar\Phi\!\left(
\sqrt{\frac{m}{1+(m-1)\vartheta_*}}\,
\frac{t}{\sigma}
\right)\\
&\le
\frac{\sigma\sqrt{1+(m-1)\vartheta_*}}{\sqrt{2\pi m}\,t}
\exp\!\left\{
-\frac{m t^2}{2\sigma^2\{1+(m-1)\vartheta_*\}}
\right\}.
\end{aligned}\label{eq:mix-rem-cluster-tail}
\end{equation}
In particular, when $m=2$ and $1\le h<d$,
\begin{equation}
\Prob(Z_0>t,Z_h>t)
\le
\bar\Phi\!\left(\sqrt{\frac{2}{1+\vartheta_*}}\,\frac{t}{\sigma}\right)
\le
\frac{\sigma\sqrt{1+\vartheta_*}}{\sqrt{4\pi}\,t}
\exp\!\left\{-\frac{t^2}{\sigma^2(1+\vartheta_*)}\right\}.
\label{eq:mix-rem-pair-tail}
\end{equation}
\end{lemma}

\begin{proof}[Proof of Lemma~\ref{lem:mix-rem-corr}]
For standard Gaussian variables $U,V$ with correlation $r$, one has
\[
\Prob(U>0,V>0)-\frac14=\frac{1}{2\pi}\arcsin(r).
\]
Therefore, with
\[
\mathcal F_0:=\sigma(Z_j:j\le0),
\qquad
\mathcal G_h:=\sigma(Z_j:j\ge h),
\]
we obtain
\[
\alpha_Z(h)
\ge
\left|\Prob(Z_0>0,Z_h>0)-\Prob(Z_0>0)\Prob(Z_h>0)\right|
=
\frac{1}{2\pi}|\arcsin\rho(h)|.
\]
Since $0\le\alpha_Z(h)\le1/4$, monotonicity of sine gives
\[
|\rho(h)|\le \sin\{2\pi\alpha_Z(h)\}\le 2\pi\alpha_Z(h)\le 2\pi C_\alpha e^{-a_\alpha h},
\]
which proves \eqref{eq:mix-rem-corr-decay}. Summing the geometric series yields \eqref{eq:mix-rem-corr-sum}.

Now fix $I=\{i_1,\dots,i_m\}\subset[d]$ with $|I|=m$ and write
\[
U_r:=Z_{i_r}/\sigma,
\qquad
1\le r\le m.
\]
By \eqref{eq:mix-rem-vartheta}, every off-diagonal correlation of $(U_1,\dots,U_m)^\top$ is bounded above by $\vartheta_*$. Therefore
\[
\Var\!\left(\sum_{r=1}^m U_r\right)
\le
m+m(m-1)\vartheta_*
=
m\{1+(m-1)\vartheta_*\}.
\]
Since
\[
\{U_r>u,\ \forall r\le m\}
\subset
\left\{\sum_{r=1}^m U_r>mu\right\},
\]
we obtain
\[
\Prob(U_r>u,\ \forall r\le m)
\le
\bar\Phi\!\left(
\sqrt{\frac{m}{1+(m-1)\vartheta_*}}\,
u
\right).
\]
Applying Mills' ratio proves \eqref{eq:mix-rem-cluster-tail}, and \eqref{eq:mix-rem-pair-tail} is the case $m=2$.
\end{proof}

\subsection*{B.2. Bonferroni remainder for the $k$th exceedance event}

For integers $m\ge k\ge1$ and a nonnegative integer-valued random variable
$N$ with $\E\binom N{m+1}<\infty$, Lemma~\ref{lem:comb} gives
\[
 \left|\Prob(N\ge k)-\sum_{s=k}^{m}(-1)^{s-k}
       \binom{s-1}{k-1}\E\binom Ns\right|
 \le\binom m{k-1}\E\binom N{m+1}.
\]
Its signed pointwise version determines the parity of the truncation.

\subsection*{B.3. Block construction and reduction to block exceedances}

Assume $q_d\ge1$. Let
\[
s_d:=d-q_d(m_d+\ell_d),
\qquad
0\le s_d<m_d+\ell_d.
\]
Define the main blocks and gaps by
\[
I_r:=\{(r-1)(m_d+\ell_d)+1,\dots,(r-1)(m_d+\ell_d)+m_d\},
\qquad r=1,\dots,q_d,
\]
\[
J_r:=\{(r-1)(m_d+\ell_d)+m_d+1,\dots,r(m_d+\ell_d)\},
\qquad r=1,\dots,q_d,
\]
and define the remainder interval
\[
R_d:=\{q_d(m_d+\ell_d)+1,\dots,d\}
\]
when $s_d\ge1$, and let $R_d=\varnothing$ when $s_d=0$.
Use the convention $\max_{j\in\varnothing}Z_j=-\infty$.
For $t\in\R$, set
\begin{align*}
B_r(t)&:=\left\{\max_{j\in I_r} Z_j>t\right\},
&Y_r(t)&:=\ind\{B_r(t)\},\\
S_d(t)&:=\sum_{r=1}^{q_d}Y_r(t),
&N_d(t)&:=\sum_{j=1}^d \ind\{Z_j>t\}.
\end{align*}
Also define
\[
q(t):=\Prob(B_1(t)),
\qquad
\mu_d(t):=q_d q(t).
\]

\begin{lemma}
\label{lem:mix-rem-block-bad}
Under Assumptions~\ref{ass:data}, \ref{ass:variance} and
\ref{ass:mixing-alt}, let $q_d\ge1$. For every $t\in\R$,
\begin{equation}
0\le m_d p(t)-q(t)
\le
\binom{m_d}{2}
\bar\Phi\!\left(\sqrt{\frac{2}{1+\vartheta_*}}\,\frac{t}{\sigma}\right).
\label{eq:mix-rem-q-bound}
\end{equation}
Moreover,
\begin{align}
\Prob\{N_d(t)\neq S_d(t)\}
&\le
(q_d\ell_d+s_d)p(t)
+
q_d\binom{m_d}{2}
\bar\Phi\!\left(\sqrt{\frac{2}{1+\vartheta_*}}\,\frac{t}{\sigma}\right),
\label{eq:mix-rem-bad-event}
\\
|\mu_d(t)-\lambda(t)|
&\le
(q_d\ell_d+m_d+\ell_d)p(t)
+
q_d\binom{m_d}{2}
\bar\Phi\!\left(\sqrt{\frac{2}{1+\vartheta_*}}\,\frac{t}{\sigma}\right).
\label{eq:mix-rem-mu-lambda}
\end{align}
\end{lemma}

\begin{proof}[Proof of Lemma~\ref{lem:mix-rem-block-bad}]
The first Bonferroni inequality gives
\[
q(t)=\Prob\Bigl(\bigcup_{j\in I_1}\{Z_j>t\}\Bigr)\le m_d p(t),
\]
and the second Bonferroni inequality yields
\[
q(t)
\ge
m_d p(t)-\sum_{1\le a<b\le m_d}\Prob(Z_a>t,Z_b>t).
\]
For $t>0$, use \eqref{eq:mix-rem-pair-tail}; for $t\le0$, the pair probability is at most $p(t)$, which is no larger than the displayed Gaussian tail. This proves \eqref{eq:mix-rem-q-bound}.

If $N_d(t)\neq S_d(t)$, then either at least one exceedance occurs in a gap or in $R_d$, or some main block contains at least two exceedances. Therefore
\[
\begin{aligned}
\Prob\{N_d(t)\neq S_d(t)\}
&\le
\sum_{r=1}^{q_d}\Prob\left\{\max_{j\in J_r}Z_j>t\right\}
+
\Prob\left\{\max_{j\in R_d}Z_j>t\right\}\\
&\quad+
\sum_{r=1}^{q_d}\Prob\left\{\sum_{j\in I_r}\ind\{Z_j>t\}\ge2\right\}.
\end{aligned}
\]
Now
\[
\Prob\left\{\max_{j\in J_r}Z_j>t\right\}\le \ell_d p(t),
\qquad
\Prob\left\{\max_{j\in R_d}Z_j>t\right\}\le s_d p(t),
\]
and, by the union bound and \eqref{eq:mix-rem-pair-tail},
\[
\begin{aligned}
\Prob\left\{\sum_{j\in I_r}\ind\{Z_j>t\}\ge2\right\}
&\le
\sum_{1\le a<b\le m_d}\Prob(Z_a>t,Z_b>t)\\
&\le
\binom{m_d}{2}
\bar\Phi\!\left(\sqrt{\frac{2}{1+\vartheta_*}}\,\frac{t}{\sigma}\right).
\end{aligned}
\]
This proves \eqref{eq:mix-rem-bad-event}. Finally,
\[
|\mu_d(t)-\lambda(t)|
\le
|q_d q(t)-q_d m_d p(t)|
+
|q_d m_d-d|\,p(t),
\]
and
\[
|q_d m_d-d|
\le
q_d\ell_d+m_d+\ell_d.
\]
Combining these displays with \eqref{eq:mix-rem-q-bound} proves \eqref{eq:mix-rem-mu-lambda}.
\end{proof}

\subsection*{B.4. Direct Poisson approximation on the quantile window}

\begin{proposition}[A block bound for the Gaussian reference distribution]
\label{prop:gauss-block-finite}
Let $(Z_j)_{j\in\mathbb Z}$ be stationary centered Gaussian with
strong-mixing coefficients $\alpha_Z(\ell)$. Suppose
\[
 \Var(Z_0)=\sigma^2>0,\qquad
 \sup_{1\le h<d}\Corr(Z_0,Z_h)\le\vartheta,\qquad 0\le\vartheta<1.
\]
Fix $t>0$, integers $m,\ell\ge1$, and
$q=\lfloor d/(m+\ell)\rfloor\ge1$. Set
\[
p=\bar\Phi(t/\sigma),\quad\lambda=dp\le\Lambda,\quad
v(t)=\bar\Phi\left(\sqrt{\frac{2}{1+\vartheta}}\frac t\sigma\right),
\]
\[
e_d(t):=(q\ell+m+\ell)p+q\binom m2v(t).
\]
For every fixed $1\le k\le d$ and integer $K\ge k$, with $L=K+1$,
\[
|G_k(t)-h_k(\lambda)|
\le 2e_d(t)+\frac{C_{k,\Lambda}}q
 +D_{q,K,k}\alpha_Z(\ell)
 +2\binom K{k-1}\frac{\Lambda^{K+1}}{(K+1)!},
\]
where the finite, explicit constant is
\[
D_{q,K,k}:=
\sum_{s=k}^{K}\binom{s-1}{k-1}(s-1)\binom qs
 +\binom K{k-1}K\binom q{K+1}.
\]
In particular, $D_{q,K,k}\le C_k(K+1)^{k+1}d^{K+1}$. This assertion concerns the Gaussian reference distribution only.
\end{proposition}

\begin{proof}[Proof of Proposition~\ref{prop:gauss-block-finite}]
Partition the coordinates into $q$ main blocks of length $m$, gaps of length $\ell$, and a final remainder of fewer than $m+\ell$ indices. Let $Y_r$ indicate an exceedance in main block $r$, and put $S=\sum_{r=1}^qY_r$, $u=\Prob(Y_1=1)$, and $\mu=qu$. For distinct coordinates,
\[
\Prob(Z_a>t,Z_b>t)
\le\Prob(Z_a+Z_b>2t)
\le v(t).
\]
Union bounds for omitted coordinates and pairs in a main block give
\[
\Prob\{N_Z(t)\ne S\}\le e_d(t),
\qquad 0\le\lambda-\mu\le e_d(t),\qquad \mu\le\Lambda.
\]
For selected blocks $r_1<\cdots<r_s$, successive use of the definition of strong mixing yields
\[
\left|\Prob(Y_{r_1}=\cdots=Y_{r_s}=1)-u^s\right|
\le(s-1)\alpha_Z(\ell).
\]
Consequently, for $1\le s\le K+1$,
\[
\left|\E\binom Ss-\binom qsu^s\right|
\le(s-1)\binom qs\alpha_Z(\ell),
\]
and a collision bound for $s$ independent uniform draws from $[q]$ gives
\[
0\le\frac{\mu^s}{s!}-\binom qsu^s
\le\frac1q\binom s2\frac{\mu^s}{s!}.
\]
Apply Lemma~\ref{lem:finite-ie-corrected} to $S$ and to a Poisson variable of mean $\mu$. Summing the preceding errors and bounding the $(K+1)$st factorial moment gives
\[
|\Prob(S\le k-1)-h_k(\mu)|
\le\frac{C_{k,\Lambda}}q+D_{q,K,k}\alpha_Z(\ell)
 +2\binom K{k-1}\frac{\Lambda^{K+1}}{(K+1)!}.
\]
The constant $C_{k,\Lambda}$ is finite by convergence of the factorial series. Finally,
\[
|h_k'(x)|=e^{-x}\frac{x^{k-1}}{(k-1)!}\le1,
\quad |h_k(\mu)-h_k(\lambda)|\le e_d(t).
\]
Combining these bounds proves the result. Bounding every binomial coefficient in $D_{q,K,k}$ by $d^s$ proves its displayed coarse bound.
\end{proof}

\begin{lemma}[Poisson approximation and the quantile window]
\label{lem:mix-rem-window-bounds}\label{lem:mix-rem-poisson}
Under Assumptions~\ref{ass:data}, \ref{ass:variance} and
\ref{ass:mixing-alt}, assume $q_d\ge1$ and $k_0\ge k$.
Let $r_d$ be defined in the main text.
For large $d$,
\[
 \sup_{\lambda(t)\le2\Lambda_{k,\epsilon}}
 |G_k(t)-h_k\{\lambda(t)\}|\le C_{k,\epsilon}r_d.
\]
If $r_d\le c_{k,\epsilon}$ for a sufficiently small constant, then
\[
 \frac{k\epsilon}{2}\le\lambda(t)\le\Lambda_{k,\epsilon},
 \qquad t\in\Tcal_{k,\epsilon}.
\]
Consequently,
\begin{equation}
 \sup_{t\in\Tcal_{k,\epsilon}}
 |G_k(t)-h_k\{\lambda(t)\}|\le C_{k,\epsilon}r_d.
 \label{eq:mix-rem-Gk-poisson}
\end{equation}
\end{lemma}
\begin{proof}[Proof of Lemma~\ref{lem:mix-rem-poisson}]
Put $\Lambda=\Lambda_{k,\epsilon}$ and
$t_{2\Lambda}=\sigma\Phi^{-1}(1-2\Lambda/d)$.
Mills' inequalities give, with $u=t_{2\Lambda}/\sigma$,
\[
 u^2=2\log d-\log\log d+O(1),\qquad
 \bar\Phi(\sqrt{1+\beta_*}\,u)
 \le C d^{-(1+\beta_*)}(\log d)^{\beta_*/2}.
\]
For $\lambda(t)\le2\Lambda$, $t\ge t_{2\Lambda}>0$. The omitted-coordinate
and within-block pair terms therefore satisfy
\[
 e_d(t)\le C\left\{\frac{\ell_d}{m_d}
        +\frac{m_d+\ell_d}{d}
        +d^{-3\beta_*/4}(\log d)^{\beta_*/2}\right\}
 \le C\eta_{1,d}.
\]
Apply Proposition~\ref{prop:gauss-block-finite} with $K=k_0$.
For fixed $k$,
\[
 \binom{k_0}{k-1}(2\Lambda)^{k_0+1}
 \le C_k(3\Lambda)^{k_0+1},
\]
so its full remainder is bounded by $Cr_d$.

Now put $t_\Lambda=\sigma\Phi^{-1}(1-\Lambda/d)$.
Since $h_k(\Lambda)=\epsilon/8$,
\[
 G_k(t_\Lambda)\le\epsilon/8+Cr_d<\epsilon/4.
\]
Every point in $\Tcal_{k,\epsilon}$ is therefore above $t_\Lambda$,
which proves $\lambda(t)\le\Lambda$. Markov's inequality gives
\[
 \epsilon/2\le1-G_k(t)
   =\Prob\{N_Z(t)\ge k\}\le\lambda(t)/k.
\]
This proves the lower bound and the uniform assertion on the quantile window.
\end{proof}

\subsection*{B.5. Threshold scale, shift/strip bounds, and weighted Gaussian bounds}

\begin{lemma}[Threshold scale]\label{lem:mix-rem-t-order}
Assume Assumptions~\ref{ass:data}, \ref{ass:variance}, and
\ref{ass:mixing-alt}, with fixed $k$, $\epsilon\in(0,1/4)$,
$q_d\ge1$, $k_0\ge k$, and $r_d\le c_{k,\epsilon}$ for sufficiently small
$c_{k,\epsilon}>0$. For large $d$,
\[
 t^2/\sigma^2=2\log d-\log\log d+O(1),
 \qquad t\in\Tcal_{k,\epsilon},
\]
uniformly over this window. In particular, $t^2\asymp\log d$.
\end{lemma}
\begin{proof}[Proof of Lemma~\ref{lem:mix-rem-t-order}]
Lemma~\ref{lem:mix-rem-poisson} gives
$k\epsilon/2\le d\bar\Phi(t/\sigma)\le\Lambda_{k,\epsilon}$ on
$\Tcal_{k,\epsilon}$. Taking logarithms in Mills' inequalities proves the assertion.
\end{proof}

The shift and strip bounds follow from Lemma~\ref{lem:global-order-strips}
with $\bm h=a\bm1_d$; Assumption~\ref{ass:data} supplies its covariance
lower bound.

For $q_d\ge1$ and $\mu_d(t)\le\Lambda$, the proof of
Proposition~\ref{prop:gauss-block-finite} gives, for every integer $s\ge1$,
\[
 \E\binom{S_d(t)}s
 \le\frac{\Lambda^s}{s!}
       +(s-1)\binom{q_d}s\alpha_Z(\ell_d).
\]
For any integer $K\ge k$, the weighted sum is bounded by multiplying
this inequality by $\binom{s-1}{k-1}$ and summing, since
\[
 \sum_{s=k}^{\infty}\binom{s-1}{k-1}\frac{\Lambda^s}{s!}<\infty.
\]
Taking $s=K+1$ and multiplying by $\binom K{k-1}$ gives the
corresponding truncation bound.

\subsection*{B.6. Regularity of $G_k$}

Under Assumption~\ref{ass:data}, Lemma~\ref{lem:rank-density-lower} gives
\[
 \begin{gathered}
 \inf_{t\in\Tcal_{k,\epsilon}}f_k(t)\ge c_{k,\epsilon}>0,
 \qquad \sup_t|f_k'(t)|\le CL_d,
 \\
 \sup_{\epsilon/2\le p\le1-\epsilon/2}|(G_k^{-1})''(p)|\le CL_d.
 \end{gathered}
\]

\subsection*{B.7. Completion of the proof of Theorem~\ref{prop:mixing-alt-main-rd}}

\begin{proof}[Proof of Theorem~\ref{prop:mixing-alt-main-rd}]
Theorem~\ref{thm:distribution} gives the two CDF expansions with
remainder $C\epsn^2$. Theorem~\ref{thm:CF} gives the quantile expansion,
with the sign $-C_{n,k}(\alpha)$, since $L_n^7/n\le r_n$.
The joint coverage expansion follows from Theorem~\ref{thm:coverage-main}.

Stationarity gives common marginal variance $\sigma^2$.
Lemma~\ref{lem:mix-rem-corr} and Gershgorin's inequality give
\[
 \lambda_{\max}(\mathbf\Sigma/\sigma^2)
 \le1+2\sum_{h\ge1}|\rho(h)|
 \le1+\frac{4\pi C_\alpha}{e^{a_\alpha}-1}<\infty.
\]
Thus Corollary~\ref{cor:thirdmatch} applies when $\gamma=1$ and $d\ge n$.
For bounded third-moment-matched multipliers at both levels, apply
Theorem~\ref{thm:doublewild}.
The separate Gaussian Poisson conclusion is \eqref{eq:mix-rem-Gk-poisson}.
\end{proof}

For $d=n^c$ with fixed $c>0$, Remark~\ref{rem:mix-rate} gives
$\ell_d=O(\log^2n)$, while $m_d\asymp d^{\beta_*/4}$. Therefore
\[
 \eta_{1,d}+q_d^{-1}
 \le C\{n^{-c\beta_*/4}\log^2n+n^{-c(1-\beta_*/4)}
            +n^{-3c\beta_*/4}(\log n)^{\beta_*/2}\}.
\]
The mixing and factorial terms satisfy the bounds in
Remark~\ref{rem:mix-rate}. Hence $r_d=O(n^{-1})$ whenever $c>4/\beta_*$.

\end{document}